\documentclass[reqno]{amsart}

\usepackage{amsmath,amssymb,color, amsthm}
\usepackage{amsfonts, amscd, epsfig,enumerate}
\usepackage{mathrsfs}

\newtheorem{theorem}{Theorem}[section]
\newtheorem{lemma}[theorem]{Lemma}
\newtheorem{proposition}[theorem]{Proposition}

\newtheorem{remark}[theorem]{Remark}
\newtheorem{definition}{Definition}[section]

\usepackage{tikz}
\usetikzlibrary{backgrounds}
\usetikzlibrary{patterns,fadings}
\usetikzlibrary{arrows,decorations.pathmorphing}
\usetikzlibrary{decorations}
\usetikzlibrary{calc}
\usetikzlibrary{shapes.misc}

\usepackage{setspace}

\usepackage{float}
\usepackage[colorlinks=true,linkcolor=blue,citecolor=magenta]{hyperref}

\newcommand{\vertiii}[1]{{\left\vert\kern-0.25ex\left\vert\kern-0.25ex\left\vert #1 
    \right\vert\kern-0.25ex\right\vert\kern-0.25ex\right\vert}}

\numberwithin{equation}{section}
\numberwithin{figure}{section}

\newcommand{\mc}[1]{{\mathcal #1}}

\newcommand{\bb}[1]{{\mathbb #1}}

\renewcommand{\epsilon}{\varepsilon}

\newcommand{\R}{\mathbb R}

\newcommand{\Z}{\mathbb Z}
\newcommand{\N}{\mathbb N}
\renewcommand{\P}{\mathbb P}
\newcommand{\T}{\mathbb T}
\newcommand{\E}{\mathbb E}
\newcommand{\Q}{\mathbb Q}

\renewcommand{\bar}{\overline}
\renewcommand{\tilde}{\widetilde}
\renewcommand{\hat}{\widehat}

\newcommand{\abs}[1]{\;\left\vert\;#1 \;\right\vert\;}

\newcommand{\normm}[1]{{\left\vert\kern-0.1ex\left\vert\kern-0.1ex\left\vert\; #1 \; \right\vert\kern-0.1ex\right\vert\kern-0.1ex\right\vert}}    
\newcommand{\cro}[1]{\left[#1\right]}
\newcommand{\pa}[1]{\left(#1\right)}

\newcommand{\gene}{\mathcal{L}}

\newcommand{\ang}[1]{\langle\!\langle #1 \rangle \!\rangle}

\renewcommand{\leq}{\leqslant}
\renewcommand{\geq}{\geqslant}

\newcommand{\uleft}{u_{-}}
\newcommand{\uright}{u_{+}}

\allowdisplaybreaks

\title{Stefan problem for a non-ergodic facilitated exclusion process}
\author{Oriane Blondel}
\address{Univ Lyon, CNRS, Universit\'e Claude Bernard Lyon 1, UMR5208, Institut Camille Jordan, F-69622 Villeurbanne, France}
\email{blondel@math.univ-lyon1.fr}
\author{Cl\'ement Erignoux}
\address{Inria, Univ. Lille, CNRS, UMR 8524 - Laboratoire Paul Painlev\'e, F-59000 Lille}
\email{clement.erignoux@inria.fr}
\author{Marielle Simon}
\address{Inria, Univ. Lille, CNRS, UMR 8524 - Laboratoire Paul Painlev\'e, F-59000 Lille}
\email{marielle.simon@inria.fr}

\thanks{C.E.\ would like to thank H. Lacoin and C. Landim for interesting discussions. M.S.\ thanks C. Canc\`es and M. Herda for very enlightening exchanges about the solutions to free boundary problems. C.E.\ gratefully acknowledges funding from the European Research Council (ERC) under the European Unions Horizon 2020 Program, ERC Consolidator Grant UniCoSM (grant agreement n\textsuperscript{o}724939). M.S.\ thanks Labex CEMPI (ANR-11-LABX-0007-01).  O.B.\ thanks the ANR projects LSD (ANR-15-CE40-0020) and MALIN (ANR-16-CE93-0003) and acknowledges support from INSMI (CNRS) through a PEPS JCJC grant. This project is partially supported by  the ANR grant MICMOV (ANR-19-CE40-0012) of the French National Research Agency (ANR), and it has also received funding from the European Research Council (ERC) under  the European Union's Horizon 2020 research and innovative program (grant agreement  n\textsuperscript{o}715734).}

\begin{document}

\begin{abstract}
We consider the \emph{facilitated exclusion process}, which is a non-ergodic, kinetically constrained exclusion process. We show that in the hydrodynamic limit, its macroscopic behavior is governed by a free boundary problem. The particles evolve on the one-dimensional lattice according to jump rates which are degenerate, since  they can vanish on non-trivial configurations and create distinct phases: indeed, configurations can be totally \emph{blocked} (they cannot evolve under the dynamics), \emph{ergodic} (they belong to an irreducible component), or \emph{transient} (after a transitive period of time they will become either blocked or ergodic). We additionally prove that the microscopic separation into blocked/ergodic phases fully coincides with the moving interface problem given by the hydrodynamic equation.
\end{abstract}

\maketitle

\section{Introduction}

In statistical physics, various types of (nonlinear)
partial differential equations have been derived from  underlying microscopic particle systems which belong to the class of stochastic lattice gases. This mathematical procedure is called  \emph{hydrodynamic limit}:  the macroscopic behavior is obtained \emph{via} a long-time and large-space scaling limit, see for instance \cite{KL} for a review on the subject. In particular, exclusion processes have attracted a lot of interest due to the variety and complexity of the results which have been obtained in the last decades, despite the simplicity of their description.  For these models, the hydrodynamic equations obtained
in the limit describe the evolution of the local density, which is conserved by the dynamics. These equations become even more interesting when they involve a phase change in the physical medium: in that case, the process of diffusion is
mathematically formulated as a \emph{Stefan problem} \cite{Stef}, or \emph{free boundary problem}.

Such macroscopic behavior can be naturally expected from \emph{kinetically constrained lattice gases,} or KCLGs, in which the configuration of particles must satisfy a local constraint in order for a particle to be able to jump. For such  models, one may predict distinct  behaviors of the system at density $\rho$, depending on whether the local constraint should typically be satisfied at density $\rho$. This, however, strongly depends on the specific mixing mechanisms of the models. According to a standard terminology, there are two types of kinetically constrained lattice gases (see \emph{e.g.}~\cite{CMRT}). In \emph{non-cooperative} KCLGs, a mobile cluster of particles of a given shape can move autonomously in the system (always respecting the kinetic constraint), and once it reaches a specific neighborhood of a particle, allow the latter to jump. The existence of such mobile clusters gives the system good mixing properties, so that their macroscopic behavior is described by diffusive equations with no phase separation. The model considered here, instead, is \emph{cooperative}, in the sense that no such mobile cluster exists. This generates intrinsic difficulties, and in our case distorts the equilibrium measures which are no longer product measures\footnote{Historically, KCLG were introduced in the physics litterature as reversible dynamics w.r.t.\ a product measure \cite{KA,ritortsollich}, to study the effect on relaxation of dynamical constraints as opposed to (equilibrium) thermodynamic interactions.}. 

In the first version of this paper, we asked whether it was possible to build a KCLG which would be at the same time cooperative, \emph{gradient} (in the sense that the generator is a discrete Laplacian), and \emph{reversible} with respect to product measures. The motivation to build such a model came from the fact that most of the interesting macroscopic phenomenology of KCLG's comes from their cooperative nature, and also from the effort involved in studying non-gradient or non-reversible models (w.r.t.\ product measures). We already knew that
\begin{itemize}
 \item [--] the  \emph{Kob-Andersen model} \cite{CMRT} is cooperative and reversible,
 \item [--] the KCLG whose macroscopic behavior is given by the \emph{porous medium equation} considered in \cite{GLT, BCSS} is gradient and reversible,
 \item [--] the \emph{facilitated exclusion process}  \cite{RPV} is cooperative and gradient.
\end{itemize}
The impossibility of combining all three characteristics was proved during the revision process by Shapira in the appendix of \cite{assaf}.


\subsection{The facilitated exclusion process} In this paper we consider the last model which has been mentioned above, namely the facilitated exclusion process, introduced in \cite{RPV} and further investigated in \cite{BBCS,BM2009,BESS,O,Lubeck}. Its dynamics can be described as follows: on the periodic domain $\T_N$, we associate independently with each site a random Poissonian clock  ringing at rate $2$. 
When the clock at site $x$ rings, if the site $x$ is occupied, the particle chooses one of its neighbors $x\pm1$ to jump to, 
each one with probability $\frac12$. However, the jump does not systematically occur, but follows two rules, (i) the \textit{exclusion principle}: if the target site $x\pm1$ is already occupied, then the jump is canceled, and (ii) a \textit{dynamical constraint}: if the other neighbor $x\mp 1$
is empty, then the jump is canceled. 
In other words, a particle, in order to jump, needs to be ``pushed'' to an empty site by a neighboring particle.

Contrarily to the vast majority of exclusion processes considered in the literature, the grand canonical measures of this process are not products of Bernoulli measures: on the one hand, the strong dynamical constraint creates a phase transition at the critical density $\frac12$. Precisely, if the equilibrium density $\rho$ satisfies $\rho > \frac12$, then there is a unique invariant measure $\pi_\rho$, while if $\rho \leqslant \frac12$, all the Dirac measures concentrated on  configurations which cannot evolve under the dynamics are invariant.  On the other hand, $\pi_\rho$ is not a product measure, but presents non-trivial correlations (which however decay exponentially fast, as proved in \cite[Section 6.3]{BESS}). Another technical issue is that the facilitated exclusion process itself is not attractive, though it can be mapped to an attractive zero-range process (see \eqref{eq:DefPi}).

At the macroscopic level, one naturally expects the same separation of phases. As conjectured in \cite{BESS}, the macroscopic behavior of this system is
described by the free boundary problem in which the same nonlinear diffusion
equation as in \cite{BESS} governs the evolution of the density in the supercritical -- active -- phase $(\frac12,1]$, while there is no evolution in the subcritical -- frozen -- phase $[0, \frac12]$.  As the frozen region is progressively filled from the growth of the active region, the latter grows and the frontier (or free boundary) between the two regions moves.
 More precisely, we show that, in the diffusive space/time scaling, the empirical density of particles is governed in the macroscopic limit by the (weak) solution to the following  Stefan problem\footnote{Uniqueness of the weak solution to \eqref{eq:stef} in the sense of Definition~\ref{Def:PDEweak} follows from the monotonicity of $\mathcal{H}$ \cite[Theorem 6, p.10]{Uchiyama}.}
\begin{equation}
\label{eq:stef} 
\partial_t\rho = \partial_u^2 \big(\mathcal{H}(\rho)\big), \qquad \text{with }\quad \mathcal{H}(\rho)=\tfrac{2\rho-1}{\rho} \mathbf{1}_{\{\rho > \frac12\}}, 
\end{equation}
where $\mathbf{1}_{\{\rho > \frac12\}}$ is the indicator function which equals 1 on the active phase $(\frac12,1]$ and 0 on the frozen phase $[0,\frac12]$.  The solution to \eqref{eq:stef} has very poor regularity properties, since it is generically discontinuous at the free boundary.
This hydrodynamic limit result (see Theorem \ref{thm:hydro} below) is the first main outcome of this article.

\subsection{Hydrodynamic limit} To derive the Stefan problem as stated in Theorem \ref{thm:hydro}, the presence of a phase transition prevents the use of standard methods, as the ones exposed in \cite{KL}.  Indeed, the presence of two phases whose stationary measures have disjoint support prohibits using the \emph{entropy method}, whose center argument relies on comparing the distribution of the process with a global reference measure. The finer \emph{relative entropy method} fails as well, because it requires the hydrodynamic limit to be smooth, which is not the case for the Stefan problem. Note that the extension of the relative entropy method to a parabolic differential equation proposed in \cite{BCSS} would also fail, since we are not able to construct a sufficiently \emph{good} approximation of the solution to our free boundary problem. 

In order to circumvent this difficulty, Funaki \cite{Funaki}, inspired by \cite{varadhan91}, exploits the concept of \emph{Young measures}. In his model (originally introduced in \cite{CS96}),  
two types of particles are present on the discrete lattice, ``ice'' particles which never
move, and ``water'' particles which evolve according to a speed-change exclusion process. They form two regions, and they interact only through the interfaces separating both regions. Funaki derives a Stefan problem by adapting Varadhan's idea coming from \cite{varadhan91} to his bi-phased model. 
One important ingredient to apply his strategy is to give a full characterization of the infinite volume stationary measures. For simple exclusion processes, this characterization follows from De Finetti's Theorem (\cite[Section 4.3]{liggett-intro}, \cite[Theorem 35.10, p.473]{Bill}). In \cite{Funaki}, the supercritical stationary measures are written as a mixture of canonical Gibbs measures using \cite{georgii}. In our case, Lemma \ref{lem:definetti} is obtained \emph{via} the mapping to a zero-range process and \cite{Andjel}.

Apart from \cite{Funaki}, other free boundary problems  have been derived from discrete microscopic models. In \cite{Tsu}, the author considers a generalized exclusion process with positive jump rates, reversible w.r.t.~product measures. He then argues that a tagged particle acts as a boundary between two phases and shows that its rescaled velocity converges to the solution of the implicit equation satisfied by the free boundary between two similar phases. \cite{CDMGP} investigates a simple exclusion process with injection and removal of mass at the boundaries, one of which is described as the right-most particle in the system rather than a fixed point in space. In \cite{LV}, the system of interest is described by two coupled simple exclusion processes with annihilating interaction at the contact point. 
The facilitated exclusion process stands apart because the two phases arise directly from the dynamics, rather than being implemented from the start in the definition of the model. The resulting Stefan problem is also more complex because it allows for a so-called \emph{mushy region}, i.e.~the frozen phase needs not be flat and featureless.

{One can also recall from \cite{GQ} that the occupancy set of the so-called \emph{internal DLA} grows according to a Stefan problem. The microscopic systems considered there are close to the zero-range process \eqref{eq:ZRgen} to which the FEP can be mapped, with the difference that the jump rates grow linearly with the number of particles on a given site. 

In \cite{demasietal}, the authors derive a \emph{two-phase} Stefan problem from a system of two exclusion processes (with different rates) in which particles of different type annihilate at a certain rate when they are on the same site. Contrary to ours, the process has product equilibrium measure and is amenable to the relative entropy method. \cite{hayashi} generalizes the process to allow different killing rates for the two types of particles, which leads to more complicated behaviors for the limiting PDE.

Let us finally mention the papers \cite{DT19}, \cite{Delarue}. These consider a somewhat reverse problem: the frozen phase invades a (supercooled) liquid phase. The former grows when a diffusing particle from the liquid phase comes in contact with the frozen region. The parameters in the frozen region are irrelevant, so that in the limit we have a one-phase Stefan problem with no mushy region. On the other hand, since the interface can travel at arbitrary large speeds in the microscopic world, the macroscopic equation can exhibit blow-up in finite time.
}

\subsection{Microscopic phases} As noted in \cite{BESS}, in addition to blocked and ergodic configurations, the facilitated exclusion process also presents \emph{transient} configurations with mixed features, contrary to \cite{Funaki}. It is clear that in finite volume they disappear in finite time, but it would be conceivable that in the hydrodynamic limit the process remains in this undecided state. It turns out that this does not happen. 

In \cite{BESS} we show that, if the initial density is larger than  the critical value $\frac12$, after a subdiffusive transition time of order $(\log N)^\alpha$, with high probability the system  enters the irreducible component -- if the initial configuration belongs to the class of so-called \emph{regular configurations} , which happens with high probability for reasonable initial conditions (see \cite[Section 4]{BESS}).  

In the present setting with two macroscopic phases, it is clear that this is no longer true. However, we can hope for the next best thing: that after a subdiffusive transition time, there is a way to split the system in two parts, one ergodic and the other blocked, that match the macroscopic super-- and subcritical phases. Since our hydrodynamic limit result is obtained in a weak sense, one cannot extract this information directly from Theorem~\ref{thm:hydro}. 
Therefore, we formulate this in an additional result,  Theorem \ref{thm:fronts}, which is the second main outcome of this paper. In order to state the desired property rigorously, we need a good notion of macroscopic interfaces, derived directly from the PDE \eqref{eq:stef}, which is given in Proposition \ref{ass:stronguniqueness}.  To prove that result, we use PDE techniques as such exposed in \cite{And,Meirmanov}. The problems of existence, regularity and uniqueness of solutions to Stefan problems have been  investigated for years, and always raise obstacles which are overcome by refined approaches: as the literature is huge, we give here only a partial list of works which treat similar equations as \eqref{eq:stef}, see for instance \cite{And, DKor, Kor, KorMoo, Meirmanov, Meir2}.

\subsection{Outline of the paper} In Section \ref{sec:model} we give a complete description of the microscopic dynamics, together with its main characteristics (presence of distinct phases), and we state the two main results (Theorem \ref{thm:hydro} and Theorem \ref{thm:fronts}). Section \ref{sec:young} is devoted to the proof of the hydrodynamic limit, following Funaki's proof based on Young measures. This strategy needs two main ingredients: the ergodic decomposition for the stationary measures (given in Lemma \ref{lem:definetti}), and a local law of large numbers reminiscent of the one-block estimate (given in Proposition \ref{prop:localerg}). In Section \ref{sec:HittingTime} we prove our second main result about the exact correspondence between the microscopic and macroscopic phases, by using ideas coming from \cite{BESS} in order to control the transition period of the microscopic system. We prove in the Appendix, for the sake of completeness,  several technical results, which do not contain important conceptual novelties, in particular the existence of macroscopic interfaces as stated in Proposition \ref{ass:stronguniqueness}.

\subsection{Notations}\label{sec:notations}
We collect here notations and conventions that we use throughout the paper. Since some of the results rely on \cite{BESS}, we will as often as possible keep the same notations.

First, $\N:=\{0,1,2,\cdots\}$ denotes the set of non-negative integers and $\N_*:=\N\backslash\{0\}$ the set of positive integers. For any finite set $\Lambda$ we denote by $|\Lambda|$ its cardinality. 

The parameter $N \in \bb N_*$ is always a scaling parameter and will go to infinity. We let $\T_N:=\Z/N\Z$ be the discrete torus of size $N$, which we will also write as $\{1,\dots,N\}$. Similarly, $\T:=\R/\Z=[0,1)$ is the one-dimensional continuous torus.
For an interval $\Lambda=[a,b] \subset \T$ or $\Lambda=[a,b] \subset \T_N$ of the discrete or continuous torus, we write $\min \Lambda=a$, $\max\Lambda=b$, even though the torus is not naturally ordered.

For any $\ell \in \N$ we set $B_\ell:=\{-\ell,\dots,\ell\}$ as the centered symmetric box of size $2\ell+1$, which can be seen as either a subset of $\T_N$ 
(if $2\ell+1\leqslant N$), or a subset of $\bb Z$. More generally, we define $B_\ell(x):=\{-\ell+x,\dots, \ell+x\}$ the box of size $2\ell+1$ centered at  $x$.
 Similarly, we set $\Lambda_\ell:=\{0,\ldots,\ell\}$ and $\Lambda_\ell(x):=\{x,\ldots,x+\ell\}$. 

We will consider  configurations of particles on discrete sets $A$, with $A$ either $\bb Z$, the discrete torus $\T_N$, or a finite box $\Lambda \Subset \bb Z$. These configurations are of \emph{exclusion type}, meaning that no more than one particle can occupy any site of the lattice. They are generically denoted by $\eta \in \{0,1\}^A$. In particular, we denote by $\Sigma_N:=\{0,1\} ^{\T_N}$ the set of periodic configurations and by $\Sigma:=\{0,1\}^{\bb Z}$ the set of infinite ones.  For any $x\in A$ and configuration $\eta\in\{0,1\}^A$, we denote by $\eta_x\in\{0,1\}$ the particle number at site $x$. 
For any $\Lambda \subset \T_N$ (or $\Lambda\subset \bb Z$) the configuration $\eta \in \Sigma_N$ (or $\in \Sigma$)  restricted to $\Lambda$ is denoted by $\eta_{|\Lambda}$. We say that a function $f:\{0,1\}^\Z \to \R$ is \emph{local} if there exists $\Lambda$ a finite subset of $\bb Z$ such that $f(\eta)$ depends only on $\eta_{|\Lambda}$.  For any  probability measure $\pi$ on $\{0,1\}^{\Lambda}$,  and $f:\{0,1\}^\Lambda\to\bb R$ measurable function,  $\pi(f)$ denotes the expectation of $f$ w.r.t.~the measure $\pi$.  For any $f:\Sigma_N \to \R$ measurable, and $x\in\T_N$, we denote by $\tau_x f$ the  function obtained by translation as follows: 
 $\tau_x f(\eta):=f(\tau_x \eta)$, where $(\tau_x\eta)_y = \eta_{x+y},$ for any $y\in\bb T_N$.
 
 More generally, if $\mathcal{P}$ is a probability measure on a set $E$, and $f$ is a measurable function defined on $E$, we denote by $\mathcal{P}(f)$ the expectation of $f$ with respect to $\mathcal{P}$.

For any sequence $(u_k)_{k\in \N}$, possibly depending on other parameters than the index $k$, we will denote $\mc O_k(u_k)$ (resp.~$o_k(u_k)$) an 
arbitrary sequence $(v_k)_{k\in \N}$ such that there exists a constant $C>0$ (resp.~a vanishing sequence $(\varepsilon_k)_{k\in \N}$) -- 
possibly depending on the other parameters -- such that 
\[\text{for all } k \in \N, \quad |v_k|\leq C |u_k|\quad (\mbox{resp. } |v_k|\leq |u_k|\varepsilon_k).\] We will omit the subscript $k$ when clear from context.

A function $f:I \times \T \to \R$, where $I \subset \R_+$ is an interval, is in  $C^{\alpha,\beta}(I \times \T)$ if it is of class $C^\alpha$ in the first variable, and of class $C^\beta$ in the second variable. If $f$ is defined on a neighborhood of $x$, we write $f(x^+)$ (resp. $f(x^-)$) for $\lim_{y \to x, y>x} f(y) =: \lim_{y \to x^+} f(y) $ (resp. $\lim_{y \to x, y <x} f(y) =: \lim_{y \to x^-} f(y)$).

\section{Model and results} 
\label{sec:model}

\subsection{The microscopic dynamics}
\label{sec:Modeldynamics}

Let us first introduce the \emph{facilitated exclusion process} described in the introduction, which  is a Markov process on the set of periodic configurations $ \eta \in \Sigma_N=\{0,1\}^{\T_N}$.

The infinitesimal generator ruling the evolution in time of this Markov process is given by $\mathcal{L}_N$, which acts on  functions $f:\Sigma_N \to \R$ as 
\begin{equation}
\label{eq:DefLN}
\mathcal{L}_Nf(\eta):=\sum_{x\in\T_N}c_{x,x+1}(\eta)\big(f(\eta^{x,x+1})-f(\eta)\big),
\end{equation}
where $\eta^{x,y}$ denotes the configuration obtained from $\eta$ by swapping the values at sites $x$ and $y$, 
namely $(\eta^{x,y})_x=\eta_y$, $(\eta^{x,y})_y=\eta_x$ and $(\eta^{x,y})_z=\eta_z$ if $z\neq x,y$. Moreover,  the jump rates $c_{x,y}(\eta)$ translate the \emph{exclusion rule} (no more than one particle at each site) and \emph{dynamical constraint} (a particle needs to be pushed to an empty site) as follows:
\begin{equation} \label{eq:rate} c_{x,x+1}(\eta)=\eta_{x-1}\eta_x(1-\eta_{x+1})+(1-\eta_{x})\eta_{x+1}\eta_{x+2}.\end{equation}
Let us recall the main properties of this model, which have been already detailed  in \cite{BESS}: first, the dynamics conserves the total number of particles $\sum_{x\in\T_N}\eta_x$. Elementary computations yield that the following local conservation law holds: for any $x\in \T_N$,
\[\mathcal{L}_N\eta_x=j_{x-1,x}-j_{x,x+1},\]
where the instantaneous current $j_{x,x+1}=-c_{x,x+1}(\eta)(\eta_{x+1}-\eta_x)=\tau_{x}h-\tau_{x+1} h$, is the discrete gradient of the local function
\begin{equation}
\label{eq:defh}
h(\eta)=\eta_{-1}\eta_0+\eta_0\eta_1-\eta_{-1}\eta_0\eta_1.
\end{equation}
Since it satisfies this last property, the facilitated exclusion process considered here is a \emph{gradient model}. It is also \emph{degenerate}, since the jump rates can vanish for non trivial configurations.

Fix an initial density profile $\rho^{\rm ini}:\bb T\to [0,1]$.
We will consider, as initial condition, a random configuration of particles which is distributed according to a non-homogeneous Bernoulli product measure on $\Sigma_N$ fitting $\rho^{\rm ini}$, defined as 
\begin{equation}
\label{eq:DefmuN}
\mu^N(\eta):=\prod_{x\in \bb T_N}\big(\rho^{\rm ini}(\tfrac x N)\eta_x+(1-\rho^{\rm ini}(\tfrac x N))(1-\eta_x)\big).
\end{equation}
The \emph{invariant measures} of this process have been deeply investigated in \cite[Section 6]{BESS}. Due to the strong dynamical constraint, they are not independent products of homogeneous Bernoulli measures (as it is often the case for exclusion processes), but they can be made fully explicit. Moreover, there is a critical density $\rho_\star$ (given in the next section) such that, if the density is bigger than $\rho_\star$, then there is a unique invariant measure, while all the invariant measures are superpositions of atoms if the density is less than $\rho_\star$. More details will be given in Section \ref{sec:invariant-measures}. 

\begin{remark}[On the initial distribution $\mu^N$]
Proving the hydrodynamic limit result (Proposition \ref{thm:hydro} below)  only requires the convergence in distribution of the empirical density at initial time, namely,
\[\frac1N\sum_{x\in\T_N}\varphi(\tfrac x N)\eta_x(0)\xrightarrow[N\to\infty]{} \int_{\T}\varphi(u)\rho^{{\rm ini}}(u)du,\] for any test function $\varphi$,  where the above convergence holds in probability under $\mu^N$. 
However, in the second part, in the investigation of the creation of  microscopic fronts (Theorem \ref{thm:fronts} below), one requires some sharp decay of the correlations of the initial distribution. For the sake of clarity, we do not aim at having minimal assumptions on the initial distribution (which is not the main issue here) and choose as initial distribution the product measure \eqref{eq:DefmuN} throughout the paper.
\end{remark}

\subsection{Ergodic and frozen phases}

The facilitated exclusion process displays a \emph{phase transition}. Indeed, because of the microscopic jump constraint, pairs of neighboring empty sites cannot be created by the dynamics. In particular, assuming that initially, at least half of the sites are occupied, particles will diffuse in the microscopic system until there are no longer two neighboring empty sites. On the other hand, if initially at least half of the sites are empty, particles will diffuse until the moment when each particle is surrounded by empty sites and can no longer move. For this reason, given  $\Lambda\subset\Z$ or $\Lambda\subset \T_N$,  we  now introduce the set of \emph{ergodic} (resp. \emph{frozen}) configurations as:
\begin{equation}
\label{eq:ergset}
\mathscr{E}_\Lambda=\big\{\eta\in \{0,1\}^\Lambda\; ; \; \eta_x+\eta_{x+1}\geq1, \text{ for all } x\in \Lambda \mbox{ such that }x+1\in \Lambda\big\},
\end{equation}
namely the set of configurations where all empty sites are isolated, 
resp.
\begin{equation}
\label{eq:froset}
\mathscr{F}_\Lambda=\big\{\eta\in \{0,1\}^\Lambda\; ; \; \eta_x+\eta_{x+1}\leq1, \text{ for all } x\in \Lambda \mbox{ such that }x+1\in \Lambda\big\},
\end{equation}
namely the set of configurations where all particles are isolated. An example of an element belonging to each set is given in Figure \ref{fig:setEF}.

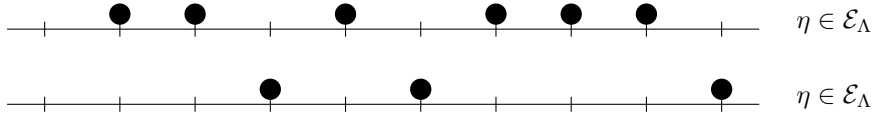
\begin{figure}[h]
\centering
\begin{tikzpicture}
\node at (1,0.5) {\color{white}x};
\draw (-1,0) -- (9,0);
\foreach \i in {-1,...,8}
{
\draw (\i+0.5,-0.1) -- (\i+0.5,0.1);
}
\node[circle,fill=black,inner sep=1mm] at (0.5,0.2) {};
\node[circle,fill=black,inner sep=1mm] at (3.5,0.2) {};
\node[circle,fill=black,inner sep=1mm] at (5.5,0.2) {};
\node[circle,fill=black,inner sep=1mm] at (1.5,0.2) {};
\node[circle,fill=black,inner sep=1mm] at (6.5,0.2) {};
\node[circle,fill=black,inner sep=1mm] at (7.5,0.2) {};

\node at (10,0.1) {$\eta \in \mathcal{E}_\Lambda$};

\draw (-1,-1) -- (9,-1);
\foreach \i in {-1,...,8}
{
\draw (\i+0.5,-1.1) -- (\i+0.5,-0.9);
}
\node[circle,fill=black,inner sep=1mm] at (4.5,-0.8) {};
\node[circle,fill=black,inner sep=1mm] at (8.5,-0.8) {};
\node[circle,fill=black,inner sep=1mm] at (2.5,-0.8) {};
\node at (10,-0.9) {$\eta \in \mathcal{E}_\Lambda$};
\end{tikzpicture}
\caption{Example of configurations belonging to the ergodic and frozen sets, with $|\Lambda| =10$.}
\label{fig:setEF}
\end{figure}

At the macroscopic level,  this means that there are two distinct regimes for the behavior of this model\footnote{Note however that $\mathscr{E}_\Lambda\cap\mathscr{F}_\Lambda$ is non-empty since it contains alternated \emph{particle/empty site} configurations.}. Either the macroscopic density is larger than  the critical value $\rho_\star:=\frac12$, 
in which case the system behaves diffusively, or the density is lower than $\frac12$,  in which case the system remains frozen after a transitive period during which the particles tend to isolate themselves. 
The interfaces between these two macroscopic phases move as particles from the supercritical phase ($\rho>\frac12$) diffuse towards the subcritical phase ($\rho<\frac12$).

As we already noted in \cite{BESS},  there are \emph{transitive} (or \emph{transient}) configurations, which are neither ergodic nor frozen ($\mathcal{E}_\Lambda \cup \mathcal{F}_\Lambda \neq \{0,1\}^\Lambda$). However, they are called \emph{transient} in \cite{BESS} because, if $\Lambda \Subset \bb Z$ is finite, then from these transient configurations the process will evolve toward $\mathcal{E}_\Lambda \cup \mathcal{F}_\Lambda$ after a number of particle jumps which is a.s.~finite. More precisely, in \cite{BESS} we show that, if the initial configuration of particles is distributed according to $\mu_N$ (defined in \eqref{eq:DefmuN}), with $\rho^{\rm ini}(\T) \subset (\frac12,1]$ (therefore, the initial density profile is uniformly larger than the critical density), then the microscopic system of size $N$ needs a subdiffusive time $t_N = o(N^2)$ in order to reach the ergodic component.

\subsection{Free boundary problem} 
In this section, we turn to the macroscopic point of view, and first give an explicit \emph{free boundary problem}, for which we explain what we intend by a solution.

\begin{definition}[Weak solution of the free boundary problem]
\label{Def:PDEweak}
For any $r\geq 0$, define the function
\begin{equation*} 
\mathcal{H}(r)=\frac{2r-1}{r}{\bf1 }_{\{r>\frac12\}},
\end{equation*}
and denote by $\langle f,g\rangle$ the inner product of $f$ and $g$ in $L^2(du)$ on $\bb T$. 

\medskip

Fix $T \geqslant 0$ and let $\rho^{\rm ini}:\T\to[0,1]$ be a measurable initial profile. 
We call a  measurable function $\rho:(t,u)\mapsto\rho_t(u)$ a \emph{weak solution to the free boundary problem}
\begin{equation}
\label{eq:PDEstrong}
\partial_t\rho=\partial_u^2\big(\mathcal{H}(\rho)\big)
\end{equation}
with initial condition $\rho_0=\rho^{\mathrm{ini}}$, if: \begin{itemize}
\item for any $(t,u) \in [0,T]\times \T$,  $\rho_t(u)\in [0,1]$,
\smallskip
\item and for any test function  $\varphi\in C^{1,2}([0,T]\times \T)$ 
\begin{equation}
\label{eq:weakF1}
\big\langle \rho_T,\varphi_T \big\rangle = \big\langle \rho^{\rm ini}, \varphi_0 \big\rangle+\int_0^T\big\langle \rho_t,\partial_t\varphi_t \big\rangle dt +\int_0^T \big\langle \mathcal{H}(\rho_t),\;  \partial_u^2\varphi_t \big\rangle dt.
\end{equation}
\end{itemize}
\end{definition}

\begin{remark} Let us briefly comment on the function $\mathcal{H}$. As pointed out in Section \ref{sec:Modeldynamics}, the generator of the process acts as a discrete Laplacian, in the sense that $\gene \eta_x=\tau_{x+1}h+\tau_{x-1}h-2\tau_{x}h$. In the frozen phase, the function $h$ vanishes. However, in the ergodic phase, and under the equilibrium measure $\pi_\rho$ at density $\rho$ (defined in Definition \ref{def:gcm} below), its average equals $\int h(\eta) d\pi_\rho(\eta)=\mathcal{H}(\rho)$. 
\end{remark}

Intuitively, the configurations evolving according to \eqref{Def:PDEweak} should be separated in two phases: a region of density above $\frac12$, where the macroscopic evolution is given by an elliptic equation, and a (frozen) region of density below $\frac12$ which is gradually filled by the spreading of the supercritical region. Making this picture rigorous is not immediate, given the weakness of the above notion of solution, even for ``simple'' initial density profiles, as we explain below.

\begin{proposition}[Uniqueness of weak solutions \cite{Uchiyama}]
\label{prop:uniqueness}
There exists  a unique weak solution of \eqref{eq:PDEstrong} with initial condition $\rho_0=\rho^{\rm ini}$, in the sense of Definition \ref{Def:PDEweak}.
\end{proposition}
Since $\mathcal{H}:[0,1]\to [0,1]$ is a non-decreasing function, this proposition is already proved by Uchiyama, see \cite[Theorem 6, p.10]{Uchiyama}, and therefore we do not reproduce the proof here.

\subsection{Main results} 

As already noted in \cite{BESS}, the macroscopic behavior of the facilitated exclusion process appears in the \emph{diffusive} time scale. Therefore, we define $\{\eta(t) \; ; \; t\geq 0\}$ as the Markov process driven by the \emph{accelerated} infinitesimal generator $N^2\mathcal{L}_N$ 
 and initially distributed as $\mu^N$ (given in \eqref{eq:DefmuN}). 
Fix $T>0$ and denote by $\P_{\mu^N}$ the probability measure on the Skorokhod path space $\mc D([0,T],\Sigma_N)$ corresponding to this dynamics. 
We denote by $\E_{\mu^N}$ the corresponding expectation.  We emphasize that, even though it is not explicit in the notation, $\P$, $\E$ and $\eta(t)$ strongly depend  on $N$: 
through the size of the state space, but also through the diffusive time scaling.

\begin{theorem}[Hydrodynamic limit]
\label{thm:hydro}
Let $\rho^{\rm ini}:\T\to[0,1]$ be a measurable initial profile. For any $t \in [0,T]$, any $\delta >0$,  and any smooth test function $\varphi:\T\to\R$, we have
\begin{equation}\label{eq:limith} \lim_{N\to\infty} \P_{\mu^N}\bigg[\bigg|\frac1N\sum_{x\in\T_N}\varphi(\tfrac x N)\eta_x(t) - \int_{\T}\varphi(u)\rho_t(u)du\bigg|>\delta\bigg]=0,\end{equation}
where $\rho$ is the unique weak solution of \eqref{eq:PDEstrong} with initial condition $\rho_0=\rho^{\rm ini}$, in the sense of Definition \ref{Def:PDEweak}.
\end{theorem}

In order to prove Theorem \ref{thm:hydro}, we adapt a technique used by Funaki \cite{Funaki}, and inspired by \cite{varadhan91}, and  we exploit the concept of \emph{Young measures}, as explained in Section \ref{sec:young}. One of the drawbacks of this method, however, is that the hydrodynamic limit is proved in a rather indirect way, which says nothing about the  separation of the microscopic configuration into two phases matching the macroscopic ones. Therefore, we also prove in Theorem \ref{thm:fronts} below that after a subdiffusive time, the microscopic picture corresponds exactly to the macroscopic one, in two steps: first, with each subcritical or supercritical macroscopic phase can be associated a connected subcritical or supercritical microscopic box, and second, the microscopic interface between those two boxes is indeed close to the macroscopic one coming from the free boundary problem \eqref{eq:PDEstrong}. We emphasize that this result is not necessary to prove Theorem \ref{thm:hydro}, it is an independent outcome.

For our second  theorem, we need two extra assumptions on the initial profile $\rho^{\rm ini}$.
Let us introduce $\mathcal{C}^0:=(\rho^{\rm ini})^{-1}(\{\frac12\})\subset \T$ the set of \emph{critical points} of the initial density profile. Assume that
\begin{align}
\label{ass:1disc}
\tag{H1}
& \mathcal{C}^0\mbox{ is a finite set with cardinality $c(\rho^{\rm ini}) \in \N$,}
\\ 
\label{ass:2reg}
\tag{H2}
&\rho^{\rm ini}\in C^2(\bb T)\quad \mbox{ and }\quad \partial_u\rho^{\rm ini}(u)\neq 0 \quad \forall u\in \mathcal{C}^0.
\end{align}
In this setting, unfortunately, formalizing rigorously the notion of macroscopic interfaces for the solutions of $\eqref{eq:PDEstrong}$, and the corresponding microscopic features, can be rather cumbersome. We will therefore state and prove all our results in the case 
\begin{equation}
\label{ass:technical}
\tag{T1}
\rho^{\rm ini}<1,\quad |\mathcal{C}^0|=2 \quad \mbox{ and }\quad (\rho^{\rm ini})^{-1}([0,\tfrac12])=[0,u_*].
\end{equation}
In other words, we assume that there are  only two critical points:  $0\equiv 1$ and $u_*\in(0,1)$, the initial subcritical phase is the segment $[0,u_*]$, the initial supercritical phase is the complementary segment $[u_*, 1]$, and at no point in the supercritical phase is density $1$ reached. We stress out that this is not necessary for our proof to hold, and that both Definition \ref{def:frontcreation} and Theorem \ref{thm:fronts} below can be straightforwardly extended to any even integer $c(\rho^{\rm ini})$.

\begin{definition}[Two-phased configurations]
\label{def:two-phases}
We say that a configuration $\eta \in \Sigma_N$ is \emph{two-phased} if there exists a partition $\T_N=\mathbf{E} \sqcup \mathbf{F}$, such that $\mathbf{E}, \mathbf{F} \subset \Sigma_N$ are (possibly empty) connected subsets, and 
\[\eta(t)_{|\mathbf{E}}\in \mathcal{E}_{\mathbf{E}}\quad \mbox{ and }\quad \eta(t)_{|\mathbf{F}}\in \mathcal{F}_{\mathbf{F}},\] 
where the set of ergodic and frozen configurations were defined in \eqref{eq:ergset} and \eqref{eq:froset}.
We denote by $\mathfrak{P}_N$ the set of non-ergodic two-phased configuration, and for any $\eta\in\mathfrak{P}_N$, we denote $\mathbf{E}_\eta$ and $\mathbf{F}_\eta$ (the latter must be nonempty if $\eta$ is non-ergodic) the corresponding ergodic and frozen sets. To ensure uniqueness, we choose $\mathbf{E}_\eta$ to be \emph{maximal} for inclusion. 
\end{definition}

Note that a two-phased configuration remains two-phased or ergodic along the dynamics:
\begin{equation}\label{eq:two-phasedstable}
\text{if } \eta(s)\in \mathfrak{P}_N, \text{ then for all }t\geq s, \ \eta(t)\in \mathfrak{P}_N\cup\mathcal{E}_{\T_N}.
\end{equation}
Therefore, we can define

\begin{definition}[Microscopic fronts]
\label{def:frontcreation}
Assume that at a time $t$, $\eta(t)\in \mathfrak{P}_N$. Then we define 
\[u^N_-(t)=\min \mathbf{F}_{\eta(t)}\quad \mbox{ and }\quad u^N_+(t)=\max \mathbf{F}_{\eta(t)},\]
which correspond to the position of the microscopic fronts, with the convention that 
\[u^N_-(t) =\max_{\substack{s\leq t\\
\eta(s)\in \mathfrak{P}_N}}u^N_-(s)\quad \mbox{ and }\quad u^N_+(t) =\min_{\substack{s\leq t\\
\eta(s)\in \mathfrak{P}_N}}u^N_-(s)\]
if $\eta(t)\in \mathcal{E}_{\T_N}$ has already become ergodic. In other words, once the microscopic fronts have merged and the configuration becomes ergodic, we arbitrarily set the position of the microscopic fronts at the last site where the frozen set was non empty.
\end{definition}

Theorem \ref{thm:fronts} below states that the configuration becomes two-phased in a subdiffusive time with high probability. It also states that  the boundaries of the frozen set (\emph{i.e.}~the microscopic fronts as defined in Definition \ref{def:frontcreation}) are never far from the macroscopic interfaces. To state this result, we need to show that the latter are well defined. The following result is proved in Appendix~\ref{a:strong-existence}.

\begin{proposition}
\label{ass:stronguniqueness}
Assume Assumptions \eqref{ass:1disc}, \eqref{ass:2reg} and \eqref{ass:technical}.
 For any $T>0$, the weak solution $\rho$ of the free boundary problem \eqref{eq:PDEstrong} admits continuous \emph{macroscopic interfaces} $u_-,u_+:[0,T]\to\T$, respectively non-decreasing and non-increasing, satisfying $u_-(0)=0$, $u_+(0)=u_*$. Moreover, there exists $\tau\in\R_+\cup\{\infty\}$ such that
\begin{enumerate}
\item for any $t<\tau$, $u_-(t)\neq u_+(t)$, and  \[ \rho_t(u) \begin{cases} < \frac12  & \text{ if } u\in(u_-(t),u_+(t))\\ >\frac12 & \text{ if } u\in(u_+(t),u_-(t)); \end{cases}\]
\item if  $\tau<\infty$ then $u_-(\tau)=u_+(\tau)$;
\item for any $t\geq\tau$,  $\rho_t\geq \frac12$ on $\T$, and moreover $u_+,u_-$ are constant: $u_+(t)=u_-(t)=u_-(\tau)$.
\end{enumerate}  
\end{proposition}

We are now ready to state our second main result.
\begin{theorem}
\label{thm:fronts}
Assume \eqref{ass:1disc}, \eqref{ass:2reg} and  \eqref{ass:technical}.
\begin{enumerate}
\item \emph{Creation of fronts}.
Letting $t_N=N^{-1/4}$,
\begin{equation*}
\lim_{N\to\infty} \P_{\mu^N}\pa{\eta(t_N)\in\mathfrak{P}_N}=1,
\end{equation*}
i.e. in a time of order $N^{-1/4}$, the microscopic configuration is two-phased with high probability.
\item \emph{Macroscopic match}. For any $t\in(0, \tau]\cap\R_+$,
\[\lim_{N\to\infty} \P_{\mu^N} \bigg(\Big| \tfrac 1 N u^N_\pm (t) - u_\pm(t)\Big| \geqslant \varepsilon \bigg) = 0,\]
where $\tau, u_+,u_-$ are defined in Proposition~\ref{ass:stronguniqueness} and $u^N_\pm$ in Definition \ref{def:frontcreation}. 
\end{enumerate}
\end{theorem}
The result actually also holds for $t_N=N^{-m}$ for any $m< \frac23$, but in order to focus on the important points of the proof, we choose simpler exponents in the required estimates, so that taking $t_N=N^{-1/4}$ is convenient. As will be shown in Section~\ref{s:proofth2}, point \emph{(2)} is actually a simple consequence of \emph{(1)} and the hydrodynamic limit result (Theorem \ref{thm:hydro}).

\begin{remark}[On assumption \eqref{ass:2reg}]
The regularity of the initial profile $\rho^{\rm ini}$ at the critical points is crucial to our proof. However, away from the critical points, the regularity assumption could be weakened. To focus on the important points of the proof, we settle for assumption \eqref{ass:2reg}.
\end{remark}

\section{Proof of Theorem \ref{thm:hydro}: Young measures and hydrodynamic limit}
\label{sec:young}

We prove in this section the hydrodynamic limit result stated in Theorem \ref{thm:hydro}, following the strategy given in \cite{Funaki}. 

\subsection{Empirical measure}
For any $t \in [0,T]$, let us define
\begin{equation}
\label{eq:DefmN}
m_t^N(du)=m^N(\eta(t),du):=\frac 1N \sum_{x\in \bb T_N}\eta_x(t)\delta_{x/N}(du),
\end{equation}
the empirical measure of the process, where $\delta_a(du)$ stands for the Dirac measure on $\T$ at point $a \in \T$. The measure $m_t^N$ is an element of the set $\mc M_+(\T)$ of positive measures on the torus $\T$, which we endow with the weak topology. We slightly abuse our  notation for the inner product in $L^2(\T)$, and also denote by $\langle m, \cdot \rangle$ the integral on $\T$ with respect to any measure $m(du)$.

Let us denote by $ \mathcal{P}_N$ the pushforward measure of $\bb P_{\mu^N}$ by the mapping 
$m^N$, namely $\mathcal{P}_N:=\bb P_{\mu^N}\circ \pa{m^N}^{-1}$. Then, $\mathcal{P}_N$ is a probability measure on the path space $\mc D([0,T], \mc M_+(\T))$, endowed with the Skorokhod topology. In order to prove Theorem \ref{thm:hydro}, we are  reduced to proving the convergence of the sequence $(\mathcal{P}_N)$ towards the Dirac probability measure concentrated on the solution of \eqref{eq:PDEstrong}.

We already know some properties of the sequence $(\mathcal{P}_N)$, which are quite standard in the literature:

\begin{proposition}[Absolute continuity w.r.t.~the Lebesgue measure]
\label{prop:Lebesgue}
The sequence $(\mathcal{P}_N)$ is weakly relatively compact in  $\mc D([0,T],\mc M_+(\T))$, and any of its limit points $ \mathcal{P}^*$ is concentrated on trajectories of measures $\{m_t(du) \; ; \; t \in [0,T]\}$ which are
\begin{enumerate}
\item continuous in time, i.e.
\begin{equation}
\label{eq:PropLeb1}
\mathcal{P}^*\pa{t\mapsto m_t\;\mbox{ is continuous}}=1,
\end{equation}

 \item and whose marginal at time $t$ is absolutely continuous w.r.t.~the Lebesgue measure on $\bb T$, i.e.
 \begin{equation}
\label{eq:PropLeb2}
\mathcal{P}^*\big(\forall \; t\in [0, T],\; \exists\; \rho_t:\T \to [0,1],\; s.t. \;  m_t(du)=\rho_t(u)du\big)=1.
\end{equation}
\end{enumerate}
In particular, these two assertions prove that 
\begin{equation*}
\mathcal{P}^*\big(\forall \; t\in [0, T],\; \exists\; \rho_t:\T \to [0,1] \text{ continuous in } t , \text{ s.t.}\; m_t(du)=\rho_t(u)du\big)=1.
\end{equation*}
\end{proposition}

This proposition will be proved in Appendix \ref{app:topo} for the sake of completeness, but it is standard. In many models, the proof of the hydrodynamic limit can be completed from there by using the \emph{entropy method}. However, for this model, this standard strategy fails because of the presence of supercritical and subcritical phases on which the time invariant measures are distinct and not absolutely continuous w.r.t.\ one another. For that reason, we now introduce  the concept of \emph{Young measures}, as given in \cite{Funaki}.

\subsection{Young measures and sketch of the proof}
 For that purpose, we need to introduce some notations. Given a configuration $\eta$, let us denote by
\begin{equation}
\label{eq:Defrhoell} 
\rho^\ell_x=\rho^\ell_x( \eta):=\frac{1}{2\ell+1}\sum_{y\in B_\ell(x)}\eta_y
\end{equation}
the local density in the box $B_\ell(x)$ of size $2\ell+1$ around  $x$ (defined in Section \ref{sec:notations}). When $x=0$, to simplify notations, we denote $\rho^\ell=\rho^\ell_0$. 
When $\eta$ is a time trajectory, and the density is observed at time $t$, we denote for the sake of clarity $\rho_x^\ell(t)=\rho^\ell_x (\eta(t))$ and $\rho^\ell(t)=\rho^\ell( \eta(t))$.
  
\begin{definition}[Young measure]
\label{def:Youngmeasure}
Let us fix an integer $\ell$.  The \emph{Young measure} $\pi^{N,\ell}$ on $\bb T\times [0,1]$ is given for any configuration $\eta$ of particles by
\[\pi^{N,\ell}(du,dr)=\pi^{N,\ell}(\eta, du,dr):=\frac{1}{N}\sum_{x\in \bb T_N}\delta_{x/N}(du)\;\delta_{\rho^\ell_x}(dr).\]
For any measure $\pi$ on $\bb T\times [0,1]$, any function $\xi$ defined on $\T$, and any function $\psi$ defined on $[0,1]$,
we denote by $\ang{\pi, \xi \cdot \psi}$ the integral of the function $(\xi \cdot \psi)(u,r):=\xi(u)\psi(r)$ w.r.t.~the measure $\pi$. 
\end{definition}

Similarly as before, let us define, for time trajectories, \[\pi_t^{N,\ell}:=\pi^{N,\ell}(\eta(t)).\]
 
\begin{remark}
Observe that, for any smooth function $\xi$ defined on $\T$, and taking $\psi(r)=r$, an integration by parts shows that there exists a constant $C(\xi)>0$ such that
\begin{equation}\label{eq:estimipp} \Big|\langle m_t^N, \xi \rangle - \ang{ \pi_t^{N,\ell}, \xi\cdot r } \Big| \leqslant C(\xi) \; \frac{\ell}{N}. \end{equation}
\end{remark}

We now define $\bar{\mathcal{P}}_{N,\ell}$ as the pushforward  measure of $\bb P_{\mu^N}$ by the mapping \[\big((m^N)^{-1}, (\pi^{N,\ell})^{-1}\big),\] namely for any measurable set  $\mathcal{B}$,
\[\bar{\mathcal{P}}_{N,\ell}\Big(\{m_t,\pi_t\}_{t\in [0,T]}\in\mathcal B\Big)=\bb P_{\mu^N}\pa{\{m_t^N,\pi_t^{N,\ell}\}_{t\in[0,T]}\in\mathcal B},\]
which is a probability measure on $\mc D\big([0,T],\mc M_+(\T)\times \mc M_+(\T\times [0,1])\big)$.
We first state a technical lemma.
\begin{lemma}
\label{lem:YM}
The sequence $(\bar{\mathcal{P}}_{N,\ell})_{1\leq \ell\leq N}$ is weakly relatively compact, and any of its limit points $\bar{\mathcal{P}}^*$ as $N\rightarrow\infty$ then $\ell\rightarrow\infty$\footnote{More precisely, by this expression we mean that we take limits of convergent subsequences as $N\rightarrow\infty$ for fixed $\ell$, then take a convergent subsequence of these objects as $\ell\rightarrow\infty$.} satisfies 
\begin{equation}
\label{eq:YM1}
\bar{\mathcal{P}}^*\bigg(\forall\; t \in [0,T],\; \exists\; \rho_t(\cdot), \; p_t(\cdot, dr), \text{ s.t. }\; \left\{ \begin{aligned} & m_t(du) =\rho_t(u)du \\ & \pi_t(du ,dr)=p_t(u,dr) du \end{aligned} \right. \bigg)=1.
\end{equation}

\end{lemma}
\begin{proof}[Proof of Lemma \ref{lem:YM}]
Since the first marginal of $\bar{\mathcal{P}}_{N,\ell}$ is ${\mathcal{P}}^N$, the fact that $\bar{\mathcal{P}}^*$--a.s., $m_t(du)$ is time continuous and absolutely continuous at 
every time $t$ w.r.t.~the Lebesgue measure is a direct consequence  of  Proposition \ref{prop:Lebesgue}. 
Moreover, $\bar{\mathcal{P}}^*$--a.s., this is also the case of $\pi_t(du ,dr)$, since  one can easily check after passing to the limit in $\ang{\pi_t^{N,\ell}, \xi\cdot 1}$ that, for any smooth function $\xi$ on $\bb T$
\[\int_{\bb T}\int_{[0,1]}\pi_t(du,dr) \xi(u)\leq\int_{\bb T} \xi(u) du,\]
which proves \eqref{eq:YM1}. Note that all those estimates are  deterministic, in the sense that the only used property  is the exclusion rule (at most one particle per site is allowed in the configuration). For this reason, the quantifier ``$\forall\; t \in [0,T]$'' can be inserted inside the probability, thus concluding the proof.
\end{proof}
We are now ready to state the main result of this section.
\begin{proposition}
\label{prop:YM}
The sequence $(\bar{\mathcal{P}}_{N,\ell})_{1\leq \ell\leq N}$ is weakly relatively compact, and any of its limit points $\bar{\mathcal{P}}^*$ as $N\to\infty$ then $\ell\to\infty$  satisfies 
\begin{equation}
\label{eq:YM2}
\bar{\mathcal{P}}^*\Big(\forall\; (t,u) \in [0,T] \times \T, \; \; p_t\big(u ,[0,\tfrac12]\big)=1, \; \mbox{ or }\;p_t(u ,dr)=\delta_{\rho_t(u)}(dr) \Big)=1,
\end{equation}
where $p_t$ and $\rho_t$ were defined ${\mathcal{P}}^*$--a.s. by \eqref{eq:YM1}.

In other words,  Young measures in $r$ either only charge the subcritical range of densities, or are trivial and given by a Dirac at $\rho_t(u)$.
Here, by limit point, we mean that we take any convergent subsequence as $ N\to\infty$, and then any convergent subsequence as $\ell\to\infty$.
\end{proposition}
Note that this proposition does not say anything about the function $\rho_t(u)$. However, we prove at the end of this paragraph that it is the weak solution of \eqref{eq:PDEstrong}.

  Proposition \ref{prop:YM} is a consequence of the following lemma.
\begin{lemma}
\label{lem:eqpi}
Any limit point $\bar{\mathcal{P}}^*$ as $N\to\infty$ then $\ell\to\infty$ of the sequence $(\bar{\mathcal{P}}_{N,\ell})_{1\leq \ell\leq N}$ satisfies 
\begin{equation}
\label{eq:rhorho'neg}
\bar{\mathcal{P}}^*\Bigg(\int_{0}^T \int_{\bb T} \int_{[0,1]}\mathcal{H}(r) \bigg(r-\int_{[0,1]}r'p_t(u, dr')\bigg)p_t( u, dr) du dt \Bigg)=0,
\end{equation}
where $\bar{\mathcal{P}}^*(\cdots)$ denotes the expectation w.r.t.\ $\bar{\mathcal{P}}^*$.
\end{lemma}
We postpone the proof of this lemma to the end of the section, see Section \ref{sec:prooflemma} for the conclusion.  Before proving it, we show that Proposition \ref{prop:YM} follows, and then we prove Theorem \ref{thm:hydro}.

\begin{proof}[Proof of Proposition \ref{prop:YM}]
We now show \eqref{eq:YM2}. Since for any fixed $(t,u) \in [0,T] \times \T$, $p_t(u,\cdot)$ is a probability measure on $[0,1]$, 
and since $\mathcal{H}(r)=\frac{2r-1}{r}{\bf 1}_{\{r\geq \frac12\}}$ 
is non-decreasing on $[0,1]$, we have
\begin{multline*}
\int_{[0,1]}\mathcal{H}(r) \bigg(r-\int_{[0,1]}r' p_t(u, dr')\bigg)p_t( u, dr)  \\
\geq \int_{[0,1]}\mathcal{H}(r')p_t(u, dr')\int_{[0,1]} \bigg(r-\int_{[0,1]}r'p_t(u, dr')\bigg)p_t( u, dr) =0. 
\end{multline*}
This follows from the inequality $\int\int (f(x)-f(y))(g(x)-g(y))d\mu(x)d\mu(y)\geq 0$, valid for any measure $\mu$ on $\R$ if $f,g$ are non-decreasing. The equality case in the above inequality happens when $(f(x)-f(y))(g(x)-g(y))=0$ a.e. In our case, this means that $\mathcal{H}$ should be constant on the support of $p_t(u,dr)$.

Therefore, from \eqref{eq:rhorho'neg}, we obtain that 
almost everywhere  w.r.t.~the Lebesgue measure in $[0,T]\times \bb T$: \begin{itemize}
\item \emph{either}: $p_t(u,[0,\frac12])=1$ (if $p_t(u,dr)\circ\mathcal{H}^{-1}=\delta_0$),
\item \emph{or}: $p_t(u,[0,\frac12])=0$ and there exists $b_t(u)\in(\frac{1}{2},1]$ such that $p_t(u, dr)=\delta_{b_t(u)}(dr)$ (because $\mathcal{H}$ is one-to-one on $(\frac{1}{2},1]$). \end{itemize}
In the second case, since under $\bar{\mathcal{P}}^*$, for any smooth function $\xi$ on $\bb T$, we have $\langle m_t, \xi\rangle=\ang{ \pi_t, \xi\cdot r}$ (recall \eqref{eq:estimipp} and pass to the limit), one finally obtains that almost everywhere in 
$ [0,T]\times \bb T$, we must further have $b_t(u)=\rho_t(u)$, which proves Proposition \ref{prop:YM}.
\end{proof}

We now conclude with the proof of the hydrodynamic result, namely \eqref{eq:limith} stated in  Theorem \ref{thm:hydro}. 

\begin{proof}[Proof of Theorem \ref{thm:hydro}]
Define the discrete laplacian $\Delta^N$, acting on functions $\varphi:\bb{T}\to \R$, as
\begin{equation}
\label{eq:DefDeltaN}
\Delta^N\varphi(u)=N^2\big(\varphi(u+\tfrac 1N)+\varphi(u-\tfrac 1N\big)-2\varphi(u)\big).
\end{equation} 
Recall from \eqref{eq:DefmN} the definition of the empirical measure $m^N$. We first write by Dynkin's formula, for any $\varphi\in C^{1,2}([0,T] \times \bb T)$
\begin{align*}\langle m^N_T,\varphi_T\rangle-\langle m^N_0,\varphi_0\rangle & - \int_0^T \langle m_t^N,\partial_t \varphi_t \rangle dt \\ & - \int_0^T\frac{1}{N}\sum_{x\in\bb T_N}\Delta^N \varphi_t\big(\tfrac x N\big) \tau_x h(\eta(t))\; dt=\mathcal{M}_T^{\varphi,N},\end{align*}
where $\mathcal{M}_T^{\varphi,N}$ is a martingale whose quadratic variation can be written explicitly (see e.g.~\cite[Appendix 1.5]{KL}) as 
\begin{align*}
\big[\mathcal{M}^{\varphi,N}\big]_t&=N^2\int_0^t \Big(\gene_N\pa{\big\langle m^N_s,\varphi_s\big\rangle^2}-2\big\langle m^N_s,\varphi_s\big\rangle\gene_N\big\langle m^N_s,\varphi_s\big\rangle \Big) ds\\
&=\frac{1}{2}\int_0^t\sum_{\substack{x\in \bb T_N \\ |y-x|=1}}\Big(\varphi_s\big(\tfrac x N\big)-\varphi_s\big(\tfrac y N\big)\Big)^2 c_{x,y}(\eta(s))ds.
\end{align*}
Since the function $\varphi$ is smooth, $[\mathcal{M}^{\varphi,N}]_t\leq tC(\varphi)/N$ and vanishes as $N\to\infty$. 
Using this, the local ergodicity proved in Proposition \ref{prop:localerg}, and the fact that replacing $\tau_xh$ by $\frac{1}{2\ell+1}\sum_{y\in B_\ell}\tau_{x+y}h$ in the integral leads to a term which is bounded by  $C\ell/N$ (where $C>0$ is a constant), we obtain
\begin{multline}
\label{eq:Dynkin}
\lim_{\ell\to\infty} \underset{N\to\infty}{\limsup\;}\E_{\mu^N}\Bigg[\bigg|\langle m^N_T,\varphi_T\rangle-\langle m^N_0,\varphi_0\rangle  - \int_0^T \langle m_t^N,\partial_t \varphi_t \rangle dt \\* -\int_0^T\frac{1}{N}\sum_{x\in\bb T_N}\Delta^N \varphi_t\big(\tfrac x N \big) \mathcal{H}(\rho_{t}^\ell(x)) dt\bigg|\Bigg]=0.
\end{multline}
Theorem \ref{thm:hydro} is now a consequence of Proposition \ref{prop:YM} above. Indeed, the expectation in the left hand side of \eqref{eq:Dynkin} rewrites as
\[
\bar{\mathcal{P}}_{N,\ell}\Bigg(\bigg|\langle m^N_T,\varphi_T\rangle-\langle m^N_0,\varphi_0\rangle - \int_0^T \langle m_t^N,\partial_t \varphi_t \rangle dt  -\int_0^T\ang{ \pi^{N,\ell}_t \,,\,\Delta^N \varphi_t \cdot \mathcal{H}} dt\bigg|\Bigg),
\]
where the short notation $\Delta^N\varphi \cdot  \mathcal{H} $ stands for $(u, r) \mapsto \mathcal{H}(r) \Delta^N \varphi(u)$ (recall Definition \ref{def:Youngmeasure}). In particular, as $N\to \infty$ then $\ell\to\infty$, we obtain according to Proposition \ref{prop:YM} that for any limit point $\bar{\mathcal{P}}^*$ of $ \bar{\mathcal{P}}_{N,\ell}$,
\begin{equation*}
\bar{\mathcal{P}}^*\pa{\bigg|\big\langle \rho_T,\varphi_T\big\rangle-\big\langle \rho_0,\varphi_0\big\rangle -\int_0^T \big\langle \rho_t, \partial_t\varphi_t\big\rangle dt-\int_0^T\big \langle  \mathcal{H}(\rho_t), \partial_u^2 \varphi\big \rangle dt\bigg|}=0,
\end{equation*}
which yields as wanted that $\bar{\mathcal{P}}^*$ is concentrated on 
trajectories $m_t(du)=\rho_t(u)du$  such that $\rho$ is a weak solution to \eqref{eq:PDEstrong}, in the sense of Definition \ref{Def:PDEweak}. 
\end{proof}

The remainder of the section is dedicated to proving Lemma \ref{lem:eqpi}. For that purpose, we need to state and demonstrate two important results: first we investigate the grand canonical measures of the process and we prove an \emph{ergodic decomposition} of any infinite volume stationary measure \emph{\`a la} De Finetti (Section \ref{sec:invariant-measures}, Lemma \ref{lem:definetti}); and second, we obtain a local law of large numbers analogous to the well-known \emph{one-block estimate} (Section \ref{sec:ergod}, Proposition \ref{prop:localerg}). The end of the proof is given in the last Section \ref{sec:prooflemma}.

\subsection{Canonical and grand canonical measures} \label{sec:invariant-measures}
Let us define the infinite volume generator associated with our dynamics (recall \eqref{eq:DefLN}), which acts on local functions $f:\{0,1\}^\Z\to\R$, as
\begin{equation}
\label{eq:DefLNinf}
\mathcal{L}_{\infty}f(\eta):=\sum_{x\in\Z}c_{x,x+1}(\eta)\big(f(\eta^{x,x+1})-f(\eta)\big).
\end{equation}
In this section, we investigate the measures on $\{0,1\}^\Z$ which are stationary for $\mathcal{L}_{\infty}$. 
One of the main ingredients needed to apply the same arguments as in Funaki's proof \cite{Funaki} is to prove that any  stationary measure for the generator $\mc L_\infty$, 
once restricted to the \emph{active} phase $\{\rho > \rho_\star=\frac12\}$, admits a decomposition along spatially ergodic measures. 

Let us  first introduce the \emph{grand canonical measures} $\pi_\rho$ for the facilitated exclusion process,  which have been studied  in detail in \cite{BESS}:

\begin{definition}[Grand canonical measures] \label{def:gcm}~
\begin{itemize} 
\item For any $\rho\in (\frac12,1)$, and any local configuration $\sigma=(\sigma_0, \dots, \sigma_\ell)$ on $\Lambda_\ell$, we define 
\begin{equation}
\label{eq:defGCM}
\pi_\rho\pa{\eta_{\mid \Lambda_\ell}=\sigma}={\bf 1}_{\{\sigma\in \mathscr{E}_{\Lambda_\ell}\}}(1-\rho)\pa{\tfrac{1-\rho}{\rho}}^{\ell-p}\pa{\tfrac{2\rho-1}{\rho}}^{2p-\ell-\sigma_0-\sigma_\ell},
\end{equation}
where $p=p(\sigma):=\sum_{y\in \Lambda_\ell}\sigma_y$ is the number of particles in $\sigma$, and  $\mathscr{E}_{\Lambda_\ell}$ 
was defined in \eqref{eq:ergset} as the set of local ergodic configurations.
\item For any $\rho \in [0,\frac12]$, we define \begin{equation}
\pi_\rho=\frac12\delta_{\circ\bullet}+\frac12\delta_{\bullet\circ}\; , \label{eq:defGCM2}
\end{equation}
where $\circ\bullet$ (resp. $\bullet\circ$) is the configuration in which there is a particle at $x$ iff $x$ is odd (resp. even), and $\delta_\eta$ is the Dirac measure concentrated on the configuration $\eta$. 
\item For $\rho =1$, let $\pi_1=\delta_{\bf 1}$, where $\bf 1$ denotes the configuration identically equal to $1$.
\end{itemize}
\end{definition}
We know from \cite[Section 6]{BESS} that the measures $\pi_\rho$ are invariant for the generator $\mc L_\infty$. Here we prove important additional properties of theses measures. 
The main result of this section is the following:

\begin{lemma}[Ergodic decomposition of stationary measures]
\label{lem:definetti}
Let $\bar \mu$ be a translation invariant, infinite volume, measure on  $\{0,1\}^\Z$, which is stationary for $\mathcal{L}_{\infty}$, i.e.~such that for any local function $f$, 
$\bar \mu(\mathcal{L}_{\infty}f)=0$. 

Then, there exist $\lambda\in [0,1]$, a probability measure $\mu_{\mc F}$ with support included 
in $\mc F_\Z$ (the set of frozen configurations, cf \eqref{eq:froset}), and a probability measure $\varpi(d\rho)$ on $[\frac12,1]$, such that, 
\begin{equation}\label{eq:decomperg}\bar \mu(\cdot)=\lambda\mu_{\mc F}(\cdot)+(1-\lambda)\int_{[\frac12,1]}\varpi(d\rho)\pi_\rho(\cdot).\end{equation}
\end{lemma}

\begin{proof}[Proof of Lemma \ref{lem:definetti}] We first discard the degenerate case, where the translation invariant measure $\bar \mu$ satisfies $\bar\mu(\eta_0=1)=1$: in this case, by translation invariance, $\bar\mu=\pi_1=\delta_{\bf1}$. Then the result is trivially true.

Fix now a translation invariant measure $\bar \mu$ on  $\{0,1\}^\Z$ which is  stationary w.r.t.~the generator $\mathcal{L}_{\infty}$, 
and such that $\bar\mu(\eta_0=1)<1$ (i.e.~$\bar\mu(\eta_0=0)>0$). 
Recall from \eqref{eq:ergset} and \eqref{eq:froset} the definition of the sets of infinite ergodic and frozen configurations $\mathscr{E}_\Z$ and $\mathscr{F}_\Z$. 
We first claim that, since $\bar\mu$ is stationary,  we must have 
\begin{equation}
\label{eq:muEF}
\bar\mu\pa{\{0,1\}^\Z\setminus (\mathscr{E}_\Z\cup \mathscr{F}_\Z)}=0,
\end{equation}
i.e.~$\bar\mu$ charges configurations which are either completely ergodic, or completely frozen.
To expose the argument as clearly as possible, let us indicate the occupied sites by $\bullet$, and the empty sites by $\circ$, and any local configuration $\eta$ by a finite sequence of $\bullet$ and $\circ$. Since $\bar\mu$ is translation invariant, there will be no need to specify the support of the configurations in the following argument.  We further use the notation
\[[ \bullet \circ]^k:=\underbrace{\bullet \circ \cdots \bullet \circ}_{2k \text{ sites }}\quad \mbox{ and } [ \circ  \bullet ]^k:=\underbrace{\circ \bullet  \cdots  \circ  \bullet }_{2k \text{ sites }}.\]
We are going to show that for any $k\geq0$,
\begin{equation}
\label{eq:statmix}
\bar\mu(\circ\circ[\bullet \circ]^k \bullet \bullet )=\bar\mu(\bullet \bullet[\circ \bullet ]^k \circ\circ )=0, 
\end{equation}
the box where the configuration is observed being arbitrary, but fixed. Since any configuration which is not in $\mathscr{E}_\Z$ nor in  
$\mathscr{F}_\Z$ must contain either $\circ\circ[\bullet \circ]^k \bullet \bullet\;$ or $\;\bullet \bullet[ \circ \bullet]^k \circ\circ$ for some $k$, 
this will prove \eqref{eq:muEF}. For $k=0$, we write by definition and using the translation invariance of $\bar\mu$
\[\bar\mu\big(\mathcal{L}_\infty{\bf 1}_{\{\circ\circ\}}\big)=-\bar\mu (\circ\circ\bullet\; \bullet )-\bar\mu(\bullet\bullet\circ\;\circ)=0,\]
since $\bar\mu$ is stationary. Therefore both probabilities on the right hand side, which are non-negative, must be equal to $0$. This proves \eqref{eq:statmix} for $k=0$. 
Assume now that \eqref{eq:statmix} holds for any $\ell< k$, then write (more explanations on the following identity are given right below):
\begin{align}0 & =\bar\mu\Big(\mathcal{L}_\infty{\bf 1}_{\{ \circ\circ \;[\bullet \circ]^{k-1} \bullet \bullet  \}}\Big) \vphantom{\int} \notag \\
& =-\bar\mu \big(\bullet \bullet \circ\circ\; [\bullet \circ]^{k-1} \bullet \bullet\big)-\bar\mu\big(\circ\circ\; [\bullet \circ]^{k-1} \bullet \bullet\big) \vphantom{\int} \label{eq:firstt}\\
& \quad -\bar\mu\big( \circ\circ\; [\bullet \circ]^{k-1} \bullet \bullet \circ\big)+\bar\mu\big(\circ\circ\; [\bullet \circ]^k \bullet \bullet\big) \vphantom{\int} \label{eq:secondt}\\
& \quad +\sum_{\ell=0}^{k-3}\bar\mu( \circ\circ\; [\bullet \circ]^{\ell} \bullet \bullet \circ\circ\; [\bullet \circ]^{k-\ell-3} \bullet \bullet) \label{eq:thirdt}\\
&  \quad +\sum_{\ell=0}^{k-2}\bar\mu( \circ\circ\; [\bullet \circ]^{\ell}\circ \bullet\; [\bullet \circ]^{k-\ell-2} \bullet \bullet).\vphantom{\int} \label{eq:fourtht}
\end{align}
Let us comment briefly on the identity above: the only terms that give a non-zero contribution to $\mathcal{L}_\infty{\bf 1}_{\{\circ \circ \; [\bullet \circ]^{k-1} \bullet \bullet \}} $ are: \begin{enumerate} \item the jumps that can happen in the configuration $\circ \circ \; [\bullet \circ]^{k-1} \bullet \bullet$ (giving contributions with the minus sign). There are three possibilities: first, if there are two extra particles to the left, then the first pair of empty sites may be broken by a particle 
coming from the left, 
\[ {\color{red}\bullet \; \bullet} \circ \circ \; [\bullet \circ]^{k-1} \bullet \bullet \qquad \mapsto \qquad {\color{red}\bullet} \circ {\color{red} \bullet}  \circ \; [\bullet \circ]^{k-1} \bullet \bullet, \]
and this gives the first term in \eqref{eq:firstt}. Similarly, the transitions corresponding to the second term in \eqref{eq:firstt} and first term in \eqref{eq:secondt} are given respectively by
\[  \circ \circ \; [\bullet \circ]^{k-2} \bullet   \circ \; {\color{red}\bullet} \; \bullet \qquad \mapsto \qquad \circ  \circ \; [\bullet \circ]^{k-2} \bullet {\color{red}\bullet} \circ \bullet \]
and
\[ 
\circ \circ \; [\bullet \circ]^{k-1} \bullet \bullet \; {\color{red}\circ} \qquad \mapsto \qquad \circ \circ \; [\bullet \circ]^{k-1} \bullet {\color{red} \circ \; \bullet}.
\]
\item Then, there are the jumps that, starting from another configuration, lead to $\circ \circ \; [\bullet \circ]^{k-1} \bullet \bullet$ (giving the three contributions with the plus sign). The corresponding transitions are depicted below:
\[ \circ \circ \; [\bullet \circ]^{k-1} \bullet {\color{red}\circ\; \bullet}\; \bullet \qquad \mapsto \qquad  \circ \circ \; [\bullet \circ]^{k-1} \bullet {\color{red}\bullet\; \circ}\; \bullet,  \]

\[ 
\circ \circ [\bullet \circ]^\ell \bullet {\color{red}\bullet\;  \circ} \circ [\bullet \circ]^{k-\ell-3} \bullet \bullet  \qquad \mapsto \qquad \circ \circ [\bullet \circ]^\ell \bullet {\color{red}\circ\;  \bullet} \circ [\bullet \circ]^{k-\ell-3} \bullet \bullet, 
\]

\[ 
\circ\circ\; [\bullet \circ]^{\ell}\; {\color{red}\circ\; \bullet}\; [\bullet \circ]^{k-\ell-2} \bullet \bullet
\qquad \mapsto \qquad \circ\circ\; [\bullet \circ]^{\ell}\; {\color{red}\bullet\; \circ}\; [\bullet \circ]^{k-\ell-2} \bullet \bullet.
\]
\end{enumerate}

Note that all terms in \eqref{eq:thirdt} contain $\bullet\bullet\circ\circ$, and all terms in \eqref{eq:fourtht} contain $\circ \circ \bullet \bullet$ and therefore vanish. 
Since we assumed that \eqref{eq:statmix} holds for any $\ell\leq k-1$, all terms in the right hand side \eqref{eq:firstt}--\eqref{eq:secondt} vanish,
except $\bar\mu(\circ\circ[\bullet \circ]^k \bullet \bullet)$. Therefore the latter must vanish as well. 
An analogous computation for $\bar\mu\big(\mathcal{L}_\infty{\bf 1}_{\{ \bullet \bullet[ \circ \bullet]^{k-1} \circ\circ  \}}\big)$ 
proves the second identity, so that \eqref{eq:statmix} holds for any $k$.

\medskip

Now, let $\lambda=\bar\mu(\mathscr{F}_\Z) \in [0,1]$ be the total mass of frozen configurations. Note that any translation invariant measure with support included in $\mathscr{F}_\Z$ is necessarily stationary for $\mathcal{L}_{\infty}$. 
In particular, in order to prove Lemma \ref{lem:definetti}, we only need to treat the decomposition of $\bar \mu$ 
restricted to the ergodic component. Without loss of generality, we can therefore assume that $\lambda=0$, i.e. $\bar\mu(\mathscr{E}_\Z)=1$. Let us also put aside the case where $\bar \mu$ gives positive weight to ${\bf 1}$ and assume $\bar\mu({\bf 1})=0$.
Let us define the set of configurations with infinitely many zeros both right and left of the origin:
\[
\Sigma_\infty:=\Big\{\eta\in\{0,1\}^{\Z}\; ; \; \sum_{x\geq 0}(1-\eta_x)=\sum_{x\leq 0}(1-\eta_x)=\infty\Big\}.
\]
We claim that, since we assumed $\bar\mu(\eta_0=0)>0$ and $\bar\mu({\bf 1})=0$, we must have $\bar\mu(\Sigma_\infty^c)=0$. To prove this claim, fix a semi infinite configuration $\eta^+$ on $\N$ with a finite number of empty sites (see Figure \ref{fig:eta+}), denote $c=\bar\mu(\eta_{|\N}=\eta^+)$. If  $\eta^+\neq {\bf 1}$, denote $z$ its rightmost empty site, and define   the set 
\[E_k:=\Big\{ \eta\in\{0,1\}^{\Z}\; ; \; \eta_{|\{k(z+1), \dots\}}=\tau_{-k(z+1)}\eta^+\Big\}.\]
\begin{figure}[H]
\centering
\begin{tikzpicture}
\node at (1,0.5) {\color{white}x};
\draw (-1,0) -- (9,0);
\foreach \i in {-1,...,8}
{
\draw (\i+0.5,-0.1) -- (\i+0.5,0.1);
}
\node at (-0.5,-0.3) {$0$};
\node at (4.5,-0.3) {$z$};
\node at (9.2,0.4) {\emph{etc.}};
\node[circle,fill=black,inner sep=1mm] at (0.5,0.2) {};
\node[circle,fill=black,inner sep=1mm] at (3.5,0.2) {};
\node[circle,fill=black,inner sep=1mm] at (5.5,0.2) {};
\node[circle,fill=black,inner sep=1mm] at (1.5,0.2) {};
\node[circle,fill=black,inner sep=1mm] at (6.5,0.2) {};
\node[circle,fill=black,inner sep=1mm] at (7.5,0.2) {};
\node[circle,fill=black,inner sep=1mm] at (8.5,0.2) {};
\end{tikzpicture}
\caption{An example of configuration $\eta_+$ with support $\N$.}
\label{fig:eta+}
\end{figure}
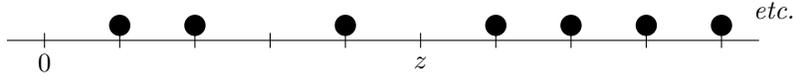

By translation invariance of $\bar\mu$, we have $\bar\mu(E_k)=c$ for any  $k\in \N$, and the sets $E_k$'s are disjoint by construction, because 
\[E_k\subset\Big\{  \eta\in\{0,1\}^{\Z}\; ; \; \sum_{x\geq k(z+1)}(1-\eta_x)>0\quad \mbox{ and }\quad \sum_{x\geq (k+1)(z+1)}(1-\eta_x)=0\Big\},\]
therefore in particular, we must have $c=0$.  Since there are countably many configurations with a finite number of empty sites to the right of the origin, and since we just proved that their probabilities vanish, making the same statement for configurations to the left of the origin yields as wanted $\bar\mu(\Sigma_\infty^c)=0$.

\medskip

To prove the ergodic decomposition \eqref{eq:decomperg}, we use a classical mapping between the facilitated exclusion process and a zero-range process, 
introduced in \cite{BM2009} and already exploited in \cite{BESS}. 
For simplicity, we define this mapping on the set 
\[\Sigma^0_\infty:=\{\eta\in \Sigma_\infty \; ; \; \eta_0=0\}\]
of configurations with an empty site at the origin. Then, given $\eta\in \Sigma_\infty^0$, and for any integer $k>0$ (resp.~$-k<0$) 
we denote $x_k(\eta)$ (resp.~$x_{-k}(\eta)$) the position of the $k$-th empty site to the right (resp.~to the left) of the origin, and let $x_0=0$. 
We then define, for any $k\in \Z$ and $\eta\in \Sigma^0_\infty$ 
\[\omega^\eta_k:=x_{k+1}(\eta)-x_k(\eta)-1.\]
In other words, $\omega^\eta\in \N^\Z$ is the zero range configuration such that the number of particles on site $k>0$ (resp.~$-k<0$) is the number of particles between the 
$k$--th and $(k\!+\!1)$--th empty site to the right (resp.~to the left) of the origin in $\eta$ (see Figure \ref{fig}).

\begin{figure}[h]
\centering
\begin{tikzpicture}
\draw (-1,0) -- (10,0);
\foreach \i in {-1,...,9}
{
\draw (\i+0.5,-0.1) -- (\i+0.5,0.1);
}
\node[circle,fill=black,inner sep=1mm] at (0.5,0.2) {};
\node[circle,fill=black,inner sep=1mm] at (3.5,0.2) {};
\node[circle,fill=black,inner sep=1mm] at (5.5,0.2) {};
\node[circle,fill=black,inner sep=1mm] at (1.5,0.2) {};
\node[circle,fill=black,inner sep=1mm] at (6.5,0.2) {};
\node[circle,fill=black,inner sep=1mm] at (7.5,0.2) {};
\node at (4.5,-0.3) {$0$};
\node at (11,0.1) {$\eta$};

\draw (2,-2) -- (6,-2);
\foreach \i in {2,...,5}
{
\draw (\i+0.5,-2.1) -- (\i+0.5,-1.9);
}
\node at (4.5,-2.3) {$0$};
\node[circle,fill=black,inner sep=1mm] at (4.5,-1.8) {};
\node[circle,fill=black,inner sep=1mm] at (4.5,-1) {};
\node[circle,fill=black,inner sep=1mm] at (4.5,-1.4) {};
\node[circle,fill=black,inner sep=1mm] at (3.5,-1.8) {};
\node[circle,fill=black,inner sep=1mm] at (2.5,-1.8) {};
\node[circle,fill=black,inner sep=1mm] at (2.5,-1.4) {};
\node at (7,-1.9) {$\omega^\eta$};
\end{tikzpicture}
\caption{An exclusion configuration $\eta$ with an empty site at the origin and its corresponding zero-range configuration $\omega^\eta$.} \label{fig}
\end{figure}
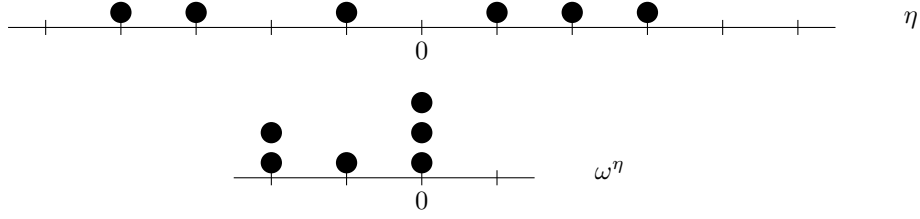

We first note that for any ergodic configuration $\eta\in\Sigma_\infty^0\cap \mathscr{E}_\Z$, we must have $\omega^\eta\in\N_*^\Z$. Let us denote by $\Pi$  the one-to-one mapping 
\begin{equation}
\label{eq:DefPi}
\begin{array}{cccc}
\Pi:&\Sigma^0_\infty \cap \mathscr{E}_{\Z}&\to&\N_*^\Z \\
&\eta&\mapsto&\omega^\eta \; . 
\end{array}
\end{equation}
 Recalling that we  assume $\bar\mu(\mathscr{E}_\Z)=1$ and $\bar\mu(\eta_0=0)>0$, we now define a measure $\bar \nu$ on the set  $\N_*^\Z$ of infinite zero-range configurations, 
 \begin{equation}
\label{eq:nubar}
\bar \nu( F):=\bar\mu(\omega^\eta\in F \mid \eta_0=0), \qquad F \subset \N_*^\Z\quad\text{measurable}.
\end{equation}
In particular, for any $E \subset \mathscr{E}_\Z$ measurable, 
\begin{equation}
\bar \mu( \eta \in E \mid \eta_0=0 ) = \bar \nu\big(\Pi(E\cap \Sigma^0_\infty)\big).
\label{eq:idmeas}
\end{equation}
Define the infinite volume \emph{zero-range generator} $\mathcal{L}^{\mathrm{ZR}}_{\infty}$ which acts on local functions  $f:\N^\Z \to \R$
\begin{equation}
\label{eq:ZRgen}
\mathcal{L}^{\mathrm{ZR}}_{\infty}f(\omega):=\sum_{x\in\Z}\sum_{\delta=\pm1}{\bf 1}_{\{\omega_x\geq 2\}}\big(f(\omega^{x,x+\delta})-f(\omega)\big),
\end{equation}
with $\omega^{x,x+\delta}$ representing the zero-range configuration where one particle in $\omega$ was moved from $x$ to $x+\delta$, i.e. 
 \[
 \omega^{x,x+\delta}_y=
 \begin{cases}
 \omega_x-1 &\mbox{ if }y=x\\
 \omega_{x+\delta}+1 &\mbox{ if }y=x+\delta\\ 
 \omega_y &\mbox{ else}
 \end{cases}.
 \]
 One easily checks that for any $\alpha \geq 1$ the geometric product homogeneous measures $\nu_\alpha$ with marginals
 \begin{equation}
\label{eq:defnualpha}
 \nu_\alpha(\omega_0=p)={\bf 1}_{\{p\in \N, p\geq 1\}}\frac{1}{\alpha}\Big(1-\frac{1}{\alpha}\Big)^{p-1}  
 \end{equation}
are reversible for $\mathcal{L}^{\mathrm{ZR}}_{\infty}$, and that $\alpha=\E_{\nu_\alpha}(\omega_0)$ then represents the average particle density per site. 
We claim the following.

\begin{lemma}
\label{lem:lemZRstat}
The measure $\bar \nu$ defined by \eqref{eq:nubar} on $\N_*^{\Z}$ is translation invariant, and stationary w.r.t.~the zero-range generator $\mathcal{L}^{\mathrm{ZR}}_{\infty}$.
In particular from \cite{Andjel}, there exists  a probability measure $\varpi_{\mathrm{ZR}}$ on $[1,+\infty)$, such that 
\begin{equation}
\label{eq:definettiZR}
\bar \nu(\cdot)=\int_{[1,+\infty)}\varpi_{\mathrm{ZR}}(d\alpha)\nu_\alpha(\cdot)\;.
\end{equation}
\end{lemma}

Before proving this result, we show that Lemma \ref{lem:definetti} follows. 
For any event $E\subset\Sigma^0_\infty \cap \mathscr{E}_\Z$, we can now write according to Lemma \ref{lem:lemZRstat} and using \eqref{eq:idmeas},
\[\bar\mu(E \mid \eta_0=0)=\int_{[1,+\infty)}\varpi_{\mathrm{ZR}}(d\alpha)\nu_\alpha\big(\Pi(E)\big).\]
for some measure $\varpi_{\mathrm{ZR}}(d\alpha)$ on $[1,+\infty)$.  Define $\mathcal{G}(\alpha)=\alpha/(1+\alpha)$, which is an increasing bijection from $[1,+\infty)$ to $[\frac12,1)$. 
Given the explicit expressions \eqref{eq:defnualpha} and \eqref{eq:defGCM} for $\nu_\alpha$ and $\pi_\rho$, one easily checks that
\[\pi_{\mathcal{G}(\alpha)}(E\mid\eta_0=0)=\nu_\alpha\big(\Pi(E)\big).\] We now define the measure $\widetilde{\varpi}$ on $[\frac12,1)$ as the pushforward of $\varpi_{\mathrm{ZR}}$ by $\mathcal{G}$ 
\[\widetilde{\varpi}=\varphi_{\mathrm{ZR}}\circ \mathcal{G}^{-1},\]
which yields after a change of variables
\[\bar\mu(E \mid \eta_0=0)=\int_{[\frac12,1)}\widetilde{\varpi}(d\rho)\pi_\rho(E\mid \eta_0=0).\]
Finally, let  
\[\varpi(d\rho)=\frac{\bar \mu(\eta_0=0)}{\pi_\rho(\eta_0=0)}\widetilde{\varpi} (d\rho),\] 
and we obtain
\[\bar\mu(E \cap \{ \eta_0=0\})=\int_{[\frac12,1)}\varpi(d\rho)\pi_\rho(E\cap\{ \eta_0=0\}).\]
Since by assumption $\bar\mu(\Sigma_\infty)=1$, for any event $E \subset \mathcal{E}_\Z$ we can write $\bar\mu(E)=\sum_{k=0}^\infty\bar\mu(E\cap\{\eta_0=\ldots=\eta_{k-1}=1,\eta_k=0\})$ and similarly with $\pi_\rho$ for any $\rho\in[\frac12,1)$. Using the translation invariance of both $\bar\mu$ and $\pi_\rho$ in these identities, we obtain
$\bar\mu(E)=\int_{[\frac12,1)}\varpi(d\rho)\pi_\rho(E)$ as wanted, which concludes the proof of Lemma  \ref{lem:definetti}.
\end{proof}

We now prove Lemma \ref{lem:lemZRstat}.
\begin{proof}[Proof of Lemma \ref{lem:lemZRstat}]
It was proved by Andjel in \cite[Theorem 1.9]{Andjel} that any translation invariant, stationary measure for the zero-range process with constant jump rate can be decomposed as 
\begin{equation}\label{e:andjel}
\int_{[1,+\infty)}\varpi(d\alpha)\tilde\nu_\alpha(\cdot), 
\end{equation}where $\tilde\nu_\alpha$ is the product measure on $\N^\Z$ with marginals $\tilde\nu_\alpha(\omega_0=p)=\frac{1}{\alpha}(1-\frac{1}{\alpha})^p$. We can couple this zero-range process (which lives on $\N^\Z$) with the process generated by  $\mathcal{L}^{\mathrm{ZR}}_{\infty}$ and restricted to configurations in $\N_*^\Z$ by simply adding a particle at every site. The decomposition \eqref{e:andjel} then yields \eqref{eq:definettiZR} for any translation invariant measure on $\N_*^\Z$ which is stationary w.r.t.~the zero-range generator $\mathcal{L}^{\mathrm{ZR}}_{\infty}$. We therefore only need to prove the first two claims, namely the translation invariance and stationary properties.

\smallskip

\emph{(i) Translation invariance.}
Recall that we denote $\Lambda_\ell=\{0,\ldots,\ell\}$. To prove that $\bar\nu$ is translation invariant, fix $\ell\geq 0$ and consider a local zero-range configuration $\sigma^\ell=(\sigma_0,\dots, \sigma_\ell)\in \N_*^{\ell+1}$. We are going to prove that for any $x\in \Z$,
\[\bar\nu(\omega_{\mid{x+\Lambda_\ell}}=\sigma^\ell)=\bar\nu(\omega_{\mid{\Lambda_\ell}}=\sigma^\ell),
\]
where we shortened $x+E=\{x+y, y\in E\}$. To prove it, first note that by definition
\[\bar\nu(\omega_{\mid{x+\Lambda_\ell}}=\sigma^\ell)=\bar \mu (\eta_0=0)^{-1}\; \bar\mu\pa{\omega^\eta_{\mid x+\Lambda_\ell} =\sigma^\ell\;\mbox{ and }\;\eta_0=0}.\]
Assume that $x\neq 0$ (otherwise the statement is trivial), and first consider the case where $0\in x+\Lambda_\ell$, so that in particular we must have $x\leq -1$. 
In this case, denote $k:=k(\sigma^\ell)={\sum_{y=0}^{-x-1}(1+\sigma_y)}$, and note the following: for any configuration $\eta  \in \Sigma_\infty^0$ (with an empty site at the origin), 
\[ \text{if \quad } \omega^\eta_{\mid x+ \Lambda_{\ell}} =\sigma^\ell \text{\quad then \quad } \eta_{-k}=0,\] 
since  $-k$ 
is the position of the (exclusion) $x$--th empty site in $\eta$, corresponding to the (zero-range) site $x$ in $\omega^\eta$. 
In particular, we can write
\begin{align*}
\bar\mu\Big(\omega^\eta_{\mid x+\Lambda_{\ell}} =\sigma^\ell&\;\mbox{ and }\;\eta_0=0\Big)\\
&= \bar\mu\pa{\omega^\eta_{\mid x+ \Lambda_{\ell}} =\sigma^\ell\;\mbox{ and }\;\eta_{-k}=0\;\mbox{ and }\;\eta_0=0}\\
&= \bar\mu\pa{\omega^{\tau_{k}\eta}_{\mid x+\Lambda_{\ell}} =\sigma^\ell\;\mbox{ and }\;\eta_{0}=0\;\mbox{ and }\;\eta_{k}=0}\\
&=\bar\mu\bigg(\eta_0=0 \text{ and } \forall\; y\in \bigg\{1,\ldots,\sum_{z=0}^\ell(1+\sigma_z)\bigg\}, \\ & \qquad \qquad \qquad \qquad \qquad \qquad \eta_y=0 \; \Leftrightarrow \;   \exists\; i\in\Lambda_\ell,\ y=\sum_{z=0}^i(1+\sigma_z) \bigg)\\
&= \bar\mu\pa{\omega^\eta_{\mid\Lambda_\ell} =\sigma^\ell\;\mbox{ and }\;\eta_{0}=0}.
\end{align*}
To establish the second identity, we use the translation invariance of $\bar \mu$ and make the change of variable $\eta=\tau_k\eta'$. To establish the third  (resp.~fourth) identity, we simply lie down what the configuration $\omega^{\tau_{k}\eta}_{\mid \Lambda_{\ell+x}}$ (resp.~$\omega^\eta_{\mid\Lambda_\ell} $) means for $\eta$. This proves the translation invariance, assuming that $0\in x+ \Lambda_{\ell}$. 

We now consider the case $x>0$. The case $x<-\ell$ being strictly analogous, we will not detail it here. For two zero-range configurations $\sigma^k$ and $\sigma^\ell$ defined respectively on $\Lambda_k$ and  $\Lambda_\ell$, denote 
$ \sigma^k \sigma^\ell$ the \emph{concatenated} configuration  (defined on $\Lambda_{k+\ell+1}$)
\[\sigma^k \sigma^\ell=(\sigma^k_0,\dots,\sigma^k_k,\sigma^\ell_0, \dots, \sigma^\ell_\ell).\]
Fix $x>0$, and write  
\[\bar\mu\pa{\omega^\eta_{\mid x+\Lambda_{\ell}} =\sigma^\ell\;\mbox{ and }\;\eta_0=0}=\sum_{\sigma^{x-1}\in (\N_*)^x}\bar\mu\pa{\omega^\eta_{\mid \Lambda_{x+\ell}} = \sigma^{x-1}\sigma^\ell\;\mbox{ and }\;\eta_0=0}.\]
As in the first case, define $k=\sum_{y=0}^{x-1}(1+\sigma^{x-1}_y)$, and rewrite for any fixed $\sigma^{x-1} \in (\N_*)^x$:
\begin{align*}
\bar\mu\Big(\omega^\eta_{\mid \Lambda_{x+\ell}} = \sigma^{x-1}\sigma^\ell\;&\mbox{ and }\;\eta_0=0\Big)\\
= &\; \bar\mu\pa{\omega^\eta_{\mid \Lambda_{x+\ell}} = \sigma^{x-1}\sigma^\ell\;\mbox{ and }\;\eta_0=0 \;\mbox{ and }\;\eta_{k}=0}\\
= &\;\bar\mu\pa{\omega^{\tau_{-k}\eta}_{\mid \Lambda_{x+\ell}} = \sigma^{x-1}\sigma^\ell\;\mbox{ and }\;\eta_{-k}=0\;\mbox{ and }\;\eta_{0}=0}\\
= &\;\bar\mu\pa{\omega^\eta_{\mid -x+ \Lambda_{\ell}} = \sigma^{x-1}\sigma^\ell\;\mbox{ and }\;\eta_{0}=0},
\end{align*}
where the third identity is derived as above.
Summing over all  $\sigma^{x-1}\in (\N_*)^x$, one finally obtains as wanted
\[\bar\mu\Big(\omega^\eta_{\mid x+ \Lambda_{\ell}} =\sigma^\ell\;\mbox{ and }\;\eta_0=0\Big)=\bar\mu\pa{\omega^\eta_{\mid\Lambda_{\ell}} =\sigma^\ell\;\mbox{ and }\;\eta_0=0}.\]
 This proves that the measure $\bar\nu$ is translation invariant.

\smallskip

\emph{(ii) Stationarity.}
We now prove that $\bar\nu$ is stationary for the zero-range generator $\mathcal{L}_\infty^\mathrm{ZR}$ as well. To do so, it is sufficient to prove that for any $\ell\geq 0$ and any local configuration $\sigma^\ell \in\N_*^{\ell+1}$,
\begin{equation}
\label{eq:statZR}
\bar\nu\pa{\mathcal L^{\mathrm{ZR}}_\infty \mathbf{1}_{\{\omega_{\mid \Lambda_\ell}=\sigma^\ell\}}}=0.
\end{equation}
Proving this identity is a matter of elementary, though lengthy, computations; to facilitate reading, we will only write it for $\ell=1$ and let  the reader  check 
that this identity also holds for $\ell=0$ and $\ell>1$. In order not to burden the notations, we now omit the exponent $\ell=1$ in the configuration $\sigma$. 
Fix a pair of integers $(\sigma_0, \sigma_1) \in \N_*^2$, we can write
\begin{align}
\mathcal L^{\mathrm{ZR}}_\infty&\mathbf{1}_{\{\omega_0=\sigma_0, \;\omega_1=\sigma_1\}}=\mathbf{1}_{\{\omega_{-1}\geq 2\}}\pa{\mathbf{1}_{\{\omega_0=\sigma_0-1, \;\omega_1=\sigma_1\}}-\mathbf{1}_{\{\omega_0=\sigma_0, \;\omega_1=\sigma_1\}}} \notag \\
&+\mathbf{1}_{\{\omega_0\geq 2\}}\pa{\mathbf{1}_{\{\omega_0=\sigma_0+1, \;\omega_1=\sigma_1\}}+\mathbf{1}_{\{\omega_0=\sigma_0+1, \;\omega_1=\sigma_1-1\}}-2\mathbf{1}_{\{\omega_0=\sigma_0, \;\omega_1=\sigma_1\}}} \notag \\
&+\mathbf{1}_{\{\omega_1\geq 2\}}\pa{\mathbf{1}_{\{\omega_0=\sigma_0, \;\omega_1=\sigma_1+1\}}+\mathbf{1}_{\{\omega_0=\sigma_0-1, \;\omega_1=\sigma_1+1\}}-2\mathbf{1}_{\{\omega_0=\sigma_0, \;\omega_1=\sigma_1\}}} \notag \\
&+\mathbf{1}_{\{\omega_2\geq 2\}}\pa{\mathbf{1}_{\{\omega_0=\sigma_0, \;\omega_1=\sigma_1-1\}}-\mathbf{1}_{\{\omega_0=\sigma_0, \;\omega_1=\sigma_1\}}}. \label{eq:id1}
\end{align}
Analogously, for $\eta\in\mathcal{E}_\mathbb{Z}$,
\begin{align}
\mathcal L _\infty &\mathbf{1}_{\{\eta_0=0, \; \omega^\eta_0=\sigma_0, \;\omega^\eta_1=\sigma_1\}}=\mathbf{1}_{\{\eta_0=0\}}\bigg[-\mathbf{1}_{\{\omega^\eta_{-1}\geq2,\; \omega^\eta_0=\sigma_0, \;\omega^\eta_1=\sigma_1\}} \notag \\
&+\mathbf{1}_{\{\omega^\eta_0\geq 2\}}\pa{\mathbf{1}_{\{\omega^\eta_0=\sigma_0+1, \;\omega^\eta_1=\sigma_1-1\}}-2\mathbf{1}_{\{\omega^\eta_0=\sigma_0, \;\omega^\eta_1=\sigma_1\}}} \notag \\
&+\mathbf{1}_{\{\omega^\eta_1\geq 2\}}\pa{\mathbf{1}_{\{\omega^\eta_0=\sigma_0, \;\omega^\eta_1=\sigma_1+1\}}+\mathbf{1}_{\{\omega^\eta_0=\sigma_0-1, \;\omega^\eta_1=\sigma_1+1\}}-2\mathbf{1}_{\{\omega^\eta_0=\sigma_0, \;\omega^\eta_1=\sigma_1\}}} \notag \\
&+\mathbf{1}_{\{\omega^\eta_2\geq 2\}}\pa{\mathbf{1}_{\{\omega^\eta_0=\sigma_0, \;\omega^\eta_1=\sigma_1-1\}}-\mathbf{1}_{\{\omega^\eta_0=\sigma_0, \;\omega^\eta_1=\sigma_1\}}}\bigg] \notag \\
&+\mathbf{1}_{\{\eta_1=0, \; \omega^{\tau_{1}\eta}_{-1}\geq 2,\;\omega^{\tau_{1}\eta}_0=\sigma_0-1, \;\omega^{\tau_{1}\eta}_1=\sigma_1\}}+\mathbf{1}_{\{\eta_{-1}=0, \; \omega^{\tau_{-1}\eta}_0=\sigma_0+1, \;\omega^{\tau_{-1}\eta}_1=\sigma_1\}}. \label{eq:id2}
\end{align}
The last two terms rewrite as
\begin{equation} \label{eq:sum} \tau_{1}\mathbf{1}_{\{\eta_0=0, \; \omega^\eta_{-1}\geq 2,\;\omega^\eta_0=\sigma_0-1, \;\omega^\eta_1=\sigma_1\}}+\tau_{-1}\mathbf{1}_{\{\eta_0=0, \;\omega^\eta_0=\sigma_0+1, \;\omega^\eta_1=\sigma_1\}}.\end{equation}
In particular, since $\bar \mu$ was assumed to be translation invariant, the expectation of \eqref{eq:sum} with respect to $\bar \mu$ is also equal to the expectation of 
\[
\mathbf{1}_{\{\eta_0=0, \; \omega^\eta_{-1}\geq 2,\;\omega^\eta_0=\sigma_0-1, \;\omega^\eta_1=\sigma_1\}}+\mathbf{1}_{\{\eta_0=0, \;\omega^\eta_0=\sigma_0+1, \;\omega^\eta_1=\sigma_1\}}\; . 
\]
Using this, and taking the $\bar \mu$--expectation in both identities \eqref{eq:id1} and \eqref{eq:id2} we obtain as wanted:
\begin{align*}
\bar \mu\pa{\mathcal L _\infty \mathbf{1}_{\{\eta_0=0, \; \omega^\eta_0=\sigma_0, \;\omega^\eta_1=\sigma_1\}}}&=\bar \mu\Big(\mathbf{1}_{\{\eta_0=0\}}\big(\mathcal L^{\mathrm {ZR}} _\infty \mathbf{1}_{\{\omega_0=\sigma_0, \;\omega_1=\sigma_1\}}\big)(\omega^\eta)\Big)\\
&=\bar\mu(\eta_0=0)\; \bar\nu\pa{\mathcal L^{\mathrm{ZR}} _\infty \mathbf{1}_{\{\omega_0=\sigma_0, \;\omega_1=\sigma_1\}}}. \vphantom{\Big)}
\end{align*}
Since $\bar\mu$ is stationary for $\mathcal L _\infty$, the left hand side above vanishes, therefore so does the right hand side. 
As seen previously, we assumed that $\bar\mu(\eta_0=0)>0$, which proves 
$\bar \nu\pa{\mathcal L^{\mathrm{ZR}} _\infty \mathbf{1}_{\{\omega_0=\sigma_0, \;\omega_1=\sigma_1\}}}=0$, and in turn \eqref{eq:statZR}. This concludes the proof of Lemma \ref{lem:lemZRstat}.
\end{proof}

\subsection{Local ergodicity} \label{sec:ergod}
We now turn to the second brick necessary to prove Lemma \ref{lem:eqpi}.  Let $\mu_t^N$ denote the distribution on $\Sigma_N$ of $\eta(t)$, and define the space-time average
\begin{equation}
\label{eq:Defmubar}
\bar \mu_T^N:=\frac{1}{TN}\int_0^{T}\sum_{x\in \bb T_N}\mu_t^N\circ\tau_x^{-1} dt. 
\end{equation}
Recall definition \eqref{eq:Defrhoell} of $\rho_x^\ell$ and recall that we set $\rho^\ell=\rho_0^\ell$. We are now ready to state the following result, which, although proved differently, is analogous to the so-called \emph{one-block estimate} (see \cite[Section 5.4]{KL}).
\begin{proposition}[Local law of large numbers in the supercritical phase]
\label{prop:localerg}
Recall that $\pi_\rho$ was defined for any $\rho \in[0,1]$  in Definition \ref{def:gcm} and that $B_\ell=\{-\ell,\dots,\ell\}$. Recall also from \eqref{eq:defh} the definition of the function $h$. Then we have
\begin{equation}
\lim_{\ell\to\infty}\limsup_{N\to\infty}\bar \mu_T^N\bigg(\bigg|\frac{1}{2\ell+1}\sum_{y\in B_\ell}\tau_y h-\pi_{\rho^\ell}(h)\bigg|\bigg)=0. \label{eq:oneb}
\end{equation}
\end{proposition}

\begin{proof}[Proof of Proposition \ref{prop:localerg}]
Thanks to the work of the previous section, and the correlation decay for the grand canonical measures $\pi_\rho$ proved in \cite[Section 6.3]{BESS}, we are able to prove this proposition using the same arguments as in \cite{Funaki}. Therefore, we simply sketch out the proof here. First, note that for any $\rho\in[0,1]$, 
\begin{equation}\label{eq:pirhoh}\pi_\rho(h)=\frac{2\rho-1}{\rho} \mathbf{1}_{\{\rho > \frac12\}}=\mathcal{H}(\rho). \end{equation}
Recall the infinite volume generator $\mc L_\infty$ introduced in \eqref{eq:DefLNinf}. By periodically extending the configurations, one can see $\bar \mu_T^N$ as a measure on the set of infinite exclusion configurations, namely $\{0,1\}^{\Z}$. For any local function $f$ defined on this set, any $x\in\T_N$, and any $t\geq 0$, we can then write 
\[\frac{d\mu_t^N(\tau_xf)}{dt}=\mu_t^N (N^2\mathcal{L}_N\tau_xf)=N^2\mu_t^N (\tau_x\mathcal{L}_\infty f)\] for any $N$ large enough (depending on the support of $f$). In particular, averaging over $x\in \T_N$ and over the segment $[0,T]$, the identity above rewrites
\[\frac{1}{N^3 T}\sum_{x\in\T_N}\pa{\mu_T^N(\tau_xf)-\mu_0^N(\tau_xf)}=\bar \mu_T^N (\mathcal{L}_\infty f).\]
Since $f$ is a local function, it is in particular bounded, therefore the left hand side above vanishes as $N\to\infty$.
As a consequence, any limit point $\bar\mu^T$ of $\bar \mu_T^N$ is stationary for the infinite volume generator $\mathcal{L}_\infty$, and to obtain \eqref{eq:oneb} it is sufficient to prove that 
\begin{equation}
\label{eq:limOBE}
\lim_{\ell\to\infty}\bar \mu^T\bigg(\bigg|\frac{1}{2\ell+1}\sum_{y\in B_\ell}\tau_y h-\pi_{\rho^\ell}(h)\bigg|\bigg)=0,
\end{equation}
for every measure $\bar\mu^T$ which is stationary and translation invariant.
According to Lemma \ref{lem:definetti}, we can  decompose $\bar\mu^T$ as
\[\bar \mu^T=\lambda\mu_{\mc F}^T+(1-\lambda)\int_{[\frac12,1]}\varpi^T(d\rho)\pi_\rho,\]
where $\lambda \in [0,1]$, where the support of $\mu_{\mc F}^T$ is included in $\mathscr{F}_{\Z}$ (recall \eqref{eq:froset}), and $\varpi^T(d\rho)$ is a probability measure on $[\frac12,1]$. 

If $\eta \in \mc F_\Z$ is a frozen configuration, then it cannot contain two neighboring particles, therefore $\rho^\ell(\eta) \leq \frac12$, and both $h(\eta)$ and $\pi_{\rho^\ell(\eta)}(h)$ vanish a.s. Hence \eqref{eq:limOBE} trivially  holds with $\mu_{\mc F}^T$ 
instead of $\bar\mu ^T$. 

Similarly, $\pi_1$--a.s.\ we have $\frac{1}{2\ell+1}\sum_{y\in B_\ell}\tau_y h-\pi_{\rho^\ell}(h)=0$.
 We now prove 
\begin{equation*}
\lim_{\ell\to\infty}\int_{[\frac12,1)}\varpi^T(d\rho)\pi_\rho\bigg(\bigg|\frac{1}{2\ell+1}\sum_{y\in B_\ell}\tau_y h-\pi_{\rho^\ell}(h)\bigg|\bigg)=0. 
\end{equation*}
To do so, fix $\varepsilon >0$, and split the integral in $\rho$ over $[\frac12,1)$ as a first contribution over $[\frac12, \frac12+\varepsilon)$ and a second  
over $[\frac12+\varepsilon,1)$: \begin{itemize} \item If $\rho \in [\frac12+\varepsilon,1)$, one can straightforwardly show using the same proof as in \cite[Corollary 6.6]{BESS}, that  the correlations under the measures $\pi_\rho$ between two boxes at distance $\ell$ decay  exponentially as $e^{-C\ell}$, uniformly in $\rho\in [\frac12+\varepsilon, 1)$. Using the Lipschitz continuity of $\rho\mapsto\pi_\rho(h)$, one easily obtains that there exists a constant $C=C(\varepsilon)>0$ such that
\begin{align}\label{eq:boundspirho}
& \pi_\rho\bigg(\bigg|\frac{1}{2\ell+1}\sum_{y\in B_\ell}\tau_y h-\pi_{\rho^\ell}(h)\bigg|\bigg)\nonumber\\
& \leq \pi_\rho\bigg(\bigg|\frac{1}{2\ell+1}\sum_{y\in B_\ell}\tau_y h-\pi_{\rho}(h)\bigg|\bigg)+\pi_\rho\pa{\Big|\pi_{\rho}(h)-\pi_{\rho^\ell}(h)\Big|}  \\ 
&  =\mathcal{O}_\ell(e^{-C\ell})\nonumber,
\end{align}
which vanishes as $\ell\to\infty$ uniformly in $\rho\in [\frac12+\varepsilon, 1)$. 
\item If $\rho\in [\frac12,\frac12+\varepsilon)$, we use the exact same bound, namely \eqref{eq:boundspirho}, and the fact that   $\pi_\rho(h)=\frac{2\rho-1}{\rho}$.  Moreover, $\rho^\ell\geq \frac12$ a.s.~under $\pi_\rho$, since $\pi_\rho$ only charges configurations without consecutive empty sites. In particular, for any $\rho\leq \frac12+\varepsilon$, the second term in \eqref{eq:boundspirho} can be estimated for any $ K>0$ by 
\[\pi_\rho\pa{\Big|\pi_{\rho}(h)-\pi_{\rho^\ell}(h)\Big|}\leq 8 K\varepsilon+\pi_\rho\big(\rho^\ell- \tfrac 12>K\varepsilon \big).\]
By Markov inequality, the second term in the right hand side is less than $ 1/K$. Therefore, letting $\ell\to\infty$, then $\varepsilon\to0$ and then $K\to\infty$, proves that the second term in \eqref{eq:boundspirho} vanishes uniformly as $\ell\to\infty$ and $\varepsilon\to 0$.  The correlations, however, no longer decay uniformly.
But we can write, recalling the expression \eqref{eq:defh} for the function $h$,
\[\pi_\rho\bigg(\bigg|\frac{1}{2\ell+1}\sum_{y\in B_\ell}\tau_y h-\pi_{\rho}(h)\bigg|\bigg)\leq 2\pi_\rho( |h|)\leq 6\pi_\rho( \eta_0\eta_1)\leq 12\varepsilon,\]
uniformly in  $\rho\in [\frac12,\frac12+\varepsilon)$. 
\end{itemize} This proves Proposition \ref{prop:localerg}. \end{proof}

\subsection{Proof of Lemma \ref{lem:eqpi}} \label{sec:prooflemma}
We  now closely follow  Funaki's proof \cite{Funaki}. Recall that to prove Lemma \ref{lem:eqpi}, one needs to show the following: any limit point $\bar{\mathcal{P}}^*$ as $N\to\infty$ then $\ell\to\infty$ of the sequence $(\bar{\mathcal{P}}_{N,\ell})_{1\leq \ell\leq N}$ satisfies 
\begin{equation}
\label{eq:rhorho'negbis}
\bar{\mathcal{P}}^*\Bigg(\int_{0}^T \int_{\bb T} \int_{[0,1]}\mathcal{H}(r) \bigg(r-\int_{[0,1]}r'p_t(u, dr')\bigg)p_t( u, dr) du dt \Bigg)=0,
\end{equation}
with $\mathcal{H}(r)=\frac{2r-1}{r}{\bf 1}_{\{r\geq \frac12\}} $ and $p_t$ defined by \eqref{eq:YM1}.

\medskip

For $u\in \frac1N \T_N$, $t\in[0,T]$, we let $H^N:t,u\mapsto H^N_t(u)$ be the solution of the discrete heat equation on $\frac1N\bb T_N$ with $N$ particles initially  at the origin, that is
\begin{equation}
\label{eq:heat}
 \begin{cases}
\partial_tH^N_t(u)=\Delta^NH^N_t(u), & \quad u\in \frac1N\bb T_N, \; t\in[0,T]\\ 
H^N_0(u)=N\;{\bf 1}_{\{u=0\}} , & \quad u\in \frac1N \T_N.
\end{cases},
\end{equation}
where $\Delta^N$ is the discrete laplacian defined in \eqref{eq:DefDeltaN}.
One obtains straightforwardly  (cf.~\cite[p.~589]{Funaki}) an explicit expression for $H^N_t(\frac x N)$: for any $x\in \T_N$
\begin{equation}
H_t^N\big(\tfrac x N \big)=1-{\bf{1}}_{\{N\;\mathrm{is\, even} \}}e^{-4tN^2}\cos(\pi x)+2\sum_{k=1}^{N/2}e^{-t\lambda_k^N}\cos\big(\tfrac{2k\pi x}{N}\big), \label{eq:HtN}
\end{equation}
where 
\[\lambda_k^N:=4N^2\sin^2\big(\tfrac{k\pi}{N}\big).\] 
For any time $\tau\geq 0$, let us now introduce 
\begin{align} 
\label{eq:DefV}
\mathcal V_\tau^{N,\ell} : & = \E_{\mu^{N}}\Bigg[\int_{0}^T\frac{1}{N^2}\sum_{y\in\T_N} H_{\tau}^{N}\big(\tfrac y N\big)\sum_{x\in\T_N}\big(\rho^{\ell}_{x+y}(t)+\rho^{\ell}_{x-y}(t)\big)\mathcal{H}(\rho^{\ell}_x(t))dt \Bigg]  \\
& = T\;\bar\mu^{N}_T\bigg(\frac{1}{N}\sum_{y\in\bb T_N} H_{\tau}^{N}\big(\tfrac y N \big)\big(\rho^{\ell}_y+\rho^{\ell}_{-y}\big)\mathcal{H}\big(\rho^{\ell}_0\big)\bigg), \label{eq:DefVbis}
\end{align}
where in the last identity we used the definition of  $\bar\mu_T^N$ given in \eqref{eq:Defmubar}.

For convenience sake, we assume that the sequence $\bar{\mathcal{P}}_{N,\ell}$ converges to $ \bar{\mathcal{P}}^*$ as $N\to\infty$ then $\ell\to \infty$ (in particular the intermediate limits as $N\rightarrow\infty$ for fixed $\ell$ exist). At any moment this assumption can be dropped by taking an arbitrary convergent 
subsequence instead. We prove two important results about $\mathcal V_\tau^{N,\ell}$:

 \begin{lemma}
\label{lem:lim1}
\begin{equation}
\label{eq:lem1}
\lim_{\Theta\to \infty}\lim_{\ell\to\infty}\lim_{N\to\infty} 
\mathcal{V}_{\Theta/N^2}^{N,\ell} 
=2\bar{\mathcal{P}}^*\pa{\int_{0}^T \int_{\bb T} \int_{[0,1]}r \mathcal{H}(r)  p_t( u, dr) du dt }.
\end{equation}
\end{lemma}

\begin{lemma}
\label{lem:lim2}
\begin{equation}
\label{eq:lem2}
\lim_{\theta\to 0}\lim_{\ell\to\infty}\lim_{N\to\infty} \mathcal V_{\theta}^{N,\ell}
=2\bar{\mathcal{P}}^*\pa{\int_{0}^T \int_{\bb T} \int_{[0,1]}\mathcal{H}(r)p_t( u, dr) \int_{[0,1]}r'p_t(u, dr') du dt }.
\end{equation}
\end{lemma}
The proofs of Lemmas \ref{lem:lim1} and  \ref{lem:lim2} are straightforward adaptations of Lemmas 5.5 and 5.6 in \cite{Funaki}.
\begin{proof}
[Proof of Lemma \ref{lem:lim1}]
We first consider $H_{\Theta/N^2}^{N}(\frac y N)$ sampled at a large \emph{microscopic} time $\Theta/N^2$. 
Roughly speaking, $H_{\Theta/N^2}^{N}$  converges to a macroscopic Dirac measure at $0$, so that convoluted with $H_{\Theta/N^2}^N$, 
we get:  as $N\to \infty$,  $\rho^{\ell}_t(x\pm y)\simeq\rho^{\ell}_t(x)$ for $\ell$ large enough, which yields the identity in the lemma. 
More precisely, we use the second expression of $\mathcal{V}_{\Theta/N^2}^{N,\ell}$ given in \eqref{eq:DefVbis}, and we split the sum in $y$ in two parts, depending on whether $|y|\leq \Theta$ or $|y|>\Theta$: \begin{itemize} \item in the second case $|y|>\Theta$, $H_{\Theta/N^2}^{N}(\frac y N )$ is small:

\noindent more precisely, consider a continuous time random walk $X_t$ initially at site $0$, and jumping at rate $N^2$ on each of its neighbors on $\T_N$, by Feynman Kac's formula, we can write  
\[\sum_{|y|>\Theta}\frac{1}{N}H_{\Theta/N^2}^{N}\big( \tfrac y N \big)=\bb P(|X_{\Theta/N^2}|>\Theta)=\mc O(e^{-\Theta}),\]
where the second identity comes from a standard estimate on symmetric random walks.
In particular, since $\bar\mu^{N}_T(\rho^{\ell}_y)$ is uniformly bounded in $y$, the contribution of the sum $|y|>\Theta $ 
vanishes as $N\to\infty$, then $\ell\to\infty$, and then $\Theta\to\infty$ ;
\item in the first case $|y|\leq \Theta$,  $|\rho^{\ell}_y-\rho^{\ell}_0|\leq \frac{2\Theta}{2\ell +1}.$ \end{itemize} 
Therefore, the left hand side in \eqref{eq:lem1}  rewrites as
\[\lim_{\ell\to\infty}\lim_{N\to\infty} T\;\bar\mu^{N}_T\pa{2\rho^{\ell}_0\;\mathcal{H}\big(\rho^{\ell}_0\big) },\]
which is also equal to (recalling Definition \ref{def:Youngmeasure}, and also \eqref{eq:Defmubar} and \eqref{eq:Defrhoell})
\[ 
\lim_{\ell\to\infty}\lim_{N\to\infty}  \int_0^T \int_{\T \times [0,1]} 2r\mathcal{H}(r) \; \pi_t^{N,\ell}(du,dr)dt.
\]
Thanks to Lemma \ref{lem:YM} we get the result \eqref{eq:lem1}.  \end{proof}

\begin{proof}[Proof of Lemma \ref{lem:lim2}]
We now consider a sample of $H_{\theta}^{N}$ at a small \emph{macroscopic} time $\theta$. 
Denote $\mathfrak{h}_s(u)$ the heat kernel on $\bb T$, namely
\[\mathfrak{h}_s(u)=1+ 2 \sum_{k =1}^\infty e^{-sk} \cos(2\pi k u).\]
Since $\E_{\mu^{N}}[\mathcal{H}(\rho^{\ell}_x(t))]$ is bounded uniformly in $N$, 
$x\in \bb T_N$, and $t\leq T$, we obtain from Lemma \ref{lem:YM}
\begin{align*}
\lim_{\theta\to 0}\lim_{\ell\to\infty}&\lim_{N\to\infty} 
\E_{\mu^N}\Bigg[\int_{0}^T\frac{1}{N^2}\sum_{y\in\T_N} H_{\theta}^{N}\big(\tfrac y N \big)\sum_{x\in\T_N}
\big(\rho^{\ell}_{x+y}(t)+\rho^{\ell}_{x-y}(t)\big)\mathcal{H}(\rho^{\ell}_x(t))dt\Bigg]\\
=&\lim_{\theta\to 0}\lim_{\ell\to\infty}\lim_{N\to\infty} 
2\E_{\mu^N}\Bigg[\int_{0}^T\frac{1}{N^2}\sum_{x,y\in\T_N} H_{\theta}^{N}\big(\tfrac{x-y}{N} \big)
\eta_y(t)\mathcal{H}(\rho^{\ell}_x(t))dt\Bigg]\\
=& \; \lim_{\theta\to 0}2\bar{\mathcal{P}}^*\Bigg(\int_{0}^T dt\int_{\bb T}du\int_{\bb T}dv \; \mathfrak{h}_{\theta}(u-v)  \rho_t(v) \int_{[0,1]}\mathcal{H}(r)p_t(u, dr)\Bigg),
\end{align*}
which converges as $\theta\to 0$ to the wanted quantity since $\mathfrak{h}_{\theta}(v)dv$ converges to a Dirac at the origin, and since $ \rho_t(v)=\int_{[0,1]}r'p_t(v, dr').$
\end{proof}

To end the proof of Lemma \ref{lem:eqpi}, we now need to show that both limits \eqref{eq:lem1} and \eqref{eq:lem2} are equal. For a configuration $\eta\in \Sigma_N$, we define the averaged  empirical measure on $\bb T_N$, where the density at each point is averaged out over a large microscopic box of size $\ell\geq 1$, namely
\[m^{N,\ell}(du)=m^{N,\ell}(\eta, du):=\frac 1N\sum_{x\in \bb T_N}\delta_{x/N}(du) \rho^\ell_x(\eta),\]
where $\rho^\ell_x(\eta)$ was defined in \eqref{eq:Defrhoell} as the density in a box of size $\ell$ around $x$. Once again, when $\eta$ depends on time, 
we shorten $ m^{N,\ell}_t=m^{N,\ell}(\eta(t))$. Note in particular that for any function $\xi$ on $\T$, we have $\langle m^{N,\ell}, \xi\rangle=\ang{\pi^{N,\ell},\xi\cdot r }$, where the Young measure $\pi^{N,\ell}$ was introduced in Definition \ref{def:Youngmeasure}, and $r$ is the short notation for the identity application on $[0,1]$. For any $1 \leq \ell \leq N$ and any time $s \geqslant 0$, we introduce 
\begin{equation}\label{eq:deR}  \mathcal{R}^{N,\ell}_s:=\E_{\mu^N}\cro{\langle m_T^{N,\ell},m_T^{N,\ell}*H_s^N\rangle-\langle m_0^{N,\ell},m_0^{N,\ell}*H_s^N\rangle},\end{equation}
where the convolution $*$ between a measure $m$ and a function $\xi^N$ on $\frac1N\bb T_N$ is defined as the function $(m*\xi^N)(u)=\langle m, \xi^N(u-\cdot)\rangle $ for any $u \in \frac1N\bb T_N$. In particular, for any times $t$, $s$ we have (recall Definition \ref{def:Youngmeasure})
\[\langle m_t^{N,\ell},m_t^{N,\ell}*H_s^N\rangle=\frac 1{N^2} \sum_{x,y\in \bb T_N}H_s^N\big(\tfrac y N\big)\rho^\ell_{x-y}(t)\rho^\ell_{x}(t).\]
We need the following two results.
\begin{lemma} \label{lem:limint}
\[\lim_{\theta\to 0}\lim_{\Theta\to\infty}\lim_{\ell\to\infty}\lim_{N\to\infty} 
\int_{\Theta/N^2}^{\theta} \mathcal{R}^{N,\ell}_s ds=0\; . \]
\end{lemma}

\begin{lemma}
\label{lem:limRkl}
For any $\theta>0$,
\[
\lim_{\Theta\to\infty}\lim_{\ell\to\infty}\lim_{N\to\infty} 
\int_{\Theta/N^2}^{\theta}\big( \mathcal{R}^{N,\ell}_s -  \mathcal{T}^{N,\ell}_s\big)  ds = 0,
\]
where 
\[
 \mathcal{T}^{N,\ell}_s:=\E_{\mu^N}\Bigg[\int_{0}^T\frac{1}{N^2}\sum_{y\in\T_N}\Delta^N H_s^{N}\big(\tfrac y N\big)\sum_{x\in\T_N}\pa{\rho^{\ell}_{x+y}(t)+\rho^{\ell}_{x-y}(t)}\mathcal{H}(\rho^\ell_x(t))dt\Bigg].
\]
\end{lemma}

Since we now have all the ingredients to do so, before turning to the proof of Lemma \ref{lem:limint} and Lemma \ref{lem:limRkl}, we conclude the proof of  Lemma  \ref{lem:eqpi}.

\begin{proof}[Proof of Lemma \ref{lem:eqpi}]
Since by definition $\Delta^N H_s^{N}=\partial_s H_s^{N}$, we have
\[
\int_{\Theta/N^2}^\theta \mathcal{T}_s^{N,\ell} ds = \mathcal V_{\theta}^{N,\ell} - \mathcal V_{\Theta/N^2}^{N,\ell}
,\]
where $\mathcal V_{\tau}^{N,\ell}$ was defined in \eqref{eq:DefV}. As a consequence of  Lemmas \ref{lem:limint} and  \ref{lem:limRkl}, we have
\[
\lim_{\theta\to 0}\lim_{\Theta\to\infty}\lim_{\ell\to\infty}\lim_{N\to\infty} 
\int_{\Theta/N^2}^{\theta} \mathcal{T}^{N,\ell}_s  ds = 0.
\]
which proves 
\begin{equation}
\label{eq:limitequality}
\lim_{\Theta\to \infty}\lim_{\ell\to\infty}\lim_{N\to\infty} 
\mathcal{V}_{\Theta/N^2}^{N,\ell}=\lim_{\theta\to 0}\lim_{\ell\to\infty}\lim_{N\to\infty} \mathcal V_{\theta}^{N,\ell}
.\end{equation}
In particular, Lemma \ref{lem:eqpi} follows  from Lemmas \ref{lem:lim1} and \ref{lem:lim2}.\end{proof}

It remains to prove Lemma \ref{lem:limint} and Lemma \ref{lem:limRkl}.

\begin{proof}[Proof of Lemma \ref{lem:limint}]
This is immediate: 
since $H_s^N$ is non-negative, for any $t$  we have 
\[\E_{\mu^N}\cro{\langle m_t^{N,\ell},m_t^{N,\ell}*H_s^N\rangle}\leq \bigg(\frac{1}{N}\sum_{y\in\T_N}H_s^N\big(\tfrac y N \big)\bigg)\E_{\mu^N}\Big[\sup_{x\in\bb T_N}\big(\rho^\ell_x(t)\big)^2\Big]\leq 1.\]
In particular, $\mathcal{R}^{N,\ell}_s$ is uniformly bounded. Since we integrate it over a time segment $s\in [\Theta/N^2, \theta]$ whose length vanishes in the limit, this concludes the proof.
\end{proof}

We now turn to the proof of Lemma \ref{lem:limRkl}, for which we need the following two  technical lemmas.
\begin{lemma}
\label{lem:correlations}
For any function $\xi^N$ on $\frac1N\bb T_N\subset \bb T$, we have the identity
\begin{align*}
N^2\gene_N\big(\langle m^N,m^N*\xi^N\rangle\big) = & \frac{1}{N^2}\sum_{x,y\in\T_N}\Delta^N \xi^N\big(\tfrac y N\big)\pa{\eta_{x+y}+\eta_{x-y}}\tau_x h(\eta) \\
&+\frac{\Delta^N \xi^N(0)}{N^2}\sum_{x\in\T_N}\pa{\eta_{x+1}+\eta_{x-1}-2\eta_x} \tau_x h(\eta). \vphantom{\Bigg(}
\end{align*} 
\end{lemma}
Lemma \ref{lem:correlations} follows from rather elementary computation, we give its proof for the sake of completeness in Appendix \ref{app:tech}.

\begin{lemma}[Equivalent formula for $\mathcal{R}_s^{N,\ell}$] \label{lem:equivformula}
Let us introduce, for any $1\leq \ell\leq N$ and $x \in \T_N$, and any function $\xi^N$ defined on $\frac1N \T_N$, the average function
\begin{equation}
\overline{\xi^N}^{N,\ell}\big(\tfrac x N \big)=\frac{1}{(2\ell+1)^2}\sum_{y_1, y_2\in B_\ell} \xi^N\big(\tfrac{x+y_1+y_2}{N}\big),  
\end{equation}
and recall from \eqref{eq:DefmN} the definition of the empirical measure $m^N_t$. 
Then, seeing $ m^N_t$ as a measure on $\frac1N\bb T_N$, we have
\begin{align}
\label{eq:DefRKl0}
\mathcal{R}^{N,\ell}_s & = \E_{\mu^N}\cro{\langle m_T^N,m_T^N*\overline{H_s^N}^{N,\ell}\rangle-\langle m_0^N,m_0^N*\overline{H_s^N}^{N,\ell}\rangle} \\
& = \E_{\mu^N}\Bigg[\int_{0}^T\bigg(\frac{1}{N^2}\sum_{y\in\T_N}\overline{\big(\Delta^N  H_s^N\big)}^{N,\ell}\big(\tfrac y N \big) \sum_{x\in\T_N}\pa{\eta_{x+y}+\eta_{x-y}}(t)\tau_x h(\eta(t))\nonumber\\
&  \quad \quad \quad \quad \quad \quad \quad  +\frac{\overline{\big(\Delta^N  H_s^N\big)}^{N,\ell}(0)}{N^2}\sum_{x\in\T_N}\pa{\eta_{x+1}+\eta_{x-1}-2\eta_x}(t) \tau_x h(\eta(t))\bigg)dt\Bigg]. \label{eq:DefRKl}
\end{align}
\end{lemma}
\begin{proof}[Proof of Lemma \ref{lem:equivformula}]
The first identity \eqref{eq:DefRKl0} is an easy integration by parts. The second one is obtained by, first, writing Dynkin's formula, and then using Lemma \ref{lem:correlations}.
\end{proof}

We are now ready to prove Lemma \ref{lem:limRkl}.
\begin{proof}[Proof of Lemma \ref{lem:limRkl}]
From \eqref{eq:HtN} one can easily obtain \[\big\|\overline{\big(\Delta^N H_s^N\big)}^{N,\ell}\big\|_\infty\leq\|\Delta^N H_s^N\|_\infty\leq 2\sum_{k=0}^{ N-1}\lambda_k^Ne^{-s\lambda_k^N}.\] One first estimates the contribution to $\int_{\Theta/N^2}^{\theta} \mathcal{R}^{N,\ell}_s ds$ of the second term in the right hand side of \eqref{eq:DefRKl}: this contribution can be crudely bounded from above for any $\ell\leq N$ by 
\begin{multline*}
\Bigg|\int_{\Theta/N^2}^{\theta}\E_{\mu^N}\Bigg[\int_{0}^T\bigg(\frac{\overline{\big(\Delta^N  H_s^N\big)}^{N,\ell}(0)}{N^2}\sum_{x\in\T_N}\pa{\eta_{x+1}+\eta_{x-1}-2\eta_x}(t) \tau_x h(\eta(t))\bigg)dt\Bigg]ds\Bigg|\\
\leq \frac{4T}{N}\sum_{k=0}^{N-1}e^{-\Theta\lambda^N_kN^{-2}}\xrightarrow[N\to\infty]{}4T\int_0^1e^{-4\Theta\sin^2(u\pi )}du,
\end{multline*}
and therefore vanishes as $N\to\infty$ then  $\Theta\to \infty$.

Furthermore, by integrations by parts, one can rewrite  the quantity appearing in the first term of \eqref{eq:DefRKl}, namely
\[\sum_{y\in\T_N}\overline{\big(\Delta^N H_s^N\big)}^{N,\ell}\big(\tfrac y N\big)\sum_{x\in\T_N}\pa{\eta_{x+y}+\eta_{x-y}}(t)\tau_xh(\eta(t)),\] 
as 
\begin{equation*}
\sum_{y\in\T_N}\Delta^N H_s^N\big(\tfrac y N \big)\sum_{x\in\T_N}\pa{\rho_{x+y}^\ell(t)+\rho^\ell_{x-y}(t)}\frac{1}{2\ell+1}\sum_{z\in B_\ell}\tau_{x+z}h(\eta(t)).
\end{equation*} 
In order to prove Lemma \ref{lem:limRkl}, it is therefore enough to show that 
\begin{multline}
\frac{1}{N^2}\E_{\mu^N}\Bigg[\int_{0}^T\sum_{x\in \bb T_N}\sum_{y \in \T_N}\bigg(\int_{\Theta/N^2}^{\theta}\Delta^N H_s^{N}\big(\tfrac y N \big)ds\bigg)\pa{\rho_{x+y}^\ell(t)+\rho^\ell_{x-y}(t)}\\
\times\bigg(\frac{1}{2\ell+1}\sum_{z\in B_\ell}\tau_{x+z}h(\eta(t))-\mathcal{H}(\rho^\ell_t(x))\bigg)dt\Bigg] \label{eq:qua}
\end{multline}
vanishes in the limit of the statement. As before, we first rewrite the integral in $s$ as 
\[\int_{\Theta/N^2}^{\theta}\Delta^N H_s^{N}\big(\tfrac y N \big)ds= H_{\theta}^{N}\big(\tfrac y N \big)-H_{\Theta/N^2}^{N}\big(\tfrac y N \big).\] The absolute value of \eqref{eq:qua} is therefore bounded from above by the sum of four terms
$A^+_{\theta}+A^-_{\theta}+A^+_{\Theta/N^2}+A^-_{\Theta/N^2}$, where $A^{\pm}_{\tau}$ is given by the following expression
\[\E_{\mu^N}\Bigg[\frac{1}{N^2}\int_{0}^T\sum_{x\in \bb T_N}\sum_{y \in\T_N} H_{\tau}^{N}\big(\tfrac y N \big)\rho^{\ell}_{x\pm y}(t)
\bigg|\frac{1}{2\ell+1}\sum_{z\in B_\ell}\tau_{x+z}h(\eta(t))-\mathcal{H}(\rho^\ell_x(t))\bigg|dt\Bigg],\] 
which is bounded uniformly in $\tau$ from above by 
\[T\bar\mu_T^N\Bigg[\bigg|\frac{1}{2\ell+1}\sum_{z\in B_\ell}\tau_{z}h(\eta)-\mathcal{H}(\rho^\ell_0)\bigg|\Bigg],\]
since $\rho^\ell_{x\pm y}\in[0,1]$ and $N^{-1}\sum_{y\in\T_N} H^N_{\tau}(\frac y N)=1$. Proposition \ref{prop:localerg}, together with \eqref{eq:pirhoh}, then conclude the proof.
\end{proof}

\section{Proof of Theorem \ref{thm:fronts}: Creation of the microscopic interfaces}
\label{sec:HittingTime}

\subsection{Creation of the microscopic fronts}
By definition of the initial measure for our process, the supercritical macroscopic phase $\{\rho^{\rm ini}\geq \frac12\}$ can contain neighboring empty sites at the microscopic level, whereas the subcritical macroscopic phase $\{\rho^{\rm ini}\leq \frac12\}$ can contain non-frozen particles (\textit{i.e.}~neighboring particles). Because of the indirect way it proves the hydrodynamic limit using Young's measures, Funaki's scheme (which we have adapted here) does not provide any information on the microscopic structure of  the free boundary problem. 

In this section we prove Theorem \ref{thm:fronts}, \textit{i.e.}~that under reasonable assumptions on the initial profile $\rho^{\rm ini}$, after a macroscopic time of order $t_N=o(1)$, the microscopic structure of the configuration matches the macroscopic one.

Recall that we now make the following assumptions on $\rho^{\rm ini}$:
\begin{align}
\label{ass:H1}
\tag{H1-T1}
&(\rho^{\rm ini})^{-1}([0,\tfrac12])=[0,u_*], \quad \rho^{\rm ini}<1,
\\
\label{ass:H2}\tag{H2}
&\rho^{\rm ini}\in C^2(\bb T)\quad \mbox{ and }\quad \partial_u\rho^{\rm ini}(0),\;\partial_u\rho^{\rm ini}(u_*)\neq 0 .
\end{align}
We emphasize once again that the assumptions on the number of  critical points (only two), and on the initial density which never hits $1$, are  purely for the simplicity of the presentation, and are not required for the proof. Indeed, most of the work of this section concerns the study of the critical interfaces, since the supercritical region (in which the density is larger than, and bounded away from the critical value $\rho_c=1/2$) has already been thoroughly studied in \cite{BESS}. In particular, the technical issues specific to the case where the density hits $1$ are solved therein. In order not to burden this section with analogous results we assume \eqref{ass:H1}.

\subsection{Mapping with the zero-range process} To prove  Theorem \ref{thm:fronts}, we will once again exploit the mapping with the zero-range process. Given a configuration $\eta\in \Sigma_N$, define $K(\eta)=N- \sum_{x\in \T_N}\eta_x$ the number of empty sites in the exclusion configuration $\eta$, and, for $K<N$, define 
\[\Sigma^0_{N,K}=\big\{\eta\in  \Sigma_N\; ;  \; \eta_0=0\mbox{ and }K(\eta)=K\big\}.\]   
We define on $\Sigma^0_{N,K}$ the finite volume counterpart  $ \Pi_{N,K}$ of \eqref{eq:DefPi}, namely
\begin{equation}
\label{eq:DefPiNK}
\begin{array}{cccc}
\Pi_{N,K}:&\Sigma^0_{N,K} &\to&\N^{\T_K} \\
&\eta&\mapsto&\omega^\eta 
\end{array},
\end{equation}
where, as before, for any $k\in \T_{K(\eta)}$, $\omega^\eta _k$ is the number of particles between the $k$--th and $(k+1)$--th empty site  (to the right of 0) in $\eta$.

Given a trajectory $\eta(t)$ of the exclusion process, define $K_0=K(\eta(0))$ the initial number of empty sites in the configuration.
Mark the first empty site in $\eta(0)$ to the right of site $0$ if it exists. We keep track of the motion of this empty site and denote $X(t)$ its position\footnote{{If $X(t^-)=x$ and a particle jumps from $x\pm 1$ to $x$ at time $t$, then $X(t)=x\pm 1$.}} at time $t$. We then denote by
\[\widetilde{\eta}(t)=\tau_{X(t)}\eta(t)\in \Sigma^0_{N,K_0} \] 
the exclusion configuration seen from the marked  empty site.

We now denote by $\widetilde{\omega}(t)=\omega^{\widetilde{\eta}(t)}=\Pi_{N, K_0}(\widetilde{\eta}(t))$ the associated zero-range configuration. If $\eta(0)\equiv \mathbf{1}$ (the constant configuration with particles at each site), which happens with vanishing probability, then we let by convention $\widetilde\omega(0)= N\in\N^{\T_1}$, the zero-range configuration with only one site and $N$ particles on this site.

 Then, as detailed in \cite[Section 3]{BESS}, $\{\widetilde{\omega}(t)\}_{t\geqslant 0}$ is a Markov process, initially in the state $\widetilde{\omega}(0)=\omega^{\widetilde{\eta}(0)}$, and driven by the generator $N^2 \mathcal{L}^{\mathrm{ZR}}_{K_0}$ (recall also \eqref{eq:ZRgen} for the infinite volume version), where
\begin{equation}
\label{eq:ZRgenfin}
\mathcal{L}^{\mathrm{ZR}}_{K}f(\omega):=\sum_{x\in\T_{K}}\sum_{\delta=\pm1}{\bf 1}_{\{\omega_x\geq 2\}}\big(f(\omega^{x,x+\delta})-f(\omega)\big).
\end{equation}
As already noted, for any $\alpha\geq 1$ one can define an equilibrium (grand canonical) distribution  $\nu_{ \alpha}^K$ of the zero-range generator $\mathcal{L}_K^{\rm ZR}$ on $\T_K$, as the geometric product homogeneous measure
\begin{equation}
\label{eq:Defnurhobar}
\nu_{\alpha}^K(\omega_0=p)={\bf 1}_{\{p\in\N, p\geq 1\}}\frac{1}{\alpha}\Big(1-\frac{1}{\alpha}\Big)^{p-1},
\end{equation}
which satisfies the detailed balance condition w.r.t.~$\mathcal{L}_K^{{\rm ZR}}$. We then denote by $\nu_{\alpha}^*$ the product measure on the set $\N^\N$ of semi infinite zero-range configurations with marginals given by \eqref{eq:Defnurhobar}.

 Given an integer $K$ and an initial zero-range configuration $\bar\omega\in\N^{\T_K}$, we denote by $ \mathcal{Q}_{K, \bar\omega}$  the probability distribution on the path space $\mathcal{D}([0,T],\N^{\T_K})$ of the zero-range process started from a \emph{fixed} configuration $\bar\omega$, and driven by the \emph{non-accelerated} zero-range generator $\mathcal{L}^{\mathrm{ZR}}_{K}$. 

\begin{remark}If $\eta(0)$ is distributed according to the initial measure $\mu^N$ fitting $\rho^{\rm ini}$, then the distribution of $\tilde\omega(0)=\omega^{\tilde\eta(0)}$ can also be associated with a profile in the following way. For $u\in\T$, define 
\begin{equation} 
\label{eq:Defv1}
v(u)=\int_{0}^u(1-\rho^{\rm ini}(u'))du',\qquad \bar v=v(1)>0, \end{equation} 
and $\alpha^{\rm ini}:[0,\bar v)\rightarrow\R_+$ such that, for any $v=v(u)\in[0,\bar v)$, 
\begin{equation}
\alpha^{\rm ini}(v)=\frac{\rho^{\rm ini}}{1-\rho^{\rm ini}}(u).\label{eq:Defv2}
\end{equation}
Then one could prove that, for all $\delta>0$ and smooth test function $\phi$, \begin{equation}
\P_{\mu_N}\bigg(\bigg|\frac{\bar v}{K_0}\sum_{x=1}^{K_0}\phi\Big(\frac{x\bar v}{K_0}\Big)\widetilde\omega_x(0)-\int_0^{\bar v}\phi(v)\alpha^{\rm ini}(v)dv\bigg|>\delta\bigg)\xrightarrow[N\to\infty]{}0.
\end{equation}
We will not need this result, but a weaker version can be found in Appendix \ref{ss:mappings}.
 \end{remark}

Note that by assumption  \eqref{ass:technical}, we have $\rho^{\rm ini}<1$, therefore $\alpha^{\rm ini}$ is well defined. Under this mapping, if \eqref{ass:technical} holds, the two critical points $0$ and $u_*$ are mapped respectively to $0$ and $v_*:=v(u_*)\in(0,\bar v]$, which satisfy $(\alpha^{\rm ini})^{-1}([0,1])=[0,v_*]$.

The main advantage of working with the zero-range process is the following \emph{monotonicity property} (see \textit{e.g.}~\cite[Chapter 2, Section 5]{KL}). 
Consider two trajectories $\{\omega(t)\}_{t\in [0,T]}$ and $\{\omega'(t)\}_{t\in [0,T]}$ 
driven by the generator $\gene^{\rm ZR}_K$, 
respectively started from two configurations $\bar\omega\leq{\bar\omega}'$. Then, one can couple both processes $\omega$ and $\omega'$ in such 
a way that at any positive time $t$, $\omega(t)\leq \omega'(t)$. 
In particular, given an event $E\subset \N^{\T_K}$ increasing in the configuration, and if $\bar \omega\leq \bar \omega'$, for any $t\geq 0$,
\begin{equation}\label{eq:monotonicity}\mathcal Q_{K, \bar\omega}(\omega(t)\in E)\;\leq \;\mathcal Q_{K, \bar\omega'}(\omega'(t)\in E).\end{equation}

\subsection{Typical zero-range configurations}
\label{sec:regZR}

In this section we define a set $\mathcal{T}_K$ of \emph{typical zero-range configurations}. 
Define 
$\ell_K=K^{\frac 34}$, and denote
\begin{equation}
\label{eq:DefFKEK}
\mathbf{B}_K:=\left\{\ell_K,\dots, k_*- \ell_K\right\},\quad\mbox{ and }\quad \mathbf{A}_K=\T_K\setminus \mathbf{B}_K,
\end{equation}
where $k_*:=\left\lfloor\frac{Kv_*}{\bar v}\right\rfloor$ is the microscopic site corresponding to the macroscopic critical point $v_*$. The set $\mathbf{B}_K$ is the set of sites in the subcritical phase at distance at least $\ell_K$ of the macroscopic critical points $\{0,v_*\}$. 
Note that for any fixed $K$, the sets $\mathbf{A}_K$ and $\mathbf{B}_K$ only depend on the initial macroscopic profile $\rho^{\rm ini}$.

Given a zero-range configuration $\omega$ and a set $\Lambda$, we denote by
\begin{equation}\label{eq:alphaLambda}\alpha_{\Lambda}(\omega)=\frac{1}{|\Lambda|}\sum_{x\in \Lambda}\omega_x\end{equation}
the empirical density of $ \omega$ in the set $\Lambda$. Define 
\begin{align}
\label{eq:DefC*}
c_*&=4\bar v \min\big\{-\partial_u \rho^{\rm ini}(0)\; ; \; \partial_u \rho^{\rm ini}(u_*)\big\} \\ & =\bar v\min\big\{-\partial_v \alpha^{\rm ini}(0) \; ; \; \partial_v \alpha^{\rm ini}(v_*)\big\}>0, \notag
 \end{align}
 and  introduce 
\[\alpha_K=1+c_*\frac{\ell_K}{K}>1\quad  \mbox{ and  }\quad \Lambda^+_K=\{1,\dots,10\ell_K\}.\]
Throughout, we will not burden the notations and write for example $\alpha_K\ell_K$ instead of $\lfloor \alpha_K\ell_K\rfloor$. We further define
\begin{equation}
\label{eq:DefcK}
c_{K}(\omega)=\sum_{x\in \Lambda^+_K}\omega_x(x-5\ell_K),
\end{equation}
which sums the arithmetic distances between particles in $\Lambda_K^+$ and the center of $\Lambda_K^+$.
We now introduce the subset $\Omega^+_K \subset \N^{\T_K}$ given by
\[\Omega^+_K=\bigg\{\omega\in\N^{\T_K}\; ; \;\alpha_{\Lambda_K^+}(\omega)=\alpha_K, \quad c_K(\omega)\leq 0, \mbox{ and }\ \forall\; x\notin\Lambda_K^+,\; \omega_x=0\bigg\}.\]
Note that in the last definition we slightly abused our notation, and by $\alpha_{\Lambda_K^+}(\omega)=\alpha_K$, we actually mean that $\sum_{\Lambda_K^+}\omega_x=\lfloor\alpha_K|\Lambda_K^+|\rfloor$.

We denote  by $\Omega^-_K$ the set of configurations such that the configuration $\omega'_{x}=\omega_{-x}$ (obtained by symmetry w.r.t.~the origin) is in $ \Omega^+_K$. In other words, configurations in $\Omega^\pm_K$ have slightly more than one particle per site in a box of size $10\ell_K$ to the left/right of the origin, and those particles, on average, are closer to the origin than to the other extremity of the box.
\begin{definition}
\label{def:regZR}
We call a configuration $\omega\in \N^{\T_K}$ \emph{typical} if it meets the  following two conditions : 
\begin{enumerate}[(i)]
\item For any $x\in \mathbf{B}_K$, and any connected set $\Lambda\subset \mathbf{B}_K$ containing $x$ such that $|\Lambda|\geq \ell_K$, we have $\alpha_{\Lambda}(\omega)\leq 1.$
\medskip

\item For any $x\in \mathbf{A}_K$, there exists $\omega'\in  \Omega^+_K\cup  \Omega^-_K$  (depending on $x$) such that $ \omega\geq \tau_{-x} \omega'.$
\end{enumerate}
We denote by $\mathcal{T}_K \subset \N^{\T_K}$  the set of typical configurations.
\end{definition}
The first condition states that no large subcritical box has an abnormally large density. The second one states that for any site $x$ close enough to the supercritical phase, one can always find a neighboring large box $x+\Lambda_K$, {containing at least $\alpha_K>1$ particle per site on average. In $\omega'$, we keep only the particles closest to $x$, which will ensure that 
 $c_K(\omega')\leq 0$ w.h.p. Then, w.h.p, at least one of those excess particles will  eventually exit the box through site $x$.}

\begin{lemma}
\label{lem:Typconf}
Recall that $K_0(\eta)$ is the number of empty sites in the exclusion configuration $\eta(0)$, which is distributed according to $\mu^N$. We have
\begin{equation*}
\lim_{N\to\infty} \bb P_{\mu_N}\big(K_0\notin \mathbf{I}_N\quad \text{\emph{or}} \quad \widetilde{\omega}(0)\not\in \mathcal{T}_{K_0}\big)=0, 
\end{equation*}
where $
\mathbf{I}_{N}=\left\{\bar v N-\log^2N,\dots, \bar v N+\log^2N\right\}.
$
\end{lemma}
The proof of this lemma requires Assumption \eqref{ass:2reg}. It is fairly technical but poses no significant difficulty, we give it  in Appendix \ref{app:typical}. 

{\begin{remark}
Theorem \ref{thm:fronts} holds for initial measures different from 
$\mu^N$, as long as they satisfy the analog of Lemma \ref{lem:Typconf}.
\end{remark}}

\subsection{Bound on the maximum of the zero-range process}

Because the jump rate per site is always $1$ (provided that the constraint $\omega_x\geq 2$ is satisfied), the facilitated zero-range process can be interpreted as  a family of random walks, 
where each random walker jumps ``independently'' at a rate $1/k$, 
where $k$ is the number of other random walkers on the same site, assuming the random walker is not alone on the site, in which case it remains there. 
With this in mind, we prove a technical lemma, giving a uniform bound on the number of particles at any site in $\omega(t)$, 
which will be useful to bound from below the jump rate 
of each individual particle.

\begin{lemma}
\label{lem:UnifBoundomega} Let $T_K=K^{7/4}$.
Then,
the following limit holds:
\begin{equation*}
 \lim_{N\to\infty}\bb P_{\mu_N}\big(\widetilde{G}_N^c\big)=0,
\end{equation*}
where \begin{equation}
\label{eq:DefGN}
\widetilde{G}_N=\Big\{\forall\; x\in \bb T_{K_0}, \; \forall\; t\leq T_{K_0}N^{-2}, \quad \widetilde{\omega}_x(t)< \log^2 K_0 \Big\}.
\end{equation}
\end{lemma}
\begin{proof}[Proof of Lemma \ref{lem:UnifBoundomega}]
Let us denote by $\mathcal{Q}_{K,\alpha}^{\rm eq}$ the distribution of the zero-range process generated by $\mathcal{L}_K^{\rm ZR}$, started from its equilibrium distribution $\nu_\alpha^K$ on $\T_K$ (recall \eqref{eq:Defnurhobar}). We first claim that, letting $\rho^*=\sup_\T \rho^{\rm ini}<1$ and $\alpha^*=\frac{\rho^*}{1-\rho^*}$, we have 
\begin{equation}
\label{eq:GNc}
\bb P_{\mu_N}\big(\widetilde{G}_N^c\big)\leq \bb P_{\mu_N}\pa{K_0\not\in \mathbf{I}_N}+ \sup_{K\in \mathbf{I}_N}\mathcal{Q}_{K,\alpha^*}^{\rm eq}\pa{\mathcal{G}_{K}^c},
\end{equation}
where $\mathcal{G}_{K}$ is defined as
\begin{equation}
\label{eq:DefGNK}
\mathcal{G}_{K}=\Big\{\forall\; x\in \T_{K}, \; \forall\; t\leq T_K, \quad \omega_x(t)< \log^2 K \Big\}.
\end{equation}
 {Indeed, one can prove by standard arguments that there exists a coupling between $\widetilde\omega(0)$ and a semi-infinite zero-range configuration $\omega^*$ with distribution $\nu^*_{\alpha^*}$ such that $\widetilde\omega(0)\leq \omega^*_{|\T_{K_0}} $ (identifying $\T_{K_0}$ with $\{1,\ldots,K_0\}$).}

In particular, by monotonicity of the zero-range process \eqref{eq:monotonicity}, 
\begin{align}
\bb P_{\mu_N}\big(\widetilde{G}_N^c\big)
&\leq \bb P_{\mu_N}\pa{K_0\not\in \mathbf{I}_N}+\sum_{K\in \mathbf{I}_N}\bb P_{\mu_N}\pa{\widetilde{G}_N^c\cap\{K_0=K\}}\nonumber\\
&\leq \bb P_{\mu_N}\pa{K_0\not\in \mathbf{I}_N}+\sup_{K\in \mathbf{I}_N}\mathcal{Q}_{K,\alpha^*}^{\rm eq}\pa{\mathcal{G}_{K}^c},\label{e:decompG_N^c}
\end{align}
where we used to establish the second bound both the coupling above and the fact that the event $\mathcal{G}_{K}^c$ is increasing in the initial configuration.

We now estimate the equilibrium probability $\mathcal{Q}_{K,\alpha^*}^{\rm eq}\pa{\mathcal{G}_{K}^c}$. The process $\{\omega(t)\}_{t\geq 0}$ can be constructed as a time-change of a discrete-time Markov chain on $(\N_*)^{\T_K}$, where $\N_*=\{1,2,\dots\}$ is the set of positive integers. Consider the transition matrix given by \[p(\omega,\omega')= \begin{cases}\tfrac{1}{2K} & \text{if there exist }x\in\T_K \text{ and } \delta\in\{\pm 1\} \text{ s.t. } \omega'=\omega^{x,x+\delta} \text{ and }\omega_x\geq 2; \\ 0 & \text{else}.\end{cases}\] Let us denote by $\{\omega^{\mathbf d}(n)\}_{n\in\N}$ this discrete-time Markov chain with initial distribution $\nu_{\alpha^*}^*$. Then \[\big\{\omega(t)\big\}_t\overset{(d)}{=}\big\{\omega^{\mathbf d}(N_t)\big\}_t,\] where $\{N_t\}_{t\geq 0} $ is a standard Poisson process independent of $\omega^{\mathbf d}$. Moreover, $\omega^{\mathbf d}$ is reversible w.r.t.~$\nu_{\alpha^*}^*$. Therefore, writing $\mathbf{P}$ for the joint distribution of $\omega^{\mathbf d}$ and $N$, 
\begin{align*}
\mathcal{Q}_{K,\alpha^*}^{\rm eq}\pa{\mathcal{G}_{K}^c}&=\sum_{n=0}^\infty \mathbf{P}(N_{{T_K}}=n)\mathbf{P}\big(\exists\; i\leq n,\ \exists\; x\in\T_K\ \text{s.t.}\ \omega^{\mathbf d}_x(i)\geq \log^2K\big)\\
&\leq K\sum_{n=0}^\infty n \mathbf{P}(N_{T_K}=n)\nu_{\alpha^*}^*(\omega_0\geq \log^2K)\\
&\leq K\mathbf{E}[N_{T_K}]\Big(1-\frac{1}{\alpha^*}\Big)^{\log^2K-1}=K^{11/4}\Big(1-\frac{1}{\alpha^*}\Big)^{\log^2K-1}.
\end{align*}
Since $K\geq \bar v N-\log^2N$ for any $K\in \mathbf{I}_N$, we obtain
\begin{equation*}
\lim_{N\rightarrow\infty}\sup_{K\in \mathbf{I}_N}\mathcal{Q}_{K,\alpha^*}^{\rm eq}\pa{\mathcal{G}_{K}^c}=0.
\end{equation*}
Combining this with \eqref{e:decompG_N^c} and Lemma \ref{lem:Typconf}, we conclude the proof of the lemma.  
\end{proof}

\subsection{Front creation for the zero-range process}
\subsubsection{Typical ZR configurations become two-phased in subdiffusive time}
\begin{definition}[Two-phased zero-range configurations]
\label{Def:TypicalConfigurations2}
A zero-range configuration $\omega\in\N^{\T_K}$ is called \emph{two-phased} if there exists a partition $\T_K =A\sqcup B$, where $A$ and $ B$ are both connected subsets of $\T_K$, and $\omega_{|A}\geq 1$ and $\omega_{|B}\leq 1$.

 We denote by $\mathfrak{P}^{{\rm ZR}}_K\subset\N^{\T_K}$ the set of two-phased zero-range configurations. Note that a zero-range configuration $\omega^\eta$ is two-phased iff an associated exclusion configuration $\eta$ also is (regardless of the marked empty site chosen in the exclusion configuration).  
\end{definition}
The main ingredient to prove Theorem \ref{thm:fronts} is an analogous result for the zero-range process started from a typical configuration. Recall that $\mathcal{Q}_{K,\bar \omega}$ denotes the distribution of the \emph{non-accelerated} zero-range process with initial configuration $\bar \omega$ and infinitesimal generator $\mathcal{L}_K^{\rm ZR}$.
\begin{proposition}[Hitting time of $\mathfrak{P}_K^{{\rm ZR}}$]
\label{prop:HTZR}
Recall $T_K=K^{\frac 74}$, 
\begin{equation*}
\lim_{K\to\infty}\sup_{\bar \omega\in \mathcal{T}_K}\mathcal{Q}_{K,\bar \omega}\Big(\mathcal{G}_{K}\cap\big\{\omega(T_K)\not\in \mathfrak{P}_{K}^{{\rm ZR}}\big\}\Big)=0,
\end{equation*}
where $\mathcal{G}_{K}$ has been defined in \eqref{eq:DefGNK}.
\end{proposition}

Proposition~\ref{prop:HTZR} is a consequence of Lemmas \ref{lem:EK} and \ref{lem:FK} below, which are proved respectively in Sections \ref{sss:supercritical} and \ref{sss:subcritical}.  Recall the definition \eqref{eq:DefFKEK} of the sets $\mathbf{A}_K$ and $\mathbf{B}_K$.

\begin{lemma}
\label{lem:EK} 
With high probability the set $\mathbf{A}_K$ becomes supercritical before time $T_K$, precisely:
\begin{equation}
\label{eq:EK}
\lim_{K\to\infty}\max_{\bar\omega\in \mathcal{T}_K}\mathcal{Q}_{K, \bar\omega}\big(\mathcal{G}_{K}\cap\{\exists\; x\in \mathbf{A}_K,  \;\omega_x(T_K)=0 \}\big)=0.
\end{equation}
\end{lemma}

\begin{lemma}
\label{lem:FK} 
With high probability, after time $T_K$, there is a unique subcritical connected set, precisely:
\begin{multline}
\label{eq:FK}
\lim_{K\to\infty}\max_{\bar\omega\in \mathcal{T}_K}\mathcal{Q}_{K, \bar\omega}\Big(\mathcal{G}_K\cap \Big\{\exists\;  x<y<z\in \mathbf{B}_K, \\
\omega_x(T_K)=\omega_z(T_K)=0\text{\emph{ and }}\omega_y(T_K)> 1\Big\}\Big)=0.
\end{multline}
\end{lemma}

First, we prove that these two lemmas imply the result stated in Proposition \ref{prop:HTZR}. 

\begin{proof}[Proof of Proposition \ref{prop:HTZR}]
One can choose  
\[\mathbf{B}=\max\Big\{\{x,\dots,z\}\subset \mathbf{B}_K,  \;\omega_x(T_K)=\omega_z(T_K)=0\Big\} ,\]
where the max is taken for the inclusion. The configuration $\omega(T_K)$ is subcritical on $\mathbf{B}$  according to Lemma \ref{lem:FK}, and supercritical on $\mathbf{A}:=\T_K\setminus \mathbf{B}$ with high probability,  according to Lemma \ref{lem:EK}.
\end{proof}
\subsubsection{Stuck zero-range}
\label{sec:chi}
{In this paragraph we introduce an auxiliary process which will be used to prove both Lemma \ref{lem:EK} and Lemma \ref{lem:FK}.} Fix a box $\Lambda \subset \T_K$, and define its exterior boundary $\partial\Lambda=\{x\in \bb T_K, d(x, \Lambda)=1\}$ and set $\overline{\Lambda}=\Lambda\cup\partial\Lambda$. 
In what follows, we will couple $\omega$  with an auxiliary process $\chi$ where the particles in $\Lambda$ copy exactly the jumps performed by $\omega$, 
but any jump occurring from a site $y\notin \Lambda$ is canceled.  Under this coupling, particles in $\chi$ behave as those in $\omega$ up to the time when they leave $\Lambda$, 
where they get stuck.
Defined in this way, the process $\{\chi(t)\}_t$ is a Markov process, driven by the generator $\mc L_\Lambda^{\rm st}$, defined as
\begin{equation*}
\mathcal{L}_{\Lambda}^{\rm st}f(\chi):=\sum_{\substack{x\in \Lambda\\|x-y|=1}} {\bf 1}_{\{\chi(x)\geq 2\}}\big(f(\chi^{x,y})-f(\chi)\big).
\end{equation*}
We denote by $\Q^{\rm st}_{\Lambda, \bar\omega}$ the distribution of the process $\{\chi(t)\}_t$ started from $\bar\omega$ and driven by the generator $\mathcal{L}_{\Lambda}^{\rm st}$ above, and we denote by
\[T^{\Lambda}_\chi=\inf\big\{t\geq 0\; \colon  \;{\sup_{x\in \bb T_K}\chi_x(t)> \log^2 |\Lambda|}\quad\text{or}\quad \chi_y(t)\leq 1,\  \forall\; y\in \Lambda\big\}\] 
the time at which either the number of particles became too high at some site, or all the particles got stuck (either by leaving $\Lambda$ or by remaining alone on a site). 

The following result is analogous to Lemma 4.4 of \cite{BESS}, and is proved in the same way:
\begin{lemma}
\label{lem:ExitTime}
For any $\theta>0$, there exists $\lambda_0(\theta)$ such that for any sequence of sets $\Lambda(K)\subset \bb T_K$ satisfying $\log^2 K |\Lambda(K)|^{2+\theta}\ll K^2$, and $ |\Lambda(K)|\geq \lambda_0(\theta)$ for all $K$ large enough, 
\[\Q^{\rm st}_{ \Lambda(K),\bar\omega}\bigg( T_\chi^{\Lambda(K)}\geq \log^2 K |\Lambda(K)|^{2+\theta} \bigg)\leq \log^2 K e^{-|\Lambda(K)|^{\theta/2}}.\]
\end{lemma}
\begin{proof}[Proof of Lemma \ref{lem:ExitTime}]
{The proof is based on a coupling argument and can be obtained with small modifications from the proof  of \cite[Lemma 4.4]{BESS}. We sketch here its more salient points. Let us fix $K$ and write $\Lambda:=\Lambda(K)$.

The first step consists in coupling the process $\chi$ with another process $\sigma$, namely a system of independent symmetric random walks that jump at rate $1/\log^2K$ inside $\Lambda$ and get stuck when they exit it. Letting \[T^{\Lambda}_\sigma=\inf\big\{t\geq 0\; \colon  \;{\sup_{x\in \bb T_K}\sigma_x(t)> \log^2 |\Lambda|}\quad\text{or}\quad \sigma_y(t)\leq 1,\  \forall\; y\in \Lambda\big\},\]standard arguments allow to show that $T^{\Lambda}_\chi\leq T^\Lambda_\sigma$.

It then remains to prove Lemma \ref{lem:ExitTime} with $T_\sigma^\Lambda$ in place of $T_\chi^\Lambda$, which follows from standard estimates on random walks (see e.g.\ \cite{varadhan} p. 173).} Note that before $T_\sigma^\Lambda$, there are at most $\log^2K|\Lambda|$ particles to consider.
\end{proof}
\subsubsection{Supercritical phase; proof of Lemma \ref{lem:EK}}
\label{sss:supercritical}

We will simply sketch the proof of Lemma \ref{lem:EK}, since given the definition of typical configurations it is analogous to Proposition 4.1 in \cite{BESS}. 
To prove Lemma  \ref{lem:EK} it is enough to show
\begin{equation}
\label{eq:EK2}
\sup_{\substack{\bar\omega\in \mathcal{T}_K\\ 
x\in \mathbf{A}_K}}
\mathcal{Q}_{K, \bar\omega}\big(\mathcal{G}_K\cap\{\omega_x(T_K)=0\}\big)=o(K^{-1}).
\end{equation}
Fix $x\in \mathbf{A}_K$, any typical configuration $\bar\omega$ satisfies condition \emph{(ii)} in Definition \ref{def:regZR}. Assume for example that there exists $\bar\omega'\in  \Omega^+_K$  such that $\bar\omega\geq \tau_{-x}\bar\omega'$; the other case is treated in the exact same way. Then, since the event $\mathcal{G}_K\cap\{\omega_x(T_K)=0\}$ is decreasing in the configuration, translating the problem back to the origin, to prove \eqref{eq:EK2} it is sufficient to prove that 
\[\limsup_{K\to\infty}\;\;K\sup_{\bar\omega\in  \Omega^+_K}\mathcal{Q}_{K, \bar\omega}\big(\mathcal{G}_K\cap\{\omega_0(T_K)=0\}\big)=0.\]
As outlined in Section \ref{sec:chi}, we now couple $\omega$ with the auxiliary zero-range process $\chi$ with generator $\mathcal{L}_{\Lambda_K^*}^{\rm st}, $ where
$\Lambda_K^*=\{1,\dots,10\ell_K\}$. 
Let us introduce the event
\begin{equation*}
\mathcal{G}_K^\chi=\bigg\{\sup_{x\in \Lambda_K^*}\sup_{t\leq T_K}\chi_t(x)< \log^2 K \bigg\},
\end{equation*}
which is $\mathcal{G}_K$'s counterpart  for $\chi$.

Fix $\bar\omega\in  \Omega^+_K$.
We can write
\begin{equation}
\label{eq:EK12}
\mathcal{Q}_{K, \bar\omega}\big(\mathcal{G}_K\cap\{\omega_0(T_K)=0\}\big) \leq  \Q^{\rm st}_{\Lambda_K^*, \bar\omega}\big(\mathcal{G}^\chi_K\cap\{\chi_0(T_K)=0\}\big). 
\end{equation}
We shorten $T_\chi=T_\chi^{\Lambda_K^*}$. 
Choose $\theta=\frac 14$ and recall that $\ell_K=K^{3/4}$, 
\[T_K=K^{7/4}\gg \log^2 K\ \ell_K^{2+\theta}.\] Therefore, according to Lemma \ref{lem:ExitTime}, for any $K$ large enough, 
\begin{equation}
\label{eq:EK13}
\Q^{\rm st}_{\Lambda_K^*, \bar\omega}\big(T_\chi\geq T_K\big)\leq e^{-K^{1/16}}.
\end{equation}
In particular, in order to prove Lemma \ref{lem:EK}, it is enough to prove 
\begin{equation}
\label{eq:chi}
\limsup_{K\to\infty} \;\; K\sup_{\bar\omega\in\Omega_K^+} \Q^{\rm st}_{\Lambda_K^*, \bar\omega}\pa{\{\chi_0(T_{\chi})=0\}\cap\mathcal{G}^\chi_K\cap\{T_\chi\leq T_K\}}=0.
\end{equation}
On the event $\mathcal{G}^\chi_K\cap\{T_\chi\leq T_K\}$, we have $\chi(T_\chi)\leq 1$ on $\Lambda_K^*$. In particular, $T_\chi=T'_\chi:=\inf\{t\geq 0\colon \chi_y(t)\leq 1,\ \forall y\in\Lambda_K^*\}$. Therefore, we only have to prove
\begin{equation}
\label{eq:chi'}
\limsup_{K\to\infty} \;\; K\sup_{\bar\omega\in\Omega_K^+} \Q^{\rm st}_{\Lambda_K^*, \bar\omega}\pa{\{\chi_0(T'_{\chi})=0\}}=0.
\end{equation}
First recall that any configuration $\omega\in \Omega^+_K$ has $10\ell_K(1+c_*\frac{\ell_K}{K})$ particles. Denote by $t_j$ the jump times  $0=t_0<t_1<\cdots<t_L=T'_{\chi}$ of the process $\chi$ before time $T'_\chi$, and define for any $j\leq L$
\[Z_j=\sum_{x=0}^{10\ell_K+1} \chi_x(t_j)(x-5\ell_K),\]
Since $\bar\omega\in \Omega^+_K$, recalling \eqref{eq:DefcK}, we have $Z_0 \leq 0$. Furthermore, 
 at time $T'_\chi$, at least $10c_*\ell^2_K/K$ particles have exited $\Lambda_K^*$, and if none is at site $0$, all those particles got stuck at site $10\ell_K+1$ and it is therefore straightforward to show that
\[\chi_0(T'_\chi)=0\quad\Rightarrow \quad Z_L\geq 50c_* \ell^3_K/K,\]
because then the minimal value for $Z_L$ is the case where \[\chi(T'_\chi)_{|\Lambda_K^*}\equiv {\bf1}, \quad \text{ and } \quad \chi_{10\ell_K+1}(T'_\chi)=10c_*\ell^2_K/K.\] 
Recall that there are less than $20\ell_K$ particles initially in $\Lambda_K^*$, and each of those particles either gets stuck or exits $\Lambda_K^*$ in $\mc O(\ell_K^{2+\theta})$ jumps with probability $1-\mc O(e^{-\ell_K^{\theta}})$. Elementary computations yield that 
\begin{equation*}
\sup_{\bar\omega\in \Omega_K^+} \Q^{\rm st}_{K, \bar\omega}\pa{ L \geq \ell_K^{3+\theta}}=\mc O( e^{-\ell_K^{\theta/2}}).  
\end{equation*}
Moreover, the process $\{Z_j\}_j$ is distributed as a discrete time, nearest-neighbor,  symmetric random walk up to time $L$, so that 
\begin{align*}\sup_{\bar\omega\in \Omega_K^+}\Q^{\rm st}_{K, \bar\omega}\pa{\chi_0(T_{\chi})=0\mbox{ and }L < \ell_K^{3+\theta}}&\leq \bb P\bigg(\sup_{0\leq j\leq \ell_K^{3+\theta}}X_j\geq 50c_*\ell_K^3 /K\bigg)\\
&=\mc O\big(e^{-C\ell_K^{3-\theta}/K^2}\big) \vphantom{\bigg(},
\end{align*}
 for some positive constant $C$ depending on $c_*$ where $\bb P$ is the distribution of a discrete time random walk $X$ initially at the origin.  
Since $\ell_K=K^{\frac 34}$, the last two bounds together finally yield for $\theta=\frac14$ that 
\[K \sup_{\bar\omega\in \Omega_K^+}\Q^{\rm st}_{K, \bar\omega}\big(\chi_0(T'_\chi)=0\big)=\mc O\big(K e^{-K^{1/16}}\big).\]
The right hand side vanishes as wanted as $K\to\infty$, which proves \eqref{eq:chi} and  then Lemma \ref{lem:EK}.

 \subsubsection{Subcritical phase; proof of Lemma \ref{lem:FK}}
\label{sss:subcritical}

Denote by $E$ the event inside brackets in \eqref{eq:FK}, and define 
\begin{align*} E_{x,y,z}=\Big\{\omega_x(T_K)=\omega_z(T_K)=0,\; & \omega_y(T_K)> 1\\ &\mbox{and} \quad\omega_j(T_K)\geq1, \;\forall\; j\in \{x+1,\dots,z-1\}\Big\},\end{align*} 
which yields straightforwardly
\[E=\bigcup_{x<y<z\in \mathbf{B}_K}E_{x,y,z}.\]
Since $|\mathbf{B}_K|\leq K$ for $K$ large enough, it is enough to show that, uniformly in $x<y<z\in \mathbf{B}_K$,  $K^3\mathcal{Q}_{K, \bar\omega}\pa{\mathcal{G}_K\cap E_{x,y,z}}$ vanishes.

\medskip

Assume first that $z-x>\ell_K$ and set $\Lambda=\{x+1,\dots, z-1\}$. If $\bar\omega$ is a regular configuration,  it satisfies condition \emph{(i)} in Definition \ref{def:regZR}, and in particular, we must have  $\alpha_{\Lambda}(\bar \omega)\leq 1$. However, no particle can cross an empty site, so that on the event $E_{x,y,z}$ we also have 
\[\alpha_{\Lambda}( \omega(T_K))=\alpha_{\Lambda}(\bar \omega)\leq 1.\]
By definition, on the event $E_{x,y,z}$ we have $\alpha_{\Lambda}(\bar \omega)>1$, because an extra particle is at site $y$, so that, finally for any $z-x> \ell_K$, and any regular configuration $\bar\omega$, $\mathcal{Q}_{K, \bar\omega}\pa{E_{x,y,z}}=0.$

\medskip

We can therefore assume that $z-x\leq \ell_K$. Now set $\Lambda=\{x+1,\ldots,x+\ell_K\}$. On $E_{x,y,z}$, there is at least one free (still able to move)
particle in $\Lambda$. However, on $E_{x,y,z}$ this particle must have remained in $\Lambda$ from time $0$ to $T_K$ because sites $x$ and $z$ are still empty at time $T_K$.
In particular, couple as in Section \ref{sec:chi} on $\Lambda$ the process $\omega$ with the stuck zero-range in $\Lambda$.
Then, by Lemma \ref{lem:ExitTime}
\begin{align*}
\mathcal{Q}_{K, \bar\omega}\big(\mathcal{G}_K \; \cap E_{x,y,z}\big)  \leq  \Q^{\rm st}_{\Lambda,\bar \omega}\big(T_{\chi}^{\Lambda}\geq T_K\big) =\mc O\pa{e^{-\ell_K^{1/8}}}.
\end{align*}
Finally, 
\[\mathcal{Q}_{K, \bar\omega}\big(\mathcal{G}_K\cap E\big)\leq K^3\sup_{x<y<z\in \mathbf{B}_K}\mathcal{Q}_{K, \bar\omega}\big(\mathcal{G}_K\cap E_{x,y,z}\big)=\mc O \pa{e^{-K^{1/16}}},\]
which vanishes as wanted as $K\to\infty$ and proves  Lemma \ref{lem:FK}.

 \subsection{Proof of Theorem \ref{thm:fronts}}\label{s:proofth2}
Let us show point \emph{(1)} of Theorem \ref{thm:fronts}. Choose $t_N=N^{-1/4}$, we first write, using Lemmas~\ref{lem:Typconf} and \ref{lem:UnifBoundomega}
\begin{align*}
\bb P_{\mu^N}(\eta(t_N)&\not\in \mathfrak{P}_N)\\
&=\bb P_{\mu^N}\big(\widetilde{\omega}(t_N)\not\in \mathfrak{P}_{K_0}^{{\rm ZR}}\big)\\
&=\bb P_{\mu^N}\pa{\{K_0\in \mathbf{I}_N\}\cap\{\widetilde{\omega}(0)\in \mathcal{T}_{K_0}\}\cap \widetilde{\mathcal{G}}_N\cap\left\{\widetilde{\omega}(t_N)\not\in \mathfrak{P}_{K_0}^{\rm ZR}\right\}}+o_N(1)\\
&\leq\max_{\substack{K\in \mathbf{I}_N\\
\bar \omega\in \mathcal{T}_K}}\bb P_{\mu^N}\pa{\left. \widetilde{\mathcal{G}}_N\cap\left\{\widetilde{\omega}(t_N)\not\in \mathfrak{P}_{K_0}^{\rm ZR}\right\}\right|K_0=K \mbox{ and }\widetilde{\omega}(0)=\bar \omega}+o_N(1).
\end{align*}
Note that for $N$ large enough, for any $K\in \mathbf{I}_N$,  $T_K=K^{7/4}\leq N^2 t_N$, therefore the probability in the right hand side above is less than $\mathcal{Q}_{K,\bar\omega}\pa{ \mc G_K\cap\left\{\omega(T_K)\not\in \mathfrak{P}_{K}^{\rm ZR}\right\}}$. This yields
\begin{equation*}
\bb P_{\mu^N}(\eta(t_N)\not\in \mathfrak{P}_N)\leq \max_{\substack{K\in \mathbf{I}_N\\
\bar \omega\in \mathcal{T}_K}}\mathcal{Q}_{K,\bar\omega}\pa{ \mathcal{G}_K\cap\left\{\omega(T_K)\not\in \mathfrak{P}_{K}^{\rm ZR}\right\}}+o_N(1).
\end{equation*}
For $K\in \mathbf{I}_N$, $K\to \infty$ as soon as $N\to\infty$. Therefore, letting $N\to\infty$, the right hand side vanishes according to Proposition \ref{prop:HTZR}.

It remains to show point \emph{(2)} of Theorem~\ref{thm:fronts}. Fix $t\in\R_+\cap(0,\tau]$; we give to $u^N_\pm(t)$ the arbitrary value $0$ if the configuration never became two-phased before time $t$ (which, according to point \emph{(1)}, occurs with vanishing probability). Recall Definition \ref{def:frontcreation} and \eqref{eq:two-phasedstable}, which ensures that $u^N_\pm(t)$ are well defined. Let us start with the leftmost interface at position $\uleft(t)$. Note that, by Proposition~\ref{ass:stronguniqueness}, $\uleft$ can be identified without any ambiguity with a continuous non-decreasing function from $\R_+$ to $[0,1)$. We first show that, for any $\varepsilon >0$, and any $t <\tau$,
\[ \P_{\mu^N} \bigg( \frac1N u_-^N(t) -u_-(t)\geq \varepsilon \bigg) \xrightarrow[N\to\infty]{} 0. \]
The other cases can be treated in the exact same way and are left to the reader. In the following we denote 
$I_\varepsilon(t):= [\uleft(t),\uleft(t) + \varepsilon] $ and we take $\varepsilon$ small enough such that $I_\varepsilon(t) \subset [\uleft(t),\uright(t))$. We also introduce its microscopic counterpart 
\[\mathbf{I}_\varepsilon^N(t) = \big\{ \lfloor N \uleft(t) \rfloor, \dots, \lfloor N(\uleft(t)+\varepsilon)\rfloor \big\}.  \]
As before, for the sake of clarity we omit integer parts $\lfloor \cdot \rfloor$ in all that follows.  By definition (recall Definition \ref{def:frontcreation}),
\begin{equation*}
\label{eq:missingproba}
\P_{\mu^N} \bigg( \frac1N u_-^N(t)- \uleft(t) \geq \varepsilon \bigg)  \leqslant \P_{\mu^N} \bigg(  \eta(t)_{| \mathbf{I}_\varepsilon^N(t)} \in \mathcal{E}_{\mathbf{I}_\varepsilon^N(t)}\bigg)+\varepsilon_N, 
\end{equation*}
where the error term $\varepsilon_N$  comes from the (very unlikely) scenario where $\frac{1}{N}u^N_-(t)\geq u_-(t)+\varepsilon\geq \frac{1}{N}u^N_+(t)\geq u_-(t)$. Since according to Lemma \ref{lem:pospart}, $u^N_+(0)$ is at distance at most $\log^2N$ of $Nu_*$, and given the respective monotonicities of $u^N_\pm,u_\pm$,  the error term $\varepsilon_N$ vanishes as $N\to\infty$. 

We now estimate $\P_{\mu^N} \big(  \eta(t)_{| \mathbf{I}_\varepsilon^N(t)} \in \mathcal{E}_{\mathbf{I}_\varepsilon^N(t)}\big)$.
 Let us choose $\varphi:\T\to [0,1]$ as a smooth test function which satisfies \[ \mathbf{1}_{[\uleft(t)+\delta, \uleft(t)+\varepsilon - \delta]} \leqslant \varphi \leqslant \mathbf{1}_{I_\varepsilon(t)},\] where $\delta$ satisfies (recall Proposition~\ref{ass:stronguniqueness})
\begin{equation}\label{eq:mic} \int_\T \rho_t(u) \varphi(u) du \leqslant \int_{I_\varepsilon(t)} \rho_t(u) du < \frac{\varepsilon}{2} - 2\delta  . \end{equation}
Moreover, on the event $\{\eta(t)_{| \mathbf{I}_\varepsilon^N(t)} \in \mathcal{E}_{\mathbf{I}_\varepsilon^N(t)}\}$, we have (recall \eqref{eq:ergset})
 \begin{equation}\label{eq:mac} \frac1N \sum_{x \in \T_N} \varphi\big(\tfrac x N\big) \eta_x(t) \geqslant \frac{1}{N}\sum_{x=  N(\uleft(t) +\delta)}^{N(\uleft(t) +\varepsilon -\delta)} \eta_x(t) \geqslant  \frac{\varepsilon}{2} - \delta.
\end{equation}
We conclude by using Theorem \ref{thm:hydro}.

\appendix

\section{Proof of Proposition \ref{prop:Lebesgue}} \label{app:topo} 

We start by proving tightness of $(\mathcal{P}_N)$ and assertion \emph{(1)}. To do so, first note that since only one particle per site is permitted,
\[\mathcal{P}_N\pa{\sup_{t\geq 0}\; \langle m_t,1\rangle>1}=0,\]
therefore we only need to show (cf.~Theorem 1.3 and Proposition 1.6, p.~51 in  \cite{KL}) that for any limit point $\mathcal{P}^*$, the following is satisfied: for any function $\xi\in C^2(\bb T)$ and for any positive $\varepsilon$,
\begin{equation}
\label{eq:1prime}
\mathcal{P}^*\bigg(\exists\ C(\xi)>0 \text{ s.t. }\; \sup_{|t-s|\leq \varepsilon}\big|\langle m_t,\xi\rangle-\langle m_s,\xi\rangle\big|\leq C(\xi)\varepsilon\bigg)=1.
\end{equation}
To prove \eqref{eq:1prime},  we can rewrite for any fixed $N$, by Dynkin's formula,
\begin{equation}
\label{eq:dynkin1}
\langle m_t^{N},\xi\rangle-\langle m_s^{N},\xi\rangle=\int_s^t N\sum_{x\in \T_N} \xi(\tfrac x N)\mathcal{L}_N\eta_x(\tau) d\tau +M_t^{N,\xi}-M_s^{N,\xi},
\end{equation}
where $M_t^{N,\xi}$ is a martingale w.r.t.~the filtration $\sigma\big( \eta(\tau), \tau \leqslant t\big)$. Since the model is gradient,  and recalling \eqref{eq:defh}, the first term in the right hand side can be rewritten as 
\begin{equation}
\int_s^t \frac 1 N\sum_{x\in \T_N} \Delta^N \xi(\tfrac x N)\tau_xh(\eta(\tau)) d\tau, \label{eq:thisterm}
\end{equation} 
where 
\begin{equation*}
\Delta^N \xi(\tfrac x N)=N^2\big(\xi(\tfrac{x+1} {N})+\xi(\tfrac{x-1} N)-2\xi(\tfrac x N)\big)=\partial_u^2 \xi(\tfrac x N)+o_N(1).
\end{equation*}
Since both $h$ and $\partial^2_u \xi$ are bounded, \eqref{eq:thisterm} is  bounded from above by $C(t-s)$ as wanted. The quadratic variation of the martingale 
$M_t^{N,\xi}$ can be explicitly computed (cf.~\cite[Lemma 5.1, p.~330]{KL}), and is given by
\[ \Big[M^{N,\xi}\Big]_t=N^2\int_0^t \Big( \mathcal{L}_N\big(\langle m_\tau^{N},\xi\rangle\big)^2-2\langle m_\tau^{N},\xi\rangle\mathcal{L}_N\langle m_\tau^{N},\xi\rangle \Big) d\tau=\mc O_N(\tfrac 1 N),\]
where the last estimate comes from elementary and classical computations, using the fact that the function $\xi$ is smooth. In particular, the martingale terms in \eqref{eq:dynkin1} vanish as well, which proves \eqref{eq:1prime} and assertion \emph{(1)}.

We now prove assertion \emph{(2)}, which is immediate because only one particle is allowed per site. This yields in particular that any limit point $\mathcal{P}^*$ of $(\mathcal{P}_N)$ satisfies 
\begin{equation*}
\mathcal{P}^*\pa{\forall\; t \in [0,T], \; \forall\;  \xi \in L^1(\T), \quad \langle m_{t},\xi\rangle\leq \int_\T \xi(u) du}=1,
\end{equation*}
which proves the assertion.

\section{Proof of Lemma \ref{lem:correlations}}
\label{app:tech}
Let us compute explicitly
\[\gene_N\big(\langle m^N,m^N*\xi^N\rangle\big)=\frac{1}{N^2}\sum_{x,y \in \T_N}\xi^N\big(\tfrac y N\big)\gene_N\big(\eta_{x-y}\eta_x\big)=:\mathrm{I}+\mathrm{II},\]
where $\mathrm{I}$ and $\mathrm{II}$  respectively correspond to the cases where $y\notin\{ 1, N-1, N\}$ and $y\in\{1,N-1\}$. Note that the contribution of the terms for $y=N$ vanishes because $\sum_{x\in \T_N}\gene_N\eta^2_x=\gene_N\pa{\sum_{x\in \T_N}\eta_x}=0,$ since the dynamics is conservative.
\medskip

More precisely, shortening $F(x)=\tau_xh(\eta)$, and defining its discrete Laplacian as 
\[ \delta^N F(x):=F(x+1)+F(x-1)-2F(x), \]
elementary computations yield
\begin{align*}
\mathrm{I}&=\frac{1}{N^2}\sum_{x\in\T_N}\sum_{y=2}^{N-2}\xi^N\big(\tfrac y N \big)\pa{\eta_{x-y}\;\delta^N F(x)+\eta_x\;\delta^NF(x-y)} \\
&=\frac{1}{N^2}\sum_{x\in\T_N}\sum_{y=2}^{N-2}\xi^N\big(\tfrac y N \big)\pa{\eta_{x+y}+\eta_{x-y}}\delta^NF(x)
\end{align*}
and 
\begin{equation*}
\mathrm{II}=  \frac{\xi^N(\frac 1 N )+\xi^N(\frac {N-1} N )}{N^2}\sum_{x\in\T_N}\pa{\eta_{x+1}+\eta_{x-1}+\eta_x} \delta^NF(x).
\end{equation*}
Finally, 
\begin{equation*}
\mathrm{I}+\mathrm{II}=\frac{1}{N^2}\sum_{x,y\in \T_N} \xi^N\big(\tfrac y N \big)\pa{\eta_{x+y}+\eta_{x-y}} \delta^NF(x)+\frac{\Delta^N \xi^N(0)}{N^4}\sum_{x\in\T_N}\eta_x \;\delta^NF(x).
\end{equation*}
Successive summations by parts in $x$ and $y$ in the first sum, and in $x$ in the second one, prove the lemma.

\section{Proof of Lemma \ref{lem:Typconf}}
\label{app:typical}
\subsection{Law of large numbers for the positions of zeros}\label{ss:mappings}

We first check that the microscopic and macroscopic mappings defined respectively in \eqref{eq:DefPiNK} and \eqref{eq:Defv1}--\eqref{eq:Defv2} match in the limit.

 For any $k\leq N\bar v$, denote $u_k$ the solution of $N\int_0^{u_k}(1-\rho^{\rm ini}(u))du=k $, and for $k>N\bar v$, we let $u_k=1$. Note in particular that  
\begin{equation}
\label{eq:uk}
u_k=v^{-1}(k/N),
\end{equation}
where the function $v(u)$ was introduced in \eqref{eq:Defv1}. From the law of large numbers, we expect that the $k$--th empty site $y_k\in \T_N$ in the initial configuration $\eta(0)$ should be close to site $Nu_k$.

\begin{lemma}
\label{lem:pospart}
\begin{equation}
\lim_{N\to\infty}\; \bb P_{\mu_N}(E_N)=0,
\label{eq:K0inIN}
\end{equation}
where 
\[ E_N:=\left\{|K_0-\bar v N|\geq \log^2N\quad  \mathrm{or}\quad \max_{k\leq K_0} |y_k-Nu_k|>\log^2 N \right\}.\] 
\end{lemma}

\begin{proof}This estimate is a simple consequence of the facts that \begin{align*}
& \text{if} \quad y_k-Nu_k>\log^2 N \quad \text{then} \quad \sum_{x=0}^{Nu_k+\log^2N}(1-\eta_x(0))\leq k \qquad \mu^N\text{--a.s.},\\
& \text{if}\quad y_k-Nu_k<\log^2 N  \quad \text{then} \quad \sum_{x=0}^{Nu_k-\log^2N}(1-\eta_x(0))\geq k \qquad \mu^N\text{--a.s.},
\end{align*}
together with standard large deviation estimates for sums of independent variables.
\end{proof}
Throughout the rest of the proof, we now assume that $E_N^c$ holds, we are going to show that condition $(i)$ and $(ii)$ of Definition \ref{def:regZR} hold as well for any $N$ large enough, which will prove Lemma \ref{lem:Typconf}.

\subsection{Subcritical phase}

We first deal with condition $(i)$. 
By Assumption \eqref{ass:2reg}, there exists $c_0=c_0(\rho^{\rm ini})>0$ such that for $N$ large enough, for any discrete interval $ \{k_1,\dots, k_2\}\subset \mathbf{B}_{K_0}$ (recall \eqref{eq:DefFKEK}), we have on the event $E_N^c$,  
\begin{equation} 
\label{eq:uk1k2}
[u_{k_1}, u_{k_2+1}]\subset \cro{c_0\frac{\ell_{K_0}}{K_0},u_*-c_0\frac{\ell_{K_0}}{K_0}}.
\end{equation}
Recall that $\widetilde{\omega}(0)$ denotes the zero-range configuration mapped from the initial exclusion configuration $\eta(0)$. Fix $\Lambda=\{k_1,\dots, k_2\}\subset \mathbf{B}_{K_0}$ with cardinality $|\Lambda|=k_2-k_1+1\geq\ell_{K_0}$. By definition \eqref{eq:alphaLambda}, 
\[\alpha_\Lambda(\widetilde{\omega}(0))=\frac{1}{k_2-k_1+1}(y_{k_2+1}-y_{k_1}-(k_2-k_1+1)).\]
On $E_N^c$, we have $y_{k_2+1}-y_{k_1}\leq N(u_{k_2+1}-u_{k_1})+ 2\log^2 N$. Moreover, using \eqref{eq:uk1k2} and Assumption \eqref{ass:2reg}, there exists $c_1=c_1(\rho^{\rm ini})>0$ such that
\[k_2+1-k_1=N\int_{u_{k_1}}^{u_{k_2+1}}(1-\rho^{\rm ini}(u))du\geq N(u_{k_2+1}-u_{k_1})\pa{\frac 12 +\frac{c_1\ell_{K_0}}{K_0}}.\]

Putting those bounds together, we obtain that on $E_N^c$
\[\alpha_\Lambda(\widetilde{\omega}(0)) \leq \bigg(\frac 12 +c_1\frac{(\bar v N-\log^2N)^{3/4}}{\bar v N+\log^2 N}\bigg)^{-1}-1+\frac{2\log^2N}{(\bar v N-\log^2N)^{3/4}}.
\]
For $N$ large enough, the right hand side above is less than $1$, therefore condition $(i)$ of Definition \ref{def:regZR} holds. 

  \subsection{Supercritical phase}
We now prove that condition $(ii)$ of Definition \ref{def:regZR} holds. With $k_*=\frac{K_0v_*}{\bar v}$ (the microscopic site corresponding to the macroscopic critical point $v_*$), note that on the event $E_N^c$, we have \[u_{k_*}\in \big[v^{-1}(v_*-\log^2 N/N),v^{-1}(v_*+\log^2 N/N)\big].\] Therefore, by Assumption \eqref{ass:2reg}, there exists $c_2=c_2(\rho^{\rm ini})$ such that
$|u_{k_*}-u_*|\leq c_2\log^2 N/N$. To prove that condition $(ii)$  holds, we need to consider sites in 
\[A_{K_0}= \{0,\dots,\ell_{K_0}\} \cup\left\{k_*-\ell_{K_0},\dots, K_0\right\}.\]
The case where $x_0$ sits in the bulk of $A_{K_0}$, \textit{i.e.}~when there is a macroscopic region around $x_0/N$ in which the density in $\tilde\omega(0)$ is larger that $1+\epsilon$ for some $\epsilon>0$, follows from the same arguments as  in \cite[Lemma 4.10]{BESS}. This is the easiest case and we do not detail it here. The hardest case is when $x_0$ is close to the interface, so the density around is close to $1$ and particles are not as much in excess.
To avoid burdensome notations, we will only detail the proof that site $x_0:=k_*-\ell_{K_0}$ satisfies $(ii)$, which would adapt straightforwardly to other sites in $A_{K_0}$. 
We therefore prove that there exists $\tau_{x_0}\omega\in \Omega^+_{K_0}$ such that $\omega\leq \widetilde{\omega}(0)$. 

\begin{lemma}
Let $\Lambda^+:=\{x_0+1,\dots,x_0+10\ell_{K_0}\}$ and $G_N=\{\alpha_{\Lambda^+}(\tilde\omega(0))\geq \alpha_{K_0}\}.$ Then 
\[\lim_{N\rightarrow\infty}\P_{\mu_N}(G_N)=1.\]
\end{lemma}
\begin{proof} From \eqref{eq:K0inIN}, it is enough to prove that $\P_{\mu_N}(G_N \cap E_N^c) \to 1$ as $N \to \infty$. There exists $C>0$ such that, on $E_N^c$
\begin{align*} &  \Big| y_{x_0} -  N v^{-1} \big(v_* - \ell_{\bar v N}/N\big)\Big| \leqslant C \log^2(N) , \\
& \Big| y_{x_0+10\ell_{K_0}+1} -  N v^{-1} \big(v_* + 9\ell_{\bar v N}/N\big)\Big| \leqslant C \log^2(N) . \end{align*}
Furthermore, 
\[\alpha_{\Lambda^+}(\tilde\omega(0))=\frac{y_{x_0+10\ell_{K_0}+1}-y_{x_0}-10\ell_{K_0}}{10\ell_{K_0}},\]
and, developing the function $v^{-1}$ at $v_*$, we obtain
\begin{equation}
\label{eq:taylorv1}
v^{-1}(a)=u^*+2(a-v_*)+4\partial_u\rho^{\rm ini}(u_*)(a-v_*)^2+\mc O((a-v_*)^3).
\end{equation}
Recall that we must have $\bar v\geq \frac12$, the four equations and bounds above together yield on $E_N^c$, since $K_0\geq N/2+\mc O(\log^2 N)$ 
\begin{align*}\alpha_{\Lambda^+}(\tilde\omega(0))&=1+24\partial_u\rho^{\rm ini}(u^*)\frac{\ell_{K_0}}{N}+\mc O((\ell_{K_0}/N)^2)\\
&\geq 1+12\partial_u\rho^{\rm ini}(u^*)\frac{\ell_{K_0}}{K_0}+\mc O((\ell_{K_0}/N)^2)\\
&\geq \alpha_{K_0}\end{align*}
for $N$ large enough.
\end{proof}

Assuming we are on $G_N$, we define $\omega$ by keeping from $\widetilde{\omega}(0)$ only the $\widehat{n}:=10\ell_{K_0}\alpha_{K_0}$ particles closest to site  $x_0$ in $\Lambda^+$. This configuration obviously satisfies both $\sum_{x\in\Lambda^+}\omega_x=|\Lambda^+|\alpha_{K_0}$ and $\omega\leq \widetilde{\omega}(0)$, so that we only need to check 
\[\sum_{k\in\Lambda^+}\omega_k\pa{k-5\ell_{K_0}-x_0}\leq 0,\]
\emph{i.e.}~that those particles are on average closer to $x_0$ than they are to the other extremity of $\Lambda^+$. 
Denote $\hat{k}>x_0$ the zero-range site where the $\widehat{n}$-th particle to the right of $x_0$ is found (\textit{i.e.}~in the exclusion configuration, there are $\hat{k}$ empty sites between the empty site $y_{x_0}$ and the $\widehat{n}$-th particle to the right of $y_{x_0}$).  

To prove $(ii)$ it is enough to show that with probability going to $1$, one has
\begin{equation}
\label{eq:omegatilde}
\sum_{k=x_0}^{\hat{k}}\widetilde{\omega}_k(0)\pa{k-5\ell_{K_0}-x_0}\leq0
.\end{equation}
We are on the event $E_N^c$, therefore each empty site is at most at a distance $\log^2 N$ of its expected position in the initial exclusion configuration. Since $\widetilde{\omega}_k(0)=y_{k+1}-y_{k}-1$, 
\[\sum_{k=x_0}^{\hat{k}-1}(y_{k+1}-y_k-1) \leq \widehat{n}\leq\sum_{k=x_0}^{\hat{k}}(y_{k+1}-y_k-1),\]
which rewrites, on $E_N^c$,
\begin{equation}
\label{eq:khat1}
0\leq \widehat{n}-(y_{\hat{k}}-y_{x_0}-(\hat{k} - x_0))\leq y_{\hat k+1}-y_{\hat k}\leq 2\log^2 N+N\sup_k(u_{k+1}-u_k).
\end{equation}

Recall that  $ \ell_{K_0}\leq  \ell_N=N^{3/4}$, using  \eqref{eq:uk} and \eqref{eq:taylorv1} yields that for any $k\in \Lambda^+\cup{\{x_0\}}$,
\[u_k=u_*+2\Big(\frac{k}{N}-v_*\Big)+4\partial_u\rho^{\rm ini}(u^*)\Big(\frac{k}{N}-v_*\Big)^2+\mc O(N^{-3/4}).\]
For any integer $j$, shorten $j'=j-k_*=j-Nv_*+\mc O(\log^2 N)$, on $E_N^c$, the identity above yields 
\begin{equation}
\label{eq:khat2}
 y_k-y_{k_*}=2 k' + c_0\frac{ {k'}^2}{N}+\mc O(N^{1/4 }),
\end{equation}
where $c_0=4\partial_u \rho^{\rm ini}(0)>0$ by Assumption \eqref{ass:2reg}. 

Using \eqref{eq:khat2}, we now rewrite the left hand side of \eqref{eq:omegatilde},
\begin{align}
\label{eq:estimdrift}
\sum_{k=x_0}^{\hat{k}}&\widetilde{\omega}_k(0)\pa{k-5\ell_{K_0}-x}
= \sum_{k=x_0}^{\hat{k}}(y_{k+1}-y_{k}-1)\pa{k'-5\ell_{K_0}-x_0'}\nonumber\\
&= \sum_{k=x_0}^{\hat{k}}k'(y_{k+1}-y_{k}-2)-(5\ell_{K_0}+x_0')\widehat{n}+\frac{(\widehat{k}'-x_0')(\widehat{k}'+x_0')}{2}+\mc O(\ell_N)\nonumber\\
&=\frac{c_0}{N} \sum_{k=x_0'}^{\hat{k}'}k(2k+1)-(5\ell_{K_0}+x_0')\widehat{n}+\frac{(\widehat{k}'-x_0')(\widehat{k}'+x_0')}{2}+\mc O(\ell_N)\nonumber\\
&=\frac{2c_0}{N} \sum_{k=x_0'}^{\hat{k}'}k^2-(5\ell_{K_0}+x_0')\widehat{n}+\frac{(\widehat{k}'-x_0')(\widehat{k}'+x_0')}{2}+\mc O(\ell_N).
\end{align}
Note that by definition, $x_0'=-\ell_{K_0}$. Using equations \eqref{eq:khat1} and \eqref{eq:khat2}, one can easily check that $\widehat{n}=(\widehat{k}'-x_0')(1+c_0(\widehat{k}'+x_0')/N)+\mc O(N^{1/4})$, so that by Taylor expansion
\[\widehat{k}'=x_0'+\widehat{n}-\frac{c_0\widehat{n}}{N}(2x_0'+\widehat{n})+\mc O(N^{1/4})=9\ell_{K_0}-c_1\frac{\ell_{K_0}^2}{K_0}+\mc O(N^{1/4}),\]
where we denoted $c_1=80c_0\bar v-10c_*>70c_0\bar v$ by definition \eqref{eq:DefC*} of $c_*$. After elementary computations, the dominant terms of order $\mc O(\ell_{K_0}^2)$ in \eqref{eq:estimdrift} cancel out, so that only remain the terms in $\mc O(\ell_{K_0}^3/K_0)$, which rewrite 
\begin{align*}
\sum_{k=x_0}^{\hat{k}}\widetilde{\omega}_k(0)\cro{k-5\ell_{K_0}-x}&=\frac{2c_0\bar v}{3K}\pa{\widehat{k}^{'3}-x^{'3}}-(4c_*+9c_1)\frac{\ell_{K_0}^3}{K_0}+\mc O(\ell_N)\\
&\leq c_0\bar v(2\cdot 3^8+1/3-9\cdot 70  )\frac{\ell_{K_0}^3}{K_0}+\mc O(\ell_N).
\end{align*}
Since the constant in parenthesis is negative, and since $\ell_{K_0}^3/K_0\gg\ell_N$, this proves the result. 

As already pointed out, we will not detail the general cases $x\in \mathbf{A}_{K_0}$, we simply sketch out why the problem is the same. Consider a macroscopic point $v\in[v_*, \bar v]$ and consider the zero-range configuration in a mesoscopic box of size $10\ell_{K_0}$ to the right of site $Kv/\bar v$. If $v\in(v_*, \bar v)$ then the zero-range density $\alpha_v=1/(1-\rho^{\rm ini}(u(v)))$ in a mesoscopic box $\Lambda_{v,K_0}$ of size $10\ell_{K_0}$ is strictly larger than $1$. In particular, since on $E_N^c$, $\alpha_{K_0}\simeq \alpha_{\bar v N}=1+o(1)$, the majority of the $\widehat{n}=10\ell_{K_0}\alpha_{K_0}$ particles closest to $x$ to its right are closer to $x$ than to the other extremity of $\Lambda_{v,K_0}$. The only problematic cases are therefore close to the boundaries $0$, $v_*$. We treated the most extreme of those cases, in which the site $x$ considered is in the subcritical phase, and at a distance $\ell_{K_0}$ of the supercritical phase, the other cases can be treated analogously.

\section{Existence of macroscopic interfaces: proof of Proposition~\ref{ass:stronguniqueness}}\label{a:strong-existence}

In this section, we lay out the proof for the existence of macroscopic interfaces for the weak solution (in the sense of Definition \ref{Def:PDEweak}) of \eqref{eq:PDEstrong}. 
The proof we present here is adapted from Meirmanov \cite{Meirmanov} to our periodic setting. It contains no significant mathematical novelty w.r.t.~\cite{Meirmanov}; we include it here for the sake of completeness. The main difficulty of the proof is that the interface speeds diverge as $t\rightarrow 0^+$. To solve this issue, we approximate the initial profile $\rho^{\rm ini}$ by 
\[\rho^{{\rm ini},n}(u)=\rho^{\rm ini}(u)\cro{1-\frac{1}{n}{\bf 1}_{\{\rho^{\rm ini}(u)<\frac12\}}},\ n\in\N, n \geqslant 3.\]
We first claim that, thanks to the discontinuity of the density at the interfaces, the Stefan problem with initial condition $\rho^{{\rm ini},n}$ admits a classical solution.

\begin{lemma}\label{lem:strong-existence-disc}
Let $\tilde\rho^{{\rm ini}}:\T\rightarrow [0,1]$ such that 
\begin{itemize}
\item  $\tilde\rho^{{\rm ini}}$ is $\mathcal{C}^2$ on $(0,u_*)$ and $[u_*,0]$ with bounded derivatives, 
\item  $\tilde\rho^{{\rm ini}}\leq \frac12-\delta$ on $(0,u_*)$ for some $\delta>0$, 
\item  $\tilde\rho^{{\rm ini}}\geq \frac12$ on $[u_*,0]$. 
\end{itemize}

Then  there exists a classical solution   $(\rho,u_-, u_+)$  to the Stefan problem \eqref{eq:PDEstrong} with initial data $\tilde\rho^{{\rm ini}}$, i.e.\ $\rho:\R_+\times\T\rightarrow [0,1]$, $u_\pm:\R_+\rightarrow\T$ such that
\begin{enumerate}
\item $u_-$ (resp.\ $u_+$) is non-decreasing (resp.\ non-increasing), with $u_-(0)=0$ and $u_+(0)=u_*$;
\item there exists $\tau\in\R_+\cup\{\infty\}$ such that $u_-(t)=u_+(t)$ iff $t\geq \tau$, and $u_\pm$ are constant on $[\tau,\infty)$; the time $\tau$ is called the \emph{merging time};
\item $\mathcal{H}\circ \rho$ is Lipschitz;
\item  for any $t \in \R_+$, if $u\in (u_-(t),u_+(t))$, then $\rho_t(u)=\tilde\rho^{{\rm ini}}(u)$;
\item if $\big\{t\leq \tau$ and $u\in (u_+(t),u_-(t)) \big\}$ or if $t>\tau$, then 
\[ \rho_t(u)>\tfrac12, \qquad \text{and} \qquad \partial_t \rho_t(u)=\partial_u^2 \mathcal{H}(\rho_t(u));\]
\item if $t\in(0,\tau)$, then \[u_\pm'(t)=-\frac{4\partial_u\rho_t(u_\pm(t)^\pm)}{\frac12-\tilde\rho^{{\rm ini}}(u_\pm(t))}.\]  
\end{enumerate}
\end{lemma}

We defer the proof of Lemma~\ref{lem:strong-existence-disc} to see how this result can lead us to Proposition~\ref{ass:stronguniqueness}.
Let us denote $(\rho^n,u^n_-, u^n_+)$ the\footnote{It is easy to check that classical solutions are also weak solutions in the sense of Definition~\ref{Def:PDEweak}, and the uniqueness of $\rho^n$ is therefore guaranteed by Proposition~\ref{prop:uniqueness}.} classical solutions provided by Lemma~\ref{lem:strong-existence-disc} when $\tilde\rho^{{\rm ini}}=\rho^{{\rm ini},n}$. 
We will show that the interfaces $u_\pm^n$ converge, and that the limits satisfy the properties required in Proposition~\ref{ass:stronguniqueness}. To that end, we exploit a monotonicity property of the interfaces defined by Lemma~\ref{lem:strong-existence-disc}:

\begin{lemma}\label{lem:monotone}
Let $\rho^{{\rm ini},>},\rho^{{\rm ini},<}$ two initial profiles satisfying the assumptions of Lemma~\ref{lem:strong-existence-disc}, such that $\rho^{{\rm ini},<}\leq\rho^{{\rm ini},>}$. 

Let $(\rho^>,u^>_-, u^>_+),(\rho^<,u^<_-, u^<_+)$ be the associated classical solutions with merging times $\tau^>,\tau^<$ respectively. 

Then $\rho^<\leq\rho^>$ and $[u^>_-,u^>_+]\subset[u^<_-,u^<_+]$ (in particular $\tau^>\leq\tau^<$).
\end{lemma}

A consequence of Lemma~\ref{lem:monotone} is that $\rho^n,u^n_-, u^n_+$ are monotone in $n$. Since they are also bounded, they have limits which we call $\rho,u_-,u_+$ respectively. The monotone convergence Theorem straightforwardly yields that   $\rho$, thus defined, is the weak solution of \ref{Def:PDEweak} with initial profile $\rho^{{\rm ini}}$. Letting $\tau:=\inf\{t\geq 0\ \colon\ u_+(t)=u_-(t)\}$, the properties of Proposition~\ref{ass:stronguniqueness} are simple consequences of the above construction. Following \cite[Theorem 2, p.~151]{Meirmanov}, one can actually show that $(\rho,u_-, u_+)$ is also a classical solution with initial profile $\rho^{{\rm ini}}$, however since we do not require it here, we will not expand further.

\medskip

We now conclude by giving the proofs of Lemmas~\ref{lem:strong-existence-disc} and \ref{lem:monotone}.

\begin{proof}[Proof of Lemma~\ref{lem:strong-existence-disc}]
This is very close to \cite[Lemma 3, p.151]{Meirmanov} and seems to be a standard result for free boundary problems. We sketch here a proof for completeness, mainly taken from \cite{And} and adapted to our periodic setting. Part of the statement is that the derivatives in \emph{(5)} and \emph{(6)} are well defined. The main idea is to construct the interfaces as solutions to a fixed point problem.

Fix $T>0$ and let $M=\sup\{|(\tilde\rho^{{\rm ini}})'(u)|,u\in\T\setminus\{0,u_*\}\}$. Let $\mathfrak{U}$ be the space of functions $u_-,u_+ : [0,T] \to \T $ which satisfy the following conditions: \begin{itemize}
\item $u_\pm$ are Lipschitz-continuous, with Lipschitz constant bounded by $M$,
\item $u_-$ (resp.~$u_+$) is non-decreasing (resp.~non-increasing),
\item $u_-(0)=0$, $u_+(0)=u_*$, and $u_-\leq u_+$. \end{itemize}
Any such function is differentiable almost everywhere in $[0,T]$. With a slight abuse of notations, we denote by  $\|u'_\pm\|_\infty\leq M$ the maximal Lipschitz constant of $u_\pm$. 
Note that $\mathfrak{U}$ is a convex compact subset of the Banach set $\mathcal{C}([0,T])\times \mathcal{C}([0,T])$ endowed with the norm 
\[\vertiii{u_-, u_+\vphantom{\big(}}:=\max \big\{ \|u_+\|_\infty, \|u_-\|_\infty, \|u'_+\|_\infty, \|u'_-\|_\infty\big\}.\]
For any $(u_-,u_+)\in\mathfrak{U}$, let \[\tau:=\inf\{t\geq 0\ \colon\ u_-(t)=u_+(t)\}\] and $\rho:[0,T]\times \T\rightarrow [0,1]$ be defined as follows: first, $\rho_0(u) = \tilde\rho^{\rm ini}(u)$ for any $u \in\T$, and \begin{itemize}
\item if $t \leqslant \tau$, 
\begin{equation}
\label{eq:Dirichletsolution}
\left\{\begin{split}
\forall\; u \in (u_+(t),u_-(t)), & \quad \partial^2_u\mathcal H(\rho_t(u))=\partial_t\rho_t(u)\\
\forall\; u \in (u_-(t),u_+(t)),&\quad \rho_t(u)=\tilde\rho^{{\rm ini}}(u) \\
\text{and }&\quad \rho_t(u_\pm(t))=\tfrac12
\end{split}\right..
\end{equation}
\item if $t >\tau$, 
\[ \forall\; u \in \T, \quad \partial^2_u\mathcal H(\rho_t(u))=\partial_t\rho_t(u)\; .\]
\end{itemize}
Note that up to time $\tau$, assuming $u_\pm$ are fixed, $\rho$ is the solution to a Dirichlet problem \eqref{eq:Dirichletsolution} with moving boundaries. It is then standard to show (see \cite[Lemma 4.1]{And} for instance) that $\rho $ is well defined, and that its spatial derivatives are continuous up to the boundaries $\{(t,u_\pm(t)), t\in(0,T)\}$. Consider the transformation $\mathcal{T}:\mathfrak{U}\to\mathfrak{U}$ defined as follows: first, let 
\begin{align*}
\mathcal{T}^1(u_-,u_+)(t) & : = -\int_0^t \frac{4\partial_u\rho_s(u_-(s)^-)}{\frac12-\tilde\rho^{{\rm ini}}(u_-(s))}ds,\\
\mathcal{T}^2(u_-,u_+)(t) & := u_*-\int_0^t \frac{4\partial_u\rho_s(u_+(s)^+)}{\frac12-\tilde\rho^{{\rm ini}}(u_+(s))}ds,
\end{align*} 
and define \[\tau^*:= \inf \big\{ t \geqslant 0 \; : \; \mathcal{T}^1(u_-,u_+)(t) = \mathcal{T}^2(u_-,u_+)(t)\big\}.\] 
 Then, let
\begin{equation*}
\mathcal{T}(u_-,u_+)(t)=\left\{
\begin{split}
\left(\mathcal{T}^1(u_-,u_+)(t),\;\mathcal{T}^2(u_-,u_+)(t)\right)&\quad\mbox{ for } t\leq \tau^*\\
(u_-,u_+)(t)\hspace{9em}&\quad\mbox{ for } t>\tau^* \vphantom{\bigg(}
\end{split}\right..
\end{equation*}
Then a fixed point for this transformation also yields the desired classical solution to our Stefan problem. By Schauder's fixed point Theorem, it is therefore enough to show that $\mathcal T$ is continuous w.r.t.~$\vertiii{\vphantom{A}\cdot}$. In turn, by regularity of $\tilde\rho^{{\rm ini}}$, and since $\tilde\rho^{{\rm ini}}$ is bounded away from $\frac12$ in $(u_-, u_+)$, it is enough to show that the application 
\[
(u_-,u_+)\in\mathfrak U\mapsto\left\{t\in[0,T] \mapsto\left(\int_0^t\partial_u\rho_s(u_-(s)^-)ds,\int_0^t\partial_u\rho_s(u_+(s)^+)ds\right)\right\}
\]
is continuous.

To that aim, let $t\leq T$, fix $(u_-,u_+)\in \mathfrak{U}$, and define 
\begin{equation}
\label{eq:gtu}
g_t(u):=\frac{u-u_+(t)}{u_-(t)-u_+(t)}{\bf 1}_{\{u\in [u_+(t),u_-(t)]\}}\in [0,1],
\end{equation}
where $u_-(t)-u_+(t)\in[1-u_*, 1]$ is the length of the diffusive phase, so that in particular $g_0(u)=\frac{u-u_*}{1-u_*}{\bf 1}_{\{u\in [u_*,1]\}}$. Since $g_s(u_+(s))=0$, $g_s(u_-(s))=1$,  by the divergence (or \emph{Gauss--Ostrogradsky}) Theorem,
\begin{align*}0=&\int_0^t\int_{u_+(s)}^{u_-(s)}g_s(u)\big[\partial_t\rho_s(u)-\partial^2_u\mathcal{H}(\rho_s(u))\big]duds\\
=&-\int_0^t\partial_u\mathcal{H}(\rho_s(u_-(s)^-))ds+\int_0^t\int_{u_+(s)}^{u_-(s)}\partial_ug_s(u) \partial_u\mathcal{H}(\rho_s(u))duds\\
& -\int_0^t\rho_s(u_-(s))u'_-(s)ds-\int_0^t\int_{u_+(s)}^{u_-(s)}\rho_s(u)\partial_tg_s(u)duds\\
&-\int_{u_*}^0g_0(u) \tilde\rho^{{\rm ini}}(u)du
+\int_{u_+(t)}^{u_-(t)}g_t(u)\rho_t(u)du.\end{align*}
Since $\mathcal{H}(\rho_s(u_+(s)))=\mathcal{H}(\rho_s(u_-(s)))=0$ and $\partial_u^2 g_s(u)=0$, a second integration by parts shows that the second term in the right hand side vanishes. Consequently, recalling that $\partial_u\mathcal{H}(\rho_s(u)) = 4 \partial_u \rho_s(u)$ and $\rho_s(u_-(s))=\frac12,$ we have 
\begin{align} 
\int_0^t\partial_u\rho_s(u_-(s)^-)ds=&-\tfrac14\int_{u_*}^0g_0(u)\tilde\rho^{{\rm ini}}(u)du+\tfrac14\int_{u_+(t)}^{u_-(t)}g_t(u)\rho_t(u)du \notag\\
&-\int_0^t\int_{u_+(s)}^{u_-(s)}\rho_s(u)\partial_tg_s(u)duds
-\tfrac18 u_-(t)\label{eq:jminus}.
\end{align}
First, we prove that  $\rho_t(u)$ is continuous w.r.t $\vertiii{\vphantom{A}\cdot}$.  
Given $(u_-,u_+),(\tilde u_-,\tilde u_+)$ two elements of $\mathfrak U$ and denoting by $(\rho,\tau),(\tilde\rho,\tilde \tau)$ the associated solutions to \eqref{eq:Dirichletsolution}, we first claim that, assuming for example $u_-(s)\leq \tilde u_-(s)$, 
\[\big|\rho_s(u_-(s))-\tilde\rho_s(u_-(s))\big|=\big|\tilde\rho_s(\tilde u_-(s))-\tilde\rho_s(u_-(s))\big|\leq M\big|u_-(s)-\tilde u_-(s)\big|.\]
The first identity follows from the fact that $\rho_s(u_-(s))=\tilde\rho_s(\tilde u_-(s))=\frac12$, whereas the second follows from the maximum principle applied to $\partial_u \tilde\rho$ in the moving boundary domain $\{(t,u), \; t\leq T, u\in[u_+(t), u_-(t)]\}$. We can now apply the maximum principle to $\rho-\tilde \rho$ in the domain  
\[\Lambda_t:=\Big\{(s,u) \colon s\in(0,t), u\in(u_+(s),u_-(s))\Big\}\cap \Big\{(s,u) \colon s\in(0,t), u\in(\tilde u_+(s),\tilde u_-(s)\Big\},\] 
to obtain that for all $(s,u)\in \Lambda_t$, 
\[\big|\rho_s(u)-\tilde \rho_s(u)\big|\leq M\max\big\{\big|u_-(s)-\tilde u_-(s)\big|,\big|u_+(s)-\tilde u_+(s)\big|\big\}.\] 
Denote by $g$, $\tilde g$ the functions given by \eqref{eq:gtu} resp.~for $(u_-,u_+)$, $(\tilde u_-,\tilde u_+)$. In particular, since $|\rho_s(u)|\leq1$, $|g_s(u)|\leq1$ and $g_s(u)$ is uniformly continuous in $(u_-, u_+)$, we obtain as wanted that 
\begin{multline*}
\abs{\int_{u_+(t)}^{u_-(t)}g_t(u)\rho_t(u)du-\int_{\tilde u_+(t)}^{\tilde u_-(t)}\tilde g_t(u)\tilde\rho_t( u)du}\\
\leq M'\max\big\{\big\|u_--\tilde u_-\big\|_\infty,\big\|u_+-\tilde u_+\big\|_\infty\big\},
\end{multline*}
so that $\int_{u_+}^{u_-}g_t(u)\rho_t(u)du$ is continuous in $(u_-,u_+)$ w.r.t. $\vertiii{\vphantom{A}\cdot}$. Since $\partial_t g$ is also continuous w.r.t. $\vertiii{\vphantom{A}\cdot}$, one obtains straightforwardly that  $\int_0^t\int_{u_+(s)}^{u_-(s)}\rho_s(u)\partial_tg_s(u)duds$ also is. This, together with \eqref{eq:jminus}, proves that $ \int_0^t\partial_u\rho_s(u_-(s)^-)ds$ is continuous in $(u_-,u_+)$. An analogous argument with 
\begin{equation}
\label{eq:gtubis}
g_t(u)=\frac{u_-(t)-u}{u_-(t)-u_+(t)}{\bf 1}_{\{u\in [u_+(t),u_-(t)]\}},
\end{equation}
proves that  $\int_0^t\partial_u\rho_s(u_+(s)^+)ds$ also is, and concludes the proof.
\end{proof}

\begin{proof}[Proof of Lemma~\ref{lem:monotone}]
This is  a simpler case of \cite[Theorem 10, p.30]{Meirmanov}, we give it for the sake of completeness. Define $\bar \rho^{{\rm ini}}=\rho^{{\rm ini},>}-\rho^{{\rm ini},<}$, as well as 
\[\bar \rho_t=\rho_t^{>}-\rho_t^{<},\quad \mbox{ and  }\quad \chi_t=\frac{\mathcal{H}(\rho^>_t)-\mathcal{H}(\rho^<_t)}{\rho_t^>-\rho_t^<}{\bf 1}_{\{\rho_t^>\neq  \rho_t^<\}}.\]
Since classical solutions of \eqref{eq:PDEstrong} are also weak solutions, for any smooth function  $\varphi\in C^{1,2}([0,T]\times \R)$, we have  
\begin{equation}
\label{eq:deltarho}
\big\langle \bar \rho_T,\varphi_T \big\rangle = \big\langle \bar \rho^{\rm ini}, \varphi_0 \big\rangle+\int_0^T\big\langle \bar \rho_t,\partial_t\varphi_t + \chi_t \partial_u^2\varphi_t \big\rangle dt.
\end{equation}
Fix $T>0$ and $\varepsilon>0$, and a bounded non-negative function $g:\T\to[0,+\infty)$,  we define $\psi^\varepsilon$ as the classical solution to the elliptic equation
\begin{equation}
\label{eq:EDPpsi}
\begin{cases}
\partial_t  \psi^\varepsilon_t=\big(\chi_{T-t}+\varepsilon\big) \partial_u^2 \psi^\varepsilon_t\\
\psi^\varepsilon_0=g
\end{cases}.
\end{equation}
Since the initial profile $g$ is non-negative, by maximum principle so is $\psi^\varepsilon_t$ for any $t\leq T$, so that \eqref{eq:deltarho} yields, choosing  $\varphi_t=\psi^\varepsilon_{T-t}$, 
\begin{equation}
\label{eq:holderpsirho}
\big\langle \bar \rho_T,g \big\rangle \geq - \varepsilon\int_0^T\big\langle\bar \rho_t, \partial_u^2\psi_{T-t} \big\rangle dt.
\end{equation}
where we used that $\bar \rho^{{\rm ini}}\geq 0$. Assume now that $g\in C^2(\T)$, multiplying the first line of \eqref{eq:EDPpsi} by $\partial_u^2 \psi^\varepsilon_t$, and integrating over $[0,T]\times \T$, yields 
\[\frac12\int_\T ( \partial_u g)^2 du =\frac12\int_\T (\partial_u \psi^\varepsilon_T)^2du
+\int_0^T\int_\T \big(\chi_{T-t}+\varepsilon\big) (\partial_u^2 \psi^\varepsilon_t )^2dudt,
\]
so  that in particular
\[\varepsilon\big\langle \partial_u^2 \psi^\varepsilon_T ,\partial_u^2 \psi^\varepsilon_T \big\rangle \leq \frac12\big\langle \partial_u g,\partial_u g \big\rangle, \]
and by H\"older's inequality, \eqref{eq:holderpsirho} yields 
\[
\big\langle\bar  \rho_T,g \big\rangle \geq - \sqrt{2\varepsilon}\int_0^T \big\langle\bar \rho_t,\bar \rho_t\big\rangle ^{1/2} \big\langle\partial_u g,\partial_u g\big\rangle^{1/2} dt\geq -T \sqrt{2\varepsilon} \big\langle\partial_u g,\partial_u g\big\rangle^{1/2}.
\]
Letting  $ \varepsilon\to 0$, we obtain that $\big\langle\bar  \rho_T,g \big\rangle\geq 0$ for any non-negative $g\in C^2(\T)$. We now choose non-negative functions $g^k\in C^2(\T)$ converging  in $L^2(\T)$ to ${\bf 1}_{\{\bar \rho_T<0\}}$ as $ k\to\infty$, to obtain that $\bar \rho_T\geq 0$ a.e., which concludes the proof.
\end{proof}

\bibliography{bibliography}

\begin{thebibliography}{10}

\bibitem{Andjel}
E.~D. Andjel.
\newblock Invariant measures for the zero range processes.
\newblock {\em Ann. Probab.}, 10(3):525--547, 1982.

\bibitem{And}
D.~Andreucci.
\newblock Lecture notes on the {S}tefan problem.
\newblock 2002.

\bibitem{BBCS}
J.~Baik, G.~Barraquand, I.~Corwin, and T.~Suidan.
\newblock {\em Facilitated Exclusion Process: The Abel Symposium, Rosendal,
  Norway, August 2016}, pages 1--35.
\newblock 01 2018.

\bibitem{BM2009}
U.~Basu and P.~K. Mohanty.
\newblock Active--absorbing--state phase transition beyond directed
  percolation: A class of exactly solvable models.
\newblock {\em Phys. Rev. E}, 79:041143, Apr 2009.

\bibitem{Bill}
P.~Billingsley.
\newblock {\em Probability and measure}.
\newblock Wiley Series in Probability and Mathematical Statistics. John Wiley
  \& Sons, Inc., New York, third edition, 1995.
\newblock A Wiley-Interscience Publication.

\bibitem{BCSS}
O.~Blondel, C.~Canc\`es, M.~Sasada, and M.~Simon.
\newblock {Convergence of a degenerate microscopic dynamics to the porous
  medium equation}.
\newblock {\em Arxiv:1802.05912}, to appear in Ann. Inst. Fourier, 2018.

\bibitem{BESS}
O.~Blondel, C.~Erignoux, M.~Sasada, and M.~Simon.
\newblock {Hydrodynamic limit for a facilitated exclusion process}.
\newblock {\em Ann. Inst. H. Poincar\'{e} Probab. Statist.}, 56(1):667--714,
  2020.

\bibitem{CMRT}
N.~Cancrini, F.~Martinelli, C.~Roberto, and C.~Toninelli.
\newblock Kinetically constrained lattice gases.
\newblock {\em Communications in Mathematical Physics}, 297(2):299--344, Jul
  2010.

\bibitem{CDMGP}
G.~Carinci, A.~De~Masi, C.~Giardin\`a, and E.~Presutti.
\newblock {\em Free boundary problems in {PDE}s and particle systems},
  volume~12 of {\em SpringerBriefs in Mathematical Physics}.
\newblock Springer, [Cham], 2016.

\bibitem{CS96}
L.~Chayes and G.~Swindle.
\newblock Hydrodynamic limits for one-dimensional particle systems with moving
  boundaries.
\newblock {\em Ann. Probab.}, 24(2):559--598, 1996.

\bibitem{DKor}
D.~Danielli and M.~Korten.
\newblock On the pointwise jump condition at the free boundary in the 1-phase
  {S}tefan problem.
\newblock {\em Commun. Pure Appl. Anal.}, 4(2):357--366, 2005.

\bibitem{demasietal}
A.~De~Masi, T.~Funaki, E.~Presutti, and M.~E. Vares.
\newblock Fast-reaction limit for {G}lauber-{K}awasaki dynamics with two
  components.
\newblock {\em ALEA Lat. Am. J. Probab. Math. Stat.}, 16(2):957--976, 2019.

\bibitem{O}
M.~J. de~Oliveira.
\newblock Conserved lattice gas model with infinitely many absorbing states in
  one dimension.
\newblock {\em Phys. Rev. E}, 71:016112, Jan 2005.

\bibitem{Delarue}
François Delarue, Sergey Nadtochiy, and Mykhaylo Shkolnikov.
\newblock Global solutions to the supercooled stefan problem with blow-ups:
  regularity and uniqueness.
\newblock {\em arxiv:1902.05174}, 02 2019.

\bibitem{DT19}
Amir Dembo and Li-Cheng Tsai.
\newblock The criticality of a randomly-driven front.
\newblock {\em Archive for Rational Mechanics and Analysis}, 233, 02 2019.

\bibitem{Funaki}
T.~Funaki.
\newblock Free boundary problem from stochastic lattice gas model.
\newblock {\em Ann. Inst. H. Poincar\'{e} Probab. Statist.}, 35(5):573--603,
  1999.

\bibitem{georgii}
H.-O. Georgii.
\newblock {\em Gibbs measures and phase transitions}, volume~9 of {\em De
  Gruyter Studies in Mathematics}.
\newblock Walter de Gruyter \& Co., Berlin, 1988.

\bibitem{GLT}
P.~Gon\c{c}alves, C.~Landim, and C.~Toninelli.
\newblock Hydrodynamic limit for a particle system with degenerate rates.
\newblock {\em Ann. Inst. Henri Poincar\'{e} Probab. Stat.}, 45(4):887--909,
  2009.

\bibitem{GQ}
J.~Gravner and J.~Quastel.
\newblock Internal {DLA} and the {S}tefan problem.
\newblock {\em Ann. Probab.}, 28(4):1528--1562, 2000.

\bibitem{hayashi}
Kohei {Hayashi}.
\newblock {Spatial-segregation limit for exclusion processes with two
  components under unbalanced reaction}.
\newblock {\em arXiv e-prints}, page arXiv:2002.04802, February 2020.

\bibitem{KL}
C.~Kipnis and C.~Landim.
\newblock {\em Scaling limits of interacting particle systems}, volume 320 of
  {\em Grundlehren der Mathematischen Wissenschaften [Fundamental Principles of
  Mathematical Sciences]}.
\newblock Springer-Verlag, Berlin, 1999.

\bibitem{KA}
W.~Kob and H.~C. Andersen.
\newblock Relaxation dynamics in a lattice gas: a test of the mode-coupling
  theory of the ideal glass transition.
\newblock {\em Physical Review E}, 47(5):3281, 1993.

\bibitem{Kor}
M.~K. Korten.
\newblock Nonnegative solutions of {$u_t=\Delta (u-1)_+$}: regularity and
  uniqueness for the {C}auchy problem.
\newblock {\em Nonlinear Anal.}, 27(5):589--603, 1996.

\bibitem{KorMoo}
M.~K. Korten and C.~N. Moore.
\newblock Regularity for solutions of the two-phase {S}tefan problem.
\newblock {\em Commun. Pure Appl. Anal.}, 7(3):591--600, 2008.

\bibitem{LV}
C.~Landim and G.~Valle.
\newblock A microscopic model for {S}tefan's melting and freezing problem.
\newblock {\em Ann. Probab.}, 34(2):779--803, 2006.

\bibitem{liggett-intro}
T.~M. Liggett.
\newblock Interacting particle systems--an introduction.
\newblock In {\em School and {C}onference on {P}robability {T}heory}, ICTP
  Lect. Notes, XVII, pages 1--56. Abdus Salam Int. Cent. Theoret. Phys.,
  Trieste, 2004.

\bibitem{Lubeck}
S.~L\"ubeck.
\newblock Scaling behavior of the absorbing phase transition in a conserved
  lattice gas around the upper critical dimension.
\newblock {\em Phys. Rev. E}, 64:016123, Jun 2001.

\bibitem{Meirmanov}
A.~M. Me\u{\i}rmanov.
\newblock {\em The Stefan Problem}.
\newblock Berlin, Boston: De Gruyter., 2011.
\newblock Translated by Niezgodka, M. and Crowley, A.

\bibitem{Meir2}
A.~M. Me\u{\i}rmanov and I.~A. Kaliev.
\newblock One-dimensional {S}tefan problem with an arbitrary initial enthalpy.
  {P}eriodical solutions.
\newblock In {\em Free boundary problems: applications and theory, {V}ol. {III}
  ({M}aubuisson, 1984)}, volume 120 of {\em Res. Notes in Math.}, pages 40--49.
  Pitman, Boston, MA, 1985.

\bibitem{ritortsollich}
F.~Ritort and P.~Sollich.
\newblock Glassy dynamics of kinetically constrained models.
\newblock {\em Advances in Physics}, 52(4):219--342, 2003.

\bibitem{RPV}
M.~Rossi, R.~Pastor-Satorras, and A.~Vespignani.
\newblock Universality class of absorbing phase transitions with a conserved
  field.
\newblock {\em Phys. Rev. Lett.}, 85:1803--1806, Aug 2000.

\bibitem{assaf}
Assaf Shapira.
\newblock Hydrodynamic limit of the kob-andersen model, 2020.

\bibitem{Stef}
J.~Stefan.
\newblock \"{U}ber die {T}heorie der {E}isbildung, insbesondere \"uber die
  {E}isbildung im {P}olarmeere.
\newblock {\em Annalen der Physik}, 278:269 -- 286, 03 2006.

\bibitem{Tsu}
K.~Tsunoda.
\newblock Derivation of {S}tefan problem from a one-dimensional exclusion
  process with speed change.
\newblock {\em Markov Process. Related Fields}, 21(2):263--273, 2015.

\bibitem{Uchiyama}
K.~Uchiyama.
\newblock Scaling limits of interacting diffusions with arbitrary initial
  distributions.
\newblock {\em Probability Theory and Related Fields}, 99(1):97--110, Mar 1994.

\bibitem{varadhan91}
S.~R.~S. Varadhan.
\newblock Scaling limits for interacting diffusions.
\newblock {\em Comm. Math. Phys.}, 135(2):313--353, 1991.

\bibitem{varadhan}
S.~R.~S. Varadhan.
\newblock {\em Probability theory}, volume~7 of {\em Courant Lecture Notes in
  Mathematics}.
\newblock New York University, Courant Institute of Mathematical Sciences, New
  York; American Mathematical Society, Providence, RI, 2001.

\end{thebibliography}
\bibliographystyle{plain}

\end{document}